\documentclass[11pt]{aart}

\usepackage[letterpaper, hmargin=1in, top=1in, bottom=1in, footskip=0.6in]{geometry}

\usepackage{titlesec}
\titleformat{\section}[block]{\filcenter\normalfont\bfseries\large}{\thesection.}{.5em}{}\titlespacing*{\section}{0pt}{2\baselineskip}{1\baselineskip}
\titleformat{\subsection}[runin]{\normalfont\bfseries}{\thesubsection.}{.4em}{}[.]\titlespacing{\subsection}{0pt}{2ex plus .1ex minus .2ex}{.8em}
\titleformat{\subsubsection}[runin]{\normalfont\itshape}{\thesubsubsection.}{.3em}{}[.]\titlespacing{\subsubsection}{0pt}{1ex plus .1ex minus .2ex}{.5em}
\titleformat{\paragraph}[runin]{\normalfont\itshape}{\theparagraph.}{.3em}{}[.]\titlespacing{\paragraph}{0pt}{1ex plus .1ex minus .2ex}{.5em}

\usepackage[T1]{fontenc}
\usepackage[utf8]{inputenc}

\usepackage{mlmodern}

\usepackage[labelfont=bf,font=small,labelsep=period]{caption}
\usepackage{microtype}

\let\originalleft\left
\let\originalright\right
\renewcommand{\left}{\mathopen{}\mathclose\bgroup\originalleft}
\renewcommand{\right}{\aftergroup\egroup\originalright}

\usepackage{amsmath}
\usepackage{amssymb}
\usepackage{amsfonts}
\usepackage{latexsym}
\usepackage{amsthm}
\usepackage{amsxtra}
\usepackage{amscd}
\usepackage{bbm}
\usepackage{mathrsfs}
\usepackage{bm}
\usepackage{mathtools}

\usepackage{graphicx, color}

\definecolor{darkred}{rgb}{0.9,0,0.3}
\definecolor{darkblue}{rgb}{0,0.3,0.9}

\definecolor{vdarkred}{rgb}{0.6,0,0.2}
\definecolor{vdarkblue}{rgb}{0,0.2,0.6}
\usepackage[pdftex, colorlinks, linkcolor=vdarkblue,citecolor=vdarkred,pagebackref]{hyperref}
\RequirePackage[capitalize,nameinlink,noabbrev]{cleveref}

\usepackage{booktabs}
\usepackage[nottoc,notlof,notlot]{tocbibind}
\usepackage{cite} %Enabling `[1-4]` - style citing
\numberwithin{equation}{section}
\numberwithin{figure}{section}

\usepackage{enumitem}
\makeatletter
\let\it@m\item% Store \item inside \it@m
\RenewDocumentCommand{\item}{ o }{%
  \IfValueTF{#1}
    {\it@m[#1]\phantomsection\protected@edef\@currentlabel{#1}}
    {\it@m}
}
\makeatother

\theoremstyle{plain} %plain, definition, remark
\newtheorem{theorem}{Theorem}[section]
\newtheorem*{theorem*}{Theorem}
\newtheorem{lemma}[theorem]{Lemma}
\newtheorem*{lemma*}{Lemma}

\newtheorem*{corollary*}{Corollary}
\newtheorem{proposition}[theorem]{Proposition}
\newtheorem*{proposition*}{Proposition}

\newtheorem*{conjecture*}{Conjecture}

\theoremstyle{definition} %plain, definition, remark
\newtheorem{definition}[theorem]{Definition}
\newtheorem*{definition*}{Definition}

\newtheorem*{example*}{Example}
\newtheorem{remark}[theorem]{Remark}
\newtheorem*{remark*}{Remark}

\newtheorem*{assumption*}{Assumption}

\renewcommand{\b}[1]{\boldsymbol{\mathrm{#1}}} %bold
\renewcommand{\r}{\mathrm} %\newcommand{\rr}{\mathrm} %upright
\newcommand{\bb}{\mathbb} %blackboard bold
\renewcommand{\cal}{\mathcal}

\newcommand{\ol}[1]{\overline{#1} \!\,} %overline
\newcommand{\wh}{\widehat}

\newcommand{\op}{\operatorname}

\renewcommand{\P}{\mathbb{P}}
\newcommand{\E}{\mathbb{E}}
\newcommand{\R}{\mathbb{R}}
\newcommand{\C}{\mathbb{C}}
\newcommand{\N}{\mathbb{N}}

\newcommand{\HH}{{\mathbb H}}

\newcommand{\ee}{\r e}
\newcommand{\ii}{\r i}
\newcommand{\dd}{\r d}

\newcommand{\col}{\vcentcolon}%{\mathrel{\vcenter{\baselineskip0.75ex \lineskiplimit0pt \hbox{.}\hbox{.}}}}
\newcommand*{\deq}{\mathrel{\vcenter{\baselineskip0.65ex \lineskiplimit0pt \hbox{.}\hbox{.}}}=}

\newcommand{\eqdist}{\overset{\r d}{=}}
\renewcommand{\leq}{\leqslant}
\renewcommand{\geq}{\geqslant}
\renewcommand{\epsilon}{\varepsilon}
\newcommand{\ind}[1]{\b 1_{#1}}

\newcommand{\pb}[1]{\bigl(#1\bigr)}
\newcommand{\pB}[1]{\Bigl(#1\Bigr)}
\newcommand{\pbb}[1]{\biggl(#1\biggr)}
\newcommand{\pBB}[1]{\Biggl(#1\Biggr)}

\newcommand{\qb}[1]{\bigl[#1\bigr]}
\newcommand{\qB}[1]{\Bigl[#1\Bigr]}
\newcommand{\qbb}[1]{\biggl[#1\biggr]}
\newcommand{\qBB}[1]{\Biggl[#1\Biggr]}

\newcommand{\h}[1]{\{#1\}}

\newcommand{\hBB}[1]{\Biggl\{#1\Biggr\}}

\newcommand{\abs}[1]{\lvert #1 \rvert}
\newcommand{\absb}[1]{\bigl\lvert #1 \bigr\rvert}

\newcommand{\absbb}[1]{\biggl\lvert #1 \biggr\rvert}

\newcommand{\norm}[1]{\lVert #1 \rVert}

\newcommand{\ang}[1]{\langle #1 \rangle}

\newcommand{\scalar}[2]{\langle#1 \mspace{2mu}, #2\rangle}
\newcommand{\scalarb}[2]{\bigl\langle#1 \mspace{2mu}, #2\bigr\rangle}

\DeclareMathOperator{\tr}{Tr}

\DeclareMathOperator{\re}{Re}
\DeclareMathOperator{\im}{Im}

\DeclareMathOperator{\diam}{diam}
\DeclareMathOperator{\spec}{spec}

\newcommand{\vertiii}[1]{\norm{#1}}

\newcommand{\Smat}{S}

\newcommand{\proba}[1]{\P\left( #1 \right)}

\newcommand{\Mart}{\mathcal{M}}

\newcommand{\Cf}{\mathfrak{C}}

\newcommand{\tG}{\widetilde{G}}
\newcommand{\Vab}{\mathcal{V}_{ab}}
\newcommand{\Ecal}{\mathcal{E}}

\newcommand{\Lcal}{\mathcal{L}}

\newcommand{\Rcal}{\mathcal{R}}

\allowdisplaybreaks
  
\title{Bernoulli flow for Erd\H{o}s-Rényi graphs}
 \author{Joscha Henheik \and Antti Knowles}
\begin{document}
	\maketitle

\begin{abstract}
We study the eigenvalues and eigenvectors of the adjacency matrix $A$ of the Erd\H{o}s-Rényi graph $\bb G(N,p)$ in the regime $Np \gg (\log N)^2$. We establish optimal isotropic delocalization for the bulk eigenvectors $\bm u$, meaning that $\scalar{\bm v}{\bm u}^2 \leq \frac{C \log N}{N}$ with very high probability for any deterministic normalized $\bm v$. In addition, we prove local spectral universality in the bulk by showing that the local spectral statistics coincide with those of the GOE. The main tool of our proof is a local law for the resolvent of $A$, down to optimal spectral scales and with optimal error bounds. Its proof relies on a new approach to local laws that we call the Bernoulli flow. It is a characteristic flow method in which the usual Brownian process is replaced by a Bernoulli process, where each edge of the graph forms an independent Markov process that jumps at unit rate from closed to open. The spectral parameter flows according to a suitably constructed matrix-valued Bernoulli characteristic flow.
\end{abstract}

\setcounter{tocdepth}{1}
\tableofcontents

\section{Introduction}

A disordered quantum system, described by a large Hermitian random matrix $H$, is typically either in a \emph{delocalized}, or \emph{conducting}, phase or a \emph{localized}, or \emph{insulating}, phase. In the former, the eigenvectors are delocalized and nearby eigenvalues exhibit level repulsion, corresponding to random matrix local spectral statistics. In the latter, the eigenvectors are localized and nearby eigenvalues are approximately independent, corresponding to Poisson local spectral statistics.

There are several ways of quantifying the extent of delocalization or localization of an eigenvector of $H \in \R^{N \times N}$. They rely on the choice of an orthonormal basis $\bm e_1, \dots, \bm e_N$ of $\R^N$, such as the standard basis. A normalized eigenvector $\bm w$ gives rise to a probability measure $x \mapsto \scalar{\bm e_x}{\bm u}^2$ on $[N] \deq \{1, \dots, N\}$. Informally, $\bm u$ is \emph{delocalized} if its mass is approximately uniformly distributed throughout $[N]$, and \emph{localized} if its mass is essentially concentrated on a small number of sites. A simple and commonly used measure of the extent of localization or delocalization is the squared $\ell^\infty$ norm
\begin{equation} \label{def_q}
q(\bm u) \deq \max_{x \in [N]} \scalar{\bm e_x}{\bm u}^2\,.
\end{equation}
Clearly, $\frac1N \leq q(\bm u) \leq 1$, whereby $q(\bm u) = \frac1N$ corresponds to complete delocalization and $q(\bm u) = 1$ to complete localization.

The archetypal random matrix in the delocalized phase is the Gaussian Orthogonal Ensemble (GOE)\footnote{The GOE is by definition an $N \times N$ real symmetric matrix whose upper-triangular entries are independent mean-zero Gaussian random variables, with variances $2/N$ on the diagonal and $1/N$ elsewhere.}. By invariance of GOE under conjugation with orthogonal matrices, any eigenvector $\bm u$ of GOE is uniformly distributed on the unit sphere.
Hence, for any fixed $C > 0$ we have
\begin{equation*}
\P \pbb{\langle \bm v, \bm u \rangle^2 \geq \frac{C \log N}{N}} \sim \frac{N^{-C/2}}{\sqrt{C \log N}}
\end{equation*}
as $N \to \infty$. By a union bound, we conclude that
\begin{equation} \label{optimal_deloc}
q(\bm u) \leq \frac{C \log N}{N}
\end{equation}
with probability at least $1 - N^{-D}$ provided that $D = \frac{C}{2} - 1$. Since \eqref{optimal_deloc} is (up to the precise value of $C$) the strongest bound on $q(\bm u)$ that can hold for any random matrix $H$, it is referred to as \emph{optimal delocalization}. When the entries of $H$ are not Gaussian, the rotational invariance of $H$ is lost and delocalization becomes nontrivial. In fact, as explained in Section \ref{sec:background} below, for sufficiently sparse entry distributions, some eigenvectors are localized.

In this paper, we investigate the delocalized phase of sparse random matrices. Sparse random matrices naturally arise as adjacency matrices of random graphs, which can be interpreted as Hamiltonians of free quantum particles hopping along the edges of the graph. We consider the Erd\H{o}s-Rényi graph $\bb G \equiv \bb G(N, p)$ on $N$ vertices, where each edge of the complete graph is kept independently with probability $p$. For $p \asymp 1$, the graph $\bb G$ is dense and corresponds to a Wigner matrix\footnote{A Wigner matrix is by definition an $N \times N$ real symmetric (or complex Hermitian) matrix whose upper-triangular entries are independent mean-zero random variables with the same law, which does not depend on $N$.}, where the particle can typically hop from any site to any other site. For $p \ll 1$, the graph $\bb G$ is sparse and the hopping is constrained to a small number of edges.

Denoting by $A$ the adjacency matrix of $\bb G(N,p)$, we set $H = \theta A$, where $\theta > 0$ is a scaling factor. Under the assumption $Np \gg (\log N)^2$, we establish both hallmarks of the delocalized phase in the bulk spectrum:
\begin{description}
\item[1.\ Optimal isotropic delocalization:]
Optimal delocalization \eqref{optimal_deloc} in the bulk with respect to any deterministic basis.
\item[2.\ Bulk spectral universality:]
The local spectral statistics in the bulk coincide with those of the GOE, i.e.\ they are asymptotically given by the $\op{Sine}_1$ process.
\end{description}

Our techniques also apply to non-Hermitian sparse random matrices, corresponding to adjacency matrices of directed Erd\H{o}s-Rényi graphs, for which we prove optimal isotropic delocalization in the bulk. In a companion paper \cite{HKGNM}, we extend our results to the $\mathbb{G}(N,M)$ model, where an (undirected) graph is chosen uniformly at random among graphs with $M \gg N (\log N)^2$ edges.

We now give a brief overview of the our proof, referring to Section \ref{subsec:overview} below for a more detailed account. Our main technical achievement is a local law for the matrix $H$ down to optimal spectral scales, with optimal error bounds. Optimal isotropic delocalization is an immediate consequence. Moreover, we obtain local spectral universality by combining our local law with a very recent result by Bourgade and Huang \cite{BourHua2026}, which uniquely characterizes the $\mathrm{Sine}_1$ point process via loop equations on the spectral scale of the mean eigenvalue spacing.

Local laws form a cornerstone of random matrix theory, and several methods have been developed to establish them, including Schur complement approaches, Gaussian integration by parts and cumulant expansions, as well as Brownian characteristic flows. However, when using them in the sparse regime $p \ll 1$, it seems difficult to obtain optimal bounds.

Instead, we develop a new approach to local laws: the \emph{Bernoulli flow}. The basic idea is to replace, in Pastur's method of characteristics \cite{pastur1972spectrum}, the Brownian process with a \emph{Bernoulli process}. In it, each edge of the graph forms an independent Markov process that jumps at unit rate from closed to open\footnote{The idea of constructing the Erd\H{o}s-Rényi graph dynamically, in discrete time, goes all the way back to the work of Erd\H{o}s and Rényi \cite{erd6s1960evolution}; see also \cite{roberts2018exceptional}. The eigenvalue process of a related symmetric random walk, which leaves $\bb G(N,p)$ for $p = 1/2$ invariant, was studied in the physics literature in \cite{joyner2015spectral}.}. The spectral parameter flows according to a suitably constructed matrix-valued \emph{Bernoulli characteristic flow}.

The key advantage of replacing the Brownian flow with the Bernoulli flow is that the flow does not generate a Gaussian component, instead flowing directly to the desired target distribution. Hence, no comparison step is required, thus avoiding the key hurdle in  establishing local laws for sparse random matrices using dynamical methods. In that sense, unlike many previously used methods such as Brownian characteristic flows and cumulant expansions, our approach is not perturbative around a Gaussian law. Moreover, although in this paper we illustrate the Bernoulli flow in the simple homogeneous setting of $\bb G(N,p)$, it is very flexible and can be applied to rather general sparse random matrix models, with variance profiles and arbitrary expectations.

Replacing the Brownian flow with the Bernoulli flow results in fundamental changes to the argument. To begin with, we replace the Itô formula for Brownian motion with the Dynkin martingale associated with the Bernoulli process. As a consequence, the structure and algebra of our proof differs from Brownian characteristic flow arguments. Another important new component of our proof is an analysis of the jump sizes of various observables under the Bernoulli process, required to control the quadratic variation of the Dynkin martingale.

To prove local spectral universality, our verification of the loop equations from \cite{BourHua2026} uses a standard cumulant expansion that is controlled using our local law. In order to reach the scale $Np \gg (\log N)^2$ the estimates have to be performed with some care, which requires a new interpolation argument explained in Section \ref{sec:GOE_overview} below.

The lower bound $Np \gg (\log N)^2$ is important for both of our main results, but for different reasons. We believe that it is a natural barrier for the methods that we use, which are based on very high probability estimates. See Remark \ref{rem:p_range} below for more details. In fact, although delocalization is known to hold down to $Np \gg \log N$, it is not clear that optimal delocalization must hold down to the same threshold.

\subsection*{Notations}
Every quantity that is not explicitly called \emph{fixed} or a \emph{constant} is a sequence depending on $N$. We use the customary notation $O(\cdot)$ in the limit $N \to \infty$. For nonnegative $X,Y$, if $X = O(Y)$ then we also write $X \lesssim Y$. Moreover, we write $X \asymp Y$ to mean $X \lesssim Y$ and $Y \lesssim X$. We say that an event $\Omega$ holds with \emph{very high probability} if for any $D > 0$ there exists $C > 0$ such that $\P(\Omega) \geq 1 - N^{-D}$ for all $N \geq C$.

We use the shorthand $\sum_x \equiv \sum_{x \in [N]}$. We denote by $\bm e_x$ the standard basis vector in $\C^N$ with entries $(\bm e_x)_y = \delta_{xy}$. We denote by $\bm 1$ the identity matrix in $\C^{N \times N}$. For two vectors $\bm v = (v_x)_{x \in [N]}$ and $\bm w = (w_x)_{x \in [N]}$ in $\C^N$ we use the standard scalar product $\scalar{\bm v}{\bm w} \deq \sum_x \ol v_x w_x$. For a matrix $M \in \C^{N \times N}$, we abbreviate $\im M \deq \frac{1}{2 \ii}(M - M^*)$, $\ang{M} \deq \frac{1}{N} \tr (M)$, and $M_{\bm v \bm w} \deq \scalar{\bm v}{M \bm w}$ for $\bm v, \bm w \in \C^N$.

\subsection{Results}
Let $A \equiv A(p)$ be the adjacency matrix of the Erd\H{o}s-Rényi graph on $N$ vertices with edge probability $p \equiv p(N) \in [0,1/2]$. That is, $A =A^*$, $A_{xx}=0$ for all $x \in [N]$, and  $( A_{xy} \col x < y)$ are independent $\op{Bernoulli}(p)$ random variables. We define the normalized adjacency matrix
\begin{equation}
\label{eq:ERmodel}
H \deq \theta A \,, \qquad \theta \deq \frac{1}{\sqrt{N p (1 - p)}}\,.
\end{equation}
 
 Our first main result is optimal delocalization for the bulk eigenvectors of $H$. To state it, we define the anisotropic norm of $\bm v \in \R^N$ through 
\begin{equation} \label{eq:normdef}
\vertiii{\bm v}_{\bm e} \deq \Vert (\mathbf{1} - \Pi_{\bm e}) \bm v \Vert + (Np)^{-1/2} \, \Vert \Pi_{\bm e} \bm v \Vert\,,
\end{equation}
where $\Pi_{\bm e} \deq \bm e \bm e^*$ denotes the orthogonal projection onto $\bm e \deq N^{-1/2}(1, \dots , 1)^* \in \R^N$. (Note that $\norm{\bm v}_{\bm e} \leq 2 \norm{\bm v}$ for $Np \geq 1$.)

\begin{theorem}[Optimal delocalization] \label{thm:deloc}
 For any fixed $\kappa > 0$ and $D > 0$ there exists a constant $C \equiv C(D, \kappa) > 0$ such that the following holds. Suppose that $C (\log N)^{2} \leq Np \leq N^{1-\kappa}$. Then, for any deterministic $\bm v \in \R^N$ and any normalized eigenvector $\bm u$ of $H$ with eigenvalue $\lambda$ satisfying $\abs{\lambda} \leq 2 - \kappa$, we have
	\begin{equation} \label{deloc_main}
\langle \bm v, \bm u \rangle^2 \leq \frac{C \log N}{N} \vertiii{\bm v}_{\bm e}^2
	\end{equation}
	with probability at least $1 - N^{-D}$. 
\end{theorem}

Our second main result is bulk universality, i.e.\ GOE local spectral statistics, for $H$. To state it, we choose an energy $E \in (-2,2)$ and define the rescaled microscopic eigenvalue process of $H$ around $E$ as\footnote{Following \cite{BourHua2026}, we rescale the eigenvalue process so that the local mean eigenvalue spacing is $\pi$.}
\begin{equation} \label{muE}
\mu^E \deq \sum_{\lambda \in \spec(H)} \delta_{N \pi \rho_E (\lambda - E)}\,,
\end{equation}
where
\begin{equation} \label{eq:rhodef}
\rho_E  \deq \frac{\sqrt{4-E^2}}{2\pi } 
\end{equation}
is the density of the semicircle law. In the sum on the right-hand side of \eqref{muE}, eigenvalues are counted with multiplicity.

\begin{theorem}[GOE local spectral statistics in the bulk] \label{thm:GOE}
Fix $E \in (-2,2)$ and $\kappa > 0$. Suppose that $(\log N)^{2+\kappa} \leq Np \leq N^{1-\kappa}$. Then, as $N \to \infty$, $\mu^E$ converges in distribution, with respect to the vague topology, to the $\mathrm{Sine}_1$ point process.
\end{theorem}

The $\mathrm{Sine}_1$ process has been extensively studied in the literature \cite{Meh, forrester2010log}; it is a universal Pfaffian point process that arises as the local limit of the GOE eigenvalue process, where \eqref{muE} is defined in terms of a GOE matrix. Explicitly, Theorem \ref{thm:GOE} states that for any continuous and compactly supported function $f$, the random variable $\int f\, \dd \mu^E$ converges in distribution to a corresponding random variable where the point process $\mu^E$ is replaced with the $\mathrm{Sine}_1$ process.

We conclude this subsection with some remarks on Theorems \ref{thm:deloc} and \ref{thm:GOE}.

\begin{remark}[Optimality of the delocalization bounds] \label{rem:deloc_optimal}
As explained before \eqref{optimal_deloc}, our delocalization bound in Theorem~\ref{thm:deloc} is optimal up to the precise dependence of $C$ on $D$. In the direction of $\bm e$, we obtain the improved delocalization bound
\begin{equation*}
\langle \bm e, \bm u \rangle^2 \leq \frac{1}{Np} \frac{C \log N}{N}\,.
\end{equation*}
Such an improvement by a factor $\frac{1}{Np}$ in the $\bm e$-direction was first observed in \cite[Corollary 1.5]{he2026extremal} in the regime $Np \geq N^c$. It stems from the near-alignement of the top eigenvector $\bm u_1$ with $\bm e$, constraining all other eigenvectors to lie in the (deterministic) orthogonal complement of $\bm e$. In \cite[Theorem 6.2]{EKYY1}, it was shown\footnote{This result is proved for $Np \gg (\log N)^6$, but we expect it and its proof to carry over to the sparser regime of the current paper.} that the top eigenvector $\bm u_1$ of $H$ satisfies $\scalar{\bm e}{\bm u_1}^2 = 1 - \frac{1 + o(1)}{Np}$, which implies that the squared projection of any other eigenvector in the direction of $\bm e$ is suppressed by the factor $\frac{1}{Np}$.  In our estimate \eqref{deloc_main}, this phenomenon is captured by the non-isotropic norm \eqref{eq:normdef}. 
\end{remark}

\begin{remark}[Directed Erd\H{o}s-Rényi graphs] \label{rem:digraphs}
Our method used to prove Theorem \ref{thm:deloc} also applies to directed Erd\H{o}s-Rényi graphs.  Define
	\begin{equation*}
\widehat{H} \deq \theta \widehat{A}\,,
	\end{equation*}
where $\widehat{A}$ has i.i.d.~off-diagonal matrix elements. That is, $(\widehat{A}_{xy} : x \neq y)$ is a family of i.i.d.\ $\op{Bernoulli}(p)$ random variables, and $\widehat{A}_{xx} = 0$ for all $x$. It is well known that the asymptotic spectrum of $\wh A$ is the unit disk in $\C$ (the circular law).  Under the conditions of Theorem \ref{thm:deloc}, for any normalized left-eigenvector $\bm l$ and right-eigenvector $\bm r$ of $\wh H$ with eigenvalue $\lambda$ satisfying $\abs{\lambda} \leq 1 - \kappa$, we obtain
\begin{equation} \label{eq:directed_deloc}
\abs{\langle \bm v, \bm l \rangle}^2 + \abs{\langle \bm v, \bm r \rangle}^2 \leq \frac{C \log N}{N} \vertiii{\bm v}_{\bm e}^2
\end{equation}
	with probability at least $1 - N^{-D}$.  We refer to Remark \ref{rem:digraphs_proof} below for a proof.
\end{remark}

\begin{remark}[Assumptions on $p$] \label{rem:p_range}
The upper bound $Np \leq N^{1-\kappa}$ is made for convenience, and it can be relaxed without fundamental changes to our argument. We impose this bound to avoid extraneous complications, bearing in mind that the main focus of this paper is the sparse regime. While the optimal delocalization results from Theorem \ref{thm:deloc} and Remark \ref{rem:digraphs} are new as soon as $Np \ll 1$, the bulk universality result from Theorem \ref{thm:GOE} is new only in the very sparse regime $(\log N)^{2 + \kappa} \leq Np \leq N^{o(1)}$, since the case $Np \geq N^{o(1)}$ was previously treated in \cite{huang2015bulk}.

The fundamental assumptions on $p$ in Theorems \ref{thm:deloc} and \ref{thm:GOE} are the lower bounds. The lower bound on $p$ in Theorem \ref{thm:deloc} is an immediate consequence of the same lower bound in the local law, Theorem \ref{thm:lolaw}, of which Theorem \ref{thm:deloc} is an immediate consequence. Interestingly, although the two lower bounds, $C (\log N)^2$ and $(\log N)^{2 + \kappa}$ respectively, are almost identical, and although the proof of Theorem \ref{thm:GOE} relies on Theorem \ref{thm:lolaw}, the origins of these restrictions are very different.

The assumption $Np \geq C (\log N)^2$ in Theorems \ref{thm:deloc} and \ref{thm:lolaw} is required in many crucial steps of our proof, the most central of which is in the estimate of the Dynkin martingale in terms of its quadratic variation, where this condition is required to ensure that the jumps of the martingale are sufficiently small. Unlike in the classical Burkholder-Davis-Gundy inequality for Brownian motion, which always yields subgaussian tails, for Markov jump processes the tails can be subgaussian or subexponential, depending on the jump sizes and how far one goes into the tail. We refer to Remark \ref{rem:lowerbound_p} below for more details.

The assumption $(\log N)^{2 + \kappa}$ for Theorem \ref{thm:GOE} is unrelated to the analogous assumption in Theorem~\ref{thm:lolaw}. In fact, even if Theorem \ref{thm:lolaw} held under the weaker assumption $pN \geq C \log N$, it would not allow us to improve the lower bound on $p$ in Theorem \ref{thm:GOE}. Essentially, the latter arises from the third-order remainder term of a cumulant expansion used to verify the GOE loop equations, which balances a term $\theta \asymp \frac{1}{\sqrt{Np}}$ with a resolvent entry $G_{ab}$. Since the whole argument takes place at the microscopic spectral scale $\eta \asymp \frac{1}{N}$ and our local law controls the entries of $G$ with very high probability by $O(1)$ at the optimal scale $\eta \asymp \frac{\log N}{N}$, we have to resort to a monotonicity argument in $G$ to estimate $G_{ab}$ by $\log N$ at the microscopic scale, which leads to the error bound $\frac{\log N}{\sqrt{Np}}$. As explained in Section \ref{sec:GOE_overview} below, the actual argument is considerably more subtle and even obtaining the lower bound $(\log N)^{2 + \kappa}$ requires some care.
\end{remark}

\begin{remark}[The spectral edge] \label{rmk:edge}
We expect our optimal delocalization result in Theorem \ref{thm:deloc} to extend to the edge as well, i.e.~\eqref{deloc_main} to hold for all but the top eigenvector $\bm u_1$; cf.~Remark \ref{rem:deloc_optimal}. Moreover, we expect edge universality to hold as well, up to a random shift of the edge locations $\pm2$; see \cite{EKYY2, lee2018local, HuangLandonYau, he2021fluctuations, JLeeedge, huang2026edge} for results in the bounded-diameter regime $Np \ge N^\epsilon$ (see \eqref{diam_G} below). The missing ingredient for proving such results is a nontrivial extension of our optimal local law in Theorem \ref{thm:lolaw} to the edge, in particular also controlling concentration of the Green function around a suitable explicit random approximation, taking into account the randomly shifted edge locations. See Section \ref{subsec:edge} below for more details. 
\end{remark}

\subsection{Background} \label{sec:background}

The literature on delocalization and bulk universality in random matrix theory is vast, and we shall not attempt a complete account of past results. Instead, we review some results that are most relevant to our work on sparse random matrices. We focus on matrices with independent or approximately independent entries, such as Wigner matrices, and more specifically on sparse random matrices. Since the independence structure is tied to the standard basis of $\R^N$, there is an important distinction in establishing delocalization with respect to the standard basis, referred to as \emph{entrywise delocalization}, and delocalization with respect to any orthonormal basis, referred to as \emph{isotropic delocalization} (recall that \eqref{def_q} depends on the choice of basis).

For \emph{Wigner matrices}, i.e.\ the dense case, delocalization in its strongest form -- isotropic with optimal bounds -- has been established under rather general assumptions on the entries. Historically, \emph{entrywise delocalization} of the form $q(\bm u) \leq \frac{(\log N)^C}{N}$ for some constant $C > 0$ was first achieved in the seminal works \cite{ESY1, ESY2, ESY3}, subsequently extended to the edge and more general entry distributions in \cite{TV1, TV2, EYY1, EYY2, EYY3, Agg16}. \emph{Optimal entrywise delocalization} was first established in \cite{vu2015random, o2016eigenvectors}, with near-optimal probability bounds and Gumbel fluctuations of $q(\bm u)$ obtained in \cite{osman2026}. The first result on \emph{isotropic delocalization} was \cite{knowles2013isotropic}, subsequently generalized in \cite{BEKYY, HKR}. \emph{Optimal isotropic delocalization} was finally proved in \cite{benigni2022optimal}, which even obtained the optimal constant $C$ in \eqref{optimal_deloc}. In addition, delocalization has been  studied for (the left- and right-eigenvectors of) non-Hermitian matrices as well, where all $N^2$ entries of the matrix are independent. \emph{Entrywise delocalization} for non-Hermitian matrices was first established in \cite{rudelson2015delocalization}, subsequently generalized in \cite{campbell2025spectral, alt2021spectral, alt2021inhomogeneous}. \emph{Optimal isotropic delocalization} for non-Hermitian matrices was recently established in \cite{cipolloni2025optimal}.

For \emph{sparse matrices} like the adjacency matrix of $\bb G(N,p)$, delocalization breaks down if the matrix is sufficiently sparse. Indeed, the degrees of $\bb G$ concentrate around their expectation with high probability if and only if $Np \gg \log N$. If $Np \lesssim \log N$, the degrees can fluctuate wildly, and in particular $\bb G$ can have isolated vertices, which leads trivially to localized eigenvectors. In fact, as explained below, a more subtle localization mechanism gives rise to localized eigenvectors even in the giant component of $\bb G$. The question of whether eigenvector delocalization holds for $\bb G(N,p)$\footnote{This question was originally asked in \cite[Section 3.3]{dekel2011eigenvectors}.} has seen some attention in recent years. The first delocalization result was \cite{tran2013sparse}, subsequently improved in \cite{dumitriu2019sparse}, where entrywise delocalization of the form $q(\bm u) \leq \frac{C \log N}{Np}$ was proved for $Np \geq C \log N$. Strong entrywise delocalization of the form $q(\bm u) \leq \frac{(\log N)^C}{N}$ was proved in \cite{EKYY1} down to $Np \geq C (\log N)^6$. Strong entrywise delocalization down to the critical regime $Np \geq C \log N$ was first achieved \cite{HeKnowlesMarcozzi2018}, with a bound of the form $q(\bm u) \leq \frac{N^\epsilon}{N}$, which was subsequently generalized in \cite{tikhomirov2022local, he2026sparse}. Isotropic delocalization was established only very recently in \cite{he2026extremal}, in the regime $Np \geq N^\epsilon$ with bound $q(\bm u) \leq \frac{N^\epsilon}{N}$. The critical and subcritical regimes $Np \lesssim \log N$ were analysed in detail in a series of papers \cite{ADK20, ADK_delocalized, ADK19, ADK21, alt2024localized}, which uncovered a phase diagram describing the coexistence of delocalized and localized eigenvectors as soon as $Np \lesssim \log N$. In particular, in the regime $Np \lesssim \log N$, localized eigenvectors for the giant component appear through two separate mechanisms emerging from the lack of concentration of the vertex degrees: vectors supported on bounded tuning fork subgraphs, and exponentially decaying vectors concentrated around vertices of high degree. In \cite{ADK20, ADK_delocalized}, the optimal region of delocalization in terms of the energy and the critical density $\frac{Np}{\log N}$ was determined. In addition, the entrywise delocalization of sparse non-Hermitian random matrices, such as the adjacency matrix of the directed Erd\H{o}s-Rényi graph, has been established in \cite{he2025edge} for $Np \geq N^\epsilon$ and extended to the critical regime $Np \asymp \log N$ (and even slightly below) in \cite{alt2026critical}.

Bulk universality for Wigner matrices was established in the landmark series of papers \cite{ESY1, ESY2, ESY3, ESY4, ESY5, EPRSY, EPRSTY, ESYY, EYY1, EYY2, EYY3} as well as in \cite{TV1, TV2}, and later gradually extended to dense mean field ensembles of greater generality: for Wigner matrices with diagonal \cite{lee2016bulk} and non-diagonal deformations \cite{knowles2017anisotropic, HKR}, Wigner-type ensembles with not necessarily identically distributed but still independent entries \cite{univWigtype}, and random matrices with weakly correlated entries \cite{ajanki2019stability, slowcorr}.  For sparse matrices, bulk universality was proved for $Np \gg N^{2/3}$ in \cite{EKYY1, EKYY2} and extended to $Np \geq N^\epsilon$ in \cite{huang2015bulk}. We also mention bulk universality results for $d$-regular graphs \cite{bauerschmidt2017bulk, BourHua2026} and random band matrices \cite{bourgade2017universality, bourgade2020random, dubova2026quantum, yau2025delocalization, RBM, dubova2025delocalization2, dubova2025delocalization3}. Finally, bulk universality has been established for invariant $\beta$-ensembles (see e.g.\ \cite{bekerman2015transport, bourgade2014universality, deift1999uniform, pastur2008bulk, shcherbina2014change, valko2009continuum}).

We now outline how our work improves on the previous works summarized above.

\begin{enumerate}[label=(\roman*)]
\item
Optimal delocalization in the bulk for sparse random matrices down to $Np \gg (\log N)^2$. Previously, optimal delocalization was known in the dense regime $p \asymp 1$ \cite{benigni2022optimal, cipolloni2025optimal}.
\item \label{itm:new2}
Isotropic delocalization in the bulk down to $Np \gg (\log N)^2$. Previously, isotropic delocalization was known down to $Np \geq N^\epsilon$ \cite{he2026extremal}. In addition, we improve the bound $N^\epsilon$ from \cite{he2026extremal} to the optimal scale $C \log N$.
\item \label{itm:new3}
Bulk universality (GOE local spectral statistics) down to $Np \gg (\log N)^2$. Previously, bulk universality was known down to $Np \geq N^\epsilon$ \cite{huang2015bulk}.
\end{enumerate}

We remark that going from the regime $Np \geq N^\epsilon$ down to $Np \gg (\log N)^2$ in \ref{itm:new2} and \ref{itm:new3} presents a serious hurdle. From a conceptual point of view, this is apparent in the behaviour of the diameter\footnote{Following \cite{ChungLu_diameter}, we define the diameter of $\bb G$ to be the maximal distance between pairs of vertices in the largest connected component of $\bb G$.} of $\bb G$, which is with high probability
\begin{equation} \label{diam_G}
\diam(\bb G) = \frac{\log N}{\log (Np)} (1 + o(1)) + O(1)
\end{equation}
whenever $Np \gg 1$; see \cite[Chapter 10]{Bol01}, \cite{ChungLu_diameter}, and the references therein. From \eqref{diam_G} we deduce that if $Np \geq N^\epsilon$ then $\diam(\bb G) = O(1)$. This is therefore the bounded-diameter regime. In it, the diameter has the same qualitative behaviour as in the dense case $p \asymp 1$. In contrast, for $Np = (\log N)^\kappa$ we have $\diam(\bb G) \sim \frac{\log N}{\kappa \log \log N}$.
In particular, our results apply to graphs $\bb G$ whose diameter is half of that of the maximal diameter $\frac{\log N}{\log \log N} $ under which delocalization can hold in the bulk. From a technical point of view, the condition $Np \geq N^\epsilon$ plays a fundamental role in many previous works, in particular in high-moment cumulant expansions, which become notoriously unwieldy if this condition is relaxed.

\subsection{Overview of the proof} \label{subsec:overview}

The key tool behind our proof is an \emph{optimal local law} for $H$, formulated in Theorem \ref{thm:lolaw} below. It provides control of the resolvent $G(z) \deq (H - z)^{-1}$ down to spectral scales $\im z \gtrsim \frac{\log N}{N}$. Our method is inspired by the method of characteristics originally due to Pastur \cite{pastur1972spectrum}, more recently revived by Lee and Schnelli \cite{lee2015edge} and von Soosten and Warzel \cite{VSW2019}, followed by many other developments\footnote{The method of characteristics has been successfully applied in the study of Dyson Brownian motion \cite{huang2019rigidity, adhikari2020dyson, adhikari2023local, aggarwal2024edge} as well as various matrix models  \cite{bourgade2021extreme,
		landon2022almost, landon2024single}. It has also proved useful in recent progress on random band matrices  \cite{dubova2026quantum, yau2025delocalization, yang2025delocalizationedge, dubova2025delocalization2, dubova2025delocalization3, RBM}
		and related random block models \cite{yang2025delocalizationRBSO, truong2025localization, fan2025localization}.}.   The method of characteristics combines a \emph{Brownian process} for the matrix $H(t)$ and a \emph{characteristic flow} for the spectral parameter $z_t$, given by
\begin{equation} \label{Brownflow}
\dd H(t) = \frac{\dd B(t)}{\sqrt{N}}\,, \qquad \partial_t z_t = - \ang{M_t(z_t)}\,,
\end{equation}
where $B(t)$ is standard Brownian motion on the space of real symmetric matrices\footnote{By convention, the diffusion constant is $2$ for the diagonal entries and $1$ and for the off-diagonal entries, so that $B(t)$ is a GOE matrix for any $t > 0$.} and $M_t(z)$ is a suitably chosen deterministic matrix that is supposed to approximate $G_t(z) \deq (H(t) - z)^{-1}$. From the Itô formula we get
\begin{equation} \label{brownian_ito}
\dd \ang{G_t(z)} = \ang{G_t(z)} \ang{G_t(z)^2} \, \dd t + \frac{1}{N} \ang{G_t(z)^3} \, \dd t- \frac{1}{\sqrt{N}} \ang{G_t(z)^2 \, \dd B_t}\,.
\end{equation}
Since
\begin{equation*}
\dd \ang{G(z_t)} = - \ang{M_t(z_t)} \ang{G(z_t)^2} \, \dd t\,,
\end{equation*}
and $M_t(z_t)$ is constructed to be constant in time\footnote{This condition arises by requiring the resulting combined evolution equation to have the simple form \eqref{dG_ito}. As we shall see in \eqref{M_z_intro1} below, for the Bernoulli flow, $M_t(z_t)$ is not constant in time.}, we therefore conclude that
\begin{equation} \label{dG_ito}
\dd \ang{G_t(z_t) - M_t(z_t)} = (\ang{G_t(z_t)} - \ang{M_t(z_t)}) \ang{G_t(z_t)^2} \, \dd t + \frac{1}{N} \ang{G_t(z_t)^3} \, \dd t- \frac{1}{\sqrt{N}} \ang{G_t(z_t)^2 \, \dd B_t}\,.
\end{equation}
The key observation of the method of characteristics is that the right-hand side is typically small: the first term exploits the cancellation between $G_t$ and $M_t$ (which has to be established self-consistently along the flow using a Duhamel argument together with Grönwall's lemma; see Section \ref{subsec:edge} below for more details), the second term is typically negligible (and in fact is absent if $B(t)$ is a complex Hermitian Brownian motion), and the third term is a martingale that can be estimated using standard martingale concentration estimates. In this way, one concludes that $G_T(z_T) \approx G_0(z_0)$, where $T$ is chosen so that $H(T)$ is the random matrix one wishes to study.

The fundamental reason why this method is so useful to investigate local spectral properties of random matrices is that under the characteristic flow, $z_t$ flows towards the real axis (because of the positive imaginary part of $\ang{M_t(z_t)}$ in the bulk of the spectrum). This means that one can start the flow from $\im z_0 \asymp 1$ and reach small spectral scales $\im z_T$, thus obtaining fine spectral data starting from rough spectral data, at the expense of adding a Gaussian component. 

 This also highlights the major drawback of the method of characteristics: because $H(T) \eqdist H(0) + \sqrt{T}V$, where $V$ is an independent copy of a GOE matrix, it inevitably yields a Gaussian component in $H(T)$, no matter the distribution of $H(0)$. Unless the matrix ensemble under consideration is Gaussian divisible, this Gaussian component has to be removed by a perturbative comparison argument. Such arguments are typically performed at fixed spectral parameter $z$, by high-order cumulant expansions relying on the closeness of the moments of the distributions to be compared. The closer, in the sense of moments, the target distribution $H(T)$ is to a Gaussian random matrix, the easier the removal of the Gaussian component.
This comparison argument does not have to be done in a single step: in \cite{CES24} the authors develop a careful splitting of the comparison argument into multiple steps, coined the zigzag strategy\footnote{The general version of the zigzag strategy was introduced in~\cite{cipolloni2023eigenstate, OTOC}, the latter also coining the term. Since then, it has been extensively applied to increasingly general random matrix models \cite{cipolloni2023universality, cipolloni2024eigenvector, WigTypeETH, cuspuniv, corrETH, smallrand, RBM}.}, which balances the size of the perturbation step (the zag step) with the spectral resolution $\im z_t$ at a given instant $t$ in the characteristic flow (the zig step).

If the target random matrix $H(T)$ is too far from the space of Gaussian divisible matrices, the comparison argument cannot be performed, and hence Pastur's method of characteristics is not applicable. The adjacency matrix of a sufficiently sparse Erd\H{o}s-Rényi graph is a good example of such a matrix, since the moments of its entries are much larger than those of Wigner matrices.

The key idea of our proof is to adapt Pastur's original insight by replacing the Brownian process with a different stochastic process, so that no comparison step is required. Since our target entry distribution is Bernoulli, we choose a stochastic process whose distribution at time $t > 0$ is precisely Bernoulli. We note, however, that this idea can also be applied to different target distributions. In particular, it is also applicable to other ensembles of random matrices and random graphs \cite{HKupcoming}.  We call the resulting object the \emph{Bernoulli flow}. It consists of two ingredients.
\begin{enumerate}[label=(\roman*)]
\item
The \emph{Bernoulli process} $(X(t))_{0 \leq t \leq T}$, a Markov jump process on the space of adjacency matrices.
\item
The \emph{Bernoulli characteristic flow} $(z_t)_{0 \leq t \leq T}$, a deterministic matrix-valued flow for the spectral parameter.
\end{enumerate}
Under the Bernoulli process, each entry $X_{xy}(t)$ with $x < y$ starts from zero and jumps to one with unit rate, independently of the others. Hence, at time $t > 0$ the matrix $X(t)$ is equal in distribution to the adjacency matrix of an Erd\H{o}s-Rényi graph with edge probability $p_t \deq 1 - \ee^{-t}$. Explicitly, the generator of the Bernoulli process is
\begin{equation*}
\Lcal\phi (X) \deq \sum_{x< y}\big( \phi(X^{xy, 1}) - \phi(X) \big)\,,
\end{equation*}
where $X^{xy,1}$ denotes the matrix obtained from $X$ by setting the $xy$-entry to $1$.

 We refer to Figures \ref{fig:eigenvalues} and \ref{fig:eigenvectors} for an illustration of the Bernoulli process. From Figure \ref{fig:eigenvectors} it is apparent that, no matter the terminal time $T$, the Bernoulli process has to pass through a complex landscape of localized and partially localized eigenvectors. This region has been partially analysed in \cite{ADK20, ADK_delocalized, ADK19, ADK21, alt2024localized}. The reason we are nevertheless able to analyse the flow through this region and obtain optimal control at the terminal time is that, as the flow passes through the partially localized region, the spectral resolution $\im z_t$ is sufficiently large, which leads to a spectral averaging that washes out singularities arising from localized eigenvectors. The situation for the Brownian flow is very different, since all eigenvectors are uniformly distributed on the unit sphere for all positive times.

\begin{figure}[!ht]
\begin{center}
\includegraphics[width=0.485\textwidth]{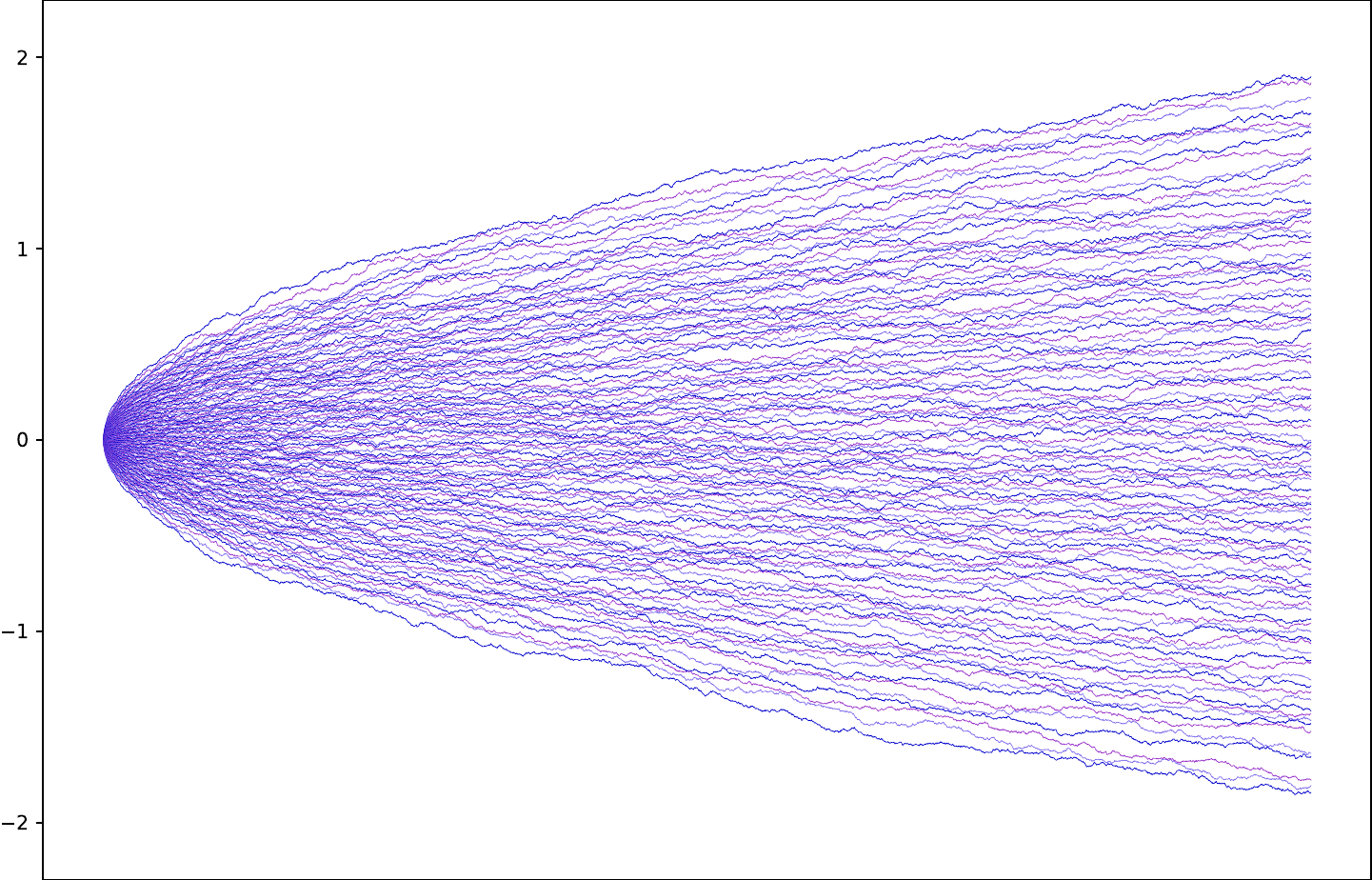}
\hspace{0.5em}
\includegraphics[width=0.485\textwidth]{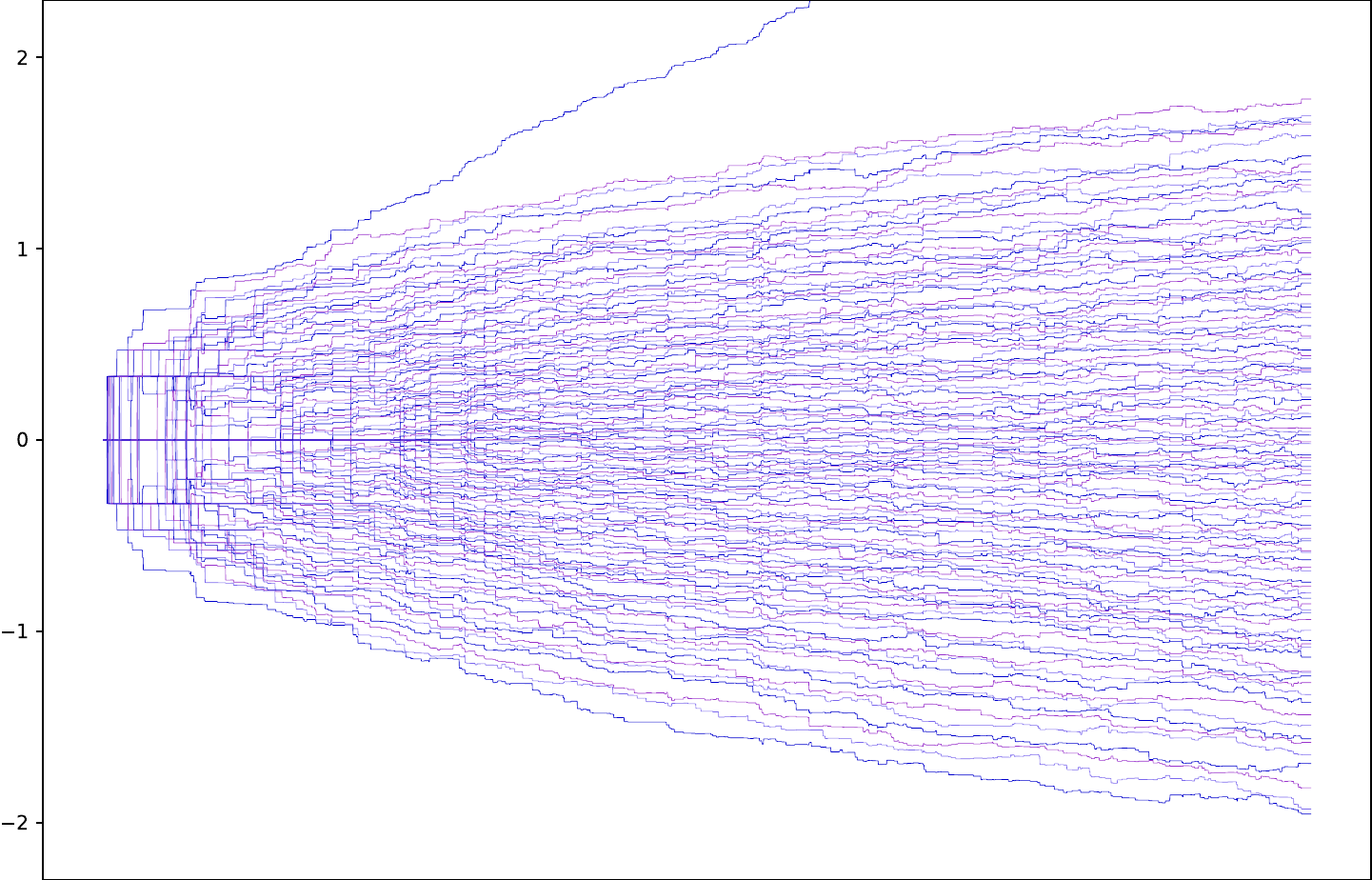}
\end{center}
\caption{An illustration of the eigenvalue process of $H(t)$ for the Brownian process (left) and the Bernoulli process (right). Here, $N = 100$ and the matrix is normalized so that at the final time $T$ the asymptotic spectrum is the interval $[-2,2]$. We choose $H(0) = 0$ in each case. For the Bernoulli flow, we choose $T$ so that $p_T = 1 - \ee^{-T} = \frac{(\log N)^2}{2N}$. The three colours used for the lines have no meaning; they serve as a visual distinction of neighbouring paths. The top eigenvalue for the Bernoulli process is the Perron-Frobenius eigenvalue arising from the nonzero expectation of the matrix.
\label{fig:eigenvalues}}
\end{figure}

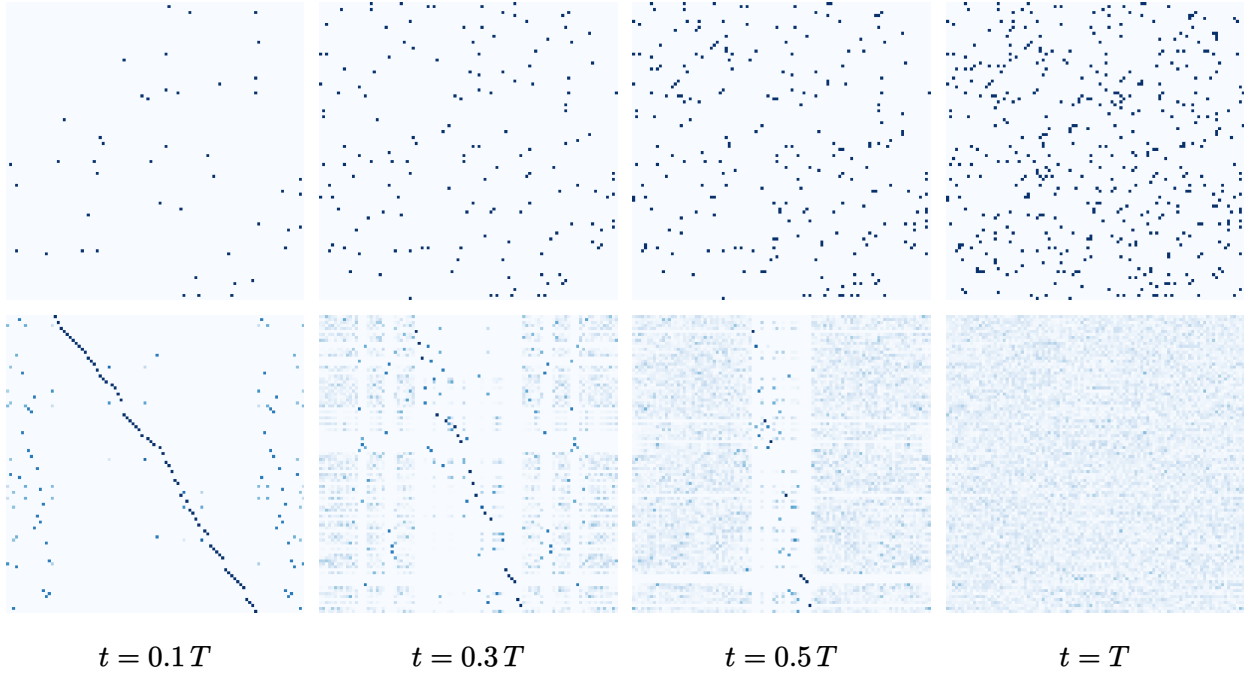
\begin{figure}[!ht]
\begin{center}
%% Creator: Inkscape 1.4.2 (ebf0e940, 2025-05-08), www.inkscape.org
%% Accompanies image file 'eigenvectors.pdf' (pdf, eps, ps)
\begingroup%
  \makeatletter%
  \providecommand\color[2][]{%
    \errmessage{(Inkscape) Color is used for the text in Inkscape, but the package 'color.sty' is not loaded}%
    \renewcommand\color[2][]{}%
  }%
  \providecommand\transparent[1]{%
    \errmessage{(Inkscape) Transparency is used (non-zero) for the text in Inkscape, but the package 'transparent.sty' is not loaded}%
    \renewcommand\transparent[1]{}%
  }%
  \providecommand\rotatebox[2]{#2}%
  \newcommand*\fsize{\dimexpr\f@size pt\relax}%
  \newcommand*\lineheight[1]{\fontsize{\fsize}{#1\fsize}\selectfont}%
  \ifx\svgwidth\undefined%
    \setlength{\unitlength}{463.43170526bp}%
    \ifx\svgscale\undefined%
      \relax%
    \else%
      \setlength{\unitlength}{\unitlength * \real{\svgscale}}%
    \fi%
  \else%
    \setlength{\unitlength}{\svgwidth}%
  \fi%
  \global\let\svgwidth\undefined%
  \global\let\svgscale\undefined%
  \makeatother%
  \begin{picture}(1,0.54096354)%
    \lineheight{1}%
    \setlength\tabcolsep{0pt}%
    \put(0,0){\includegraphics[width=\unitlength,page=1]{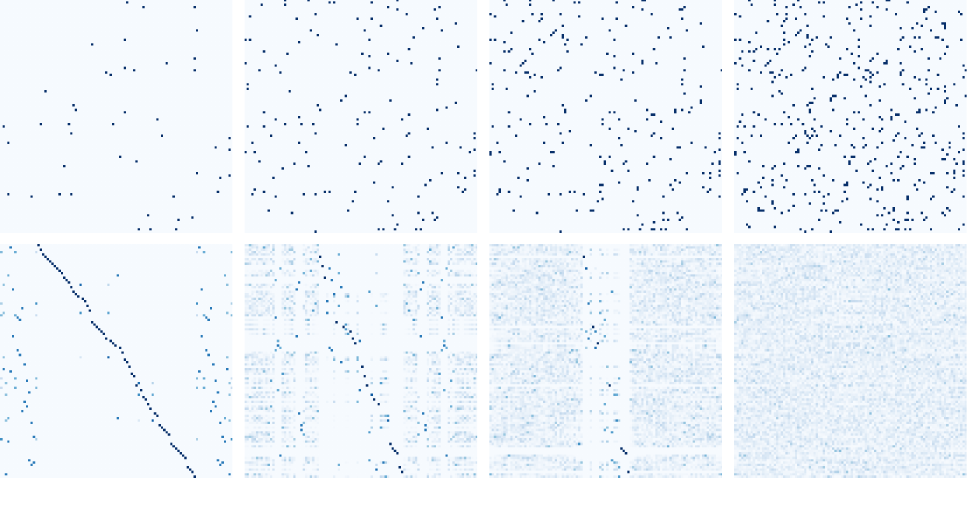}}%
    \put(0.12017097,0.00389839){\color[rgb]{0,0,0}\makebox(0,0)[t]{\lineheight{1.29999995}\smash{\begin{tabular}[t]{c}$t = 0.1 \, T$\end{tabular}}}}%
    \put(0.37334815,0.00389839){\color[rgb]{0,0,0}\makebox(0,0)[t]{\lineheight{1.29999995}\smash{\begin{tabular}[t]{c}$t = 0.3 \, T$\end{tabular}}}}%
    \put(0.62652546,0.00389839){\color[rgb]{0,0,0}\makebox(0,0)[t]{\lineheight{1.29999995}\smash{\begin{tabular}[t]{c}$t = 0.5 \, T$\end{tabular}}}}%
    \put(0.87970268,0.00389839){\color[rgb]{0,0,0}\makebox(0,0)[t]{\lineheight{1.29999995}\smash{\begin{tabular}[t]{c}$t = T$\end{tabular}}}}%
  \end{picture}%
\endgroup%

\end{center}
\caption{
A simulation of the Bernoulli process (top) and its matrix of eigenvectors (bottom), at the four indicated times, where $T = \frac{\log N}{N}$ and $N = 100$. The eigenvectors form the columns of the bottom matrix, ordered according to the values of the associated eigenvalues. Each matrix is plotted as a colour map. Note the prevalence of localized eigenvectors for small times, which are progressively replaced with delocalized eigenvectors. In contrast, for the Brownian process, at any positive time $t > 0$ the matrix of eigenvectors is uniformly distributed on the orthogonal group.
\label{fig:eigenvectors}}
\end{figure}

Replacing the Brownian flow with the Bernoulli flow leads to some changes to the setup of the argument. Our starting point is that, for any time-dependent observable $\phi_t$, the process
\begin{equation} \label{Dynkin}
\cal M_t \deq \phi_t(X(t)) - \phi_0(X(0)) - \int_0^t \dd s \, \pb{\cal L \phi_s(X(s)) + \partial_s \phi_s(X(s))}
\end{equation}
is a martingale. We apply this observation to the observables
\begin{subequations}
\label{observables_phi}
\begin{align}
\phi_t(X) &= \scalarb{\bm v}{\pb{(\theta X - z_t)^{-1} - M_t(z_t)} \bm w}
\\ \label{observables_phi2}
\phi_t(X) &= \ang{(\theta X - z_t)^{-1} - M_t(z_t)}\,.
\end{align}
\end{subequations}
The goal of the argument is to show that for $T \asymp p$ we have
\begin{equation} \label{X_T_approx_X_0}
\phi_T(X(T)) \approx \phi_0(X(0))  = 0\,,
\end{equation}
where the left-hand side characterizes the resolvent of the adjacency matrix $X(T)$ at small spectral scales $z_T$, while the right-hand side is a trivial function of the empty graph $X(0)$,  which vanishes by construction of $M_0$. We do this by showing that the martingale $\cal M_t$ is approximately zero for all times, and the integral on the right-hand side of \eqref{Dynkin} is small.

More precisely, the proof of \eqref{X_T_approx_X_0} consists of three main ingredients.
\begin{enumerate}
\item
Estimate the \emph{generator terms}
\begin{equation*}
\cal L \phi_s(X(s)) + \partial_s \phi_s(X(s))\,,
\end{equation*}
by exploiting a cancellation between the two terms, which in particular dictates the choice of the Bernoulli characteristic flow $z_t$. This cancellation is analogous to the cancellation in \eqref{dG_ito}.
\item
Estimate the \emph{jump sizes}
\begin{equation*}
\abs{\cal M_t - \cal M_{t -}}\,.
\end{equation*}
\item
Estimate the \emph{predictable quadratic variation}\footnote{We refer to Appendix \ref{app:CDC} for a few basic facts about Markov jump processes.}
\begin{equation*}
\ang{\cal M}_t = \int_0^t \dd s \, \sum_y \cal L(X(s), y) \, \abs{\phi_s(y) - \phi_s(X(s))}^2\,.
\end{equation*}
\end{enumerate}

We now sketch the proof of Step 1 and explain how the Bernoulli characteristic flow arises. For concreteness, we choose $\phi_t(X)$ as in \eqref{observables_phi2}. With the notation $G_t \equiv G_t(z_t) \deq (\theta X(t) - z_t)^{-1}$, we use a resolvent expansion to find
\begin{equation} \label{LG_outline}
\cal L \ang{G_t} = - \frac{\theta}{N} \sum_{x \neq y} (1 - X_{xy}(t)) (G^2_t)_{xy} + \frac{\theta^2}{N} \sum_{x \neq y} (1 - X_{xy}(t)) (G^2_t)_{xx} (G_t)_{yy} + \cdots\,,
\end{equation}
where $\dots$ denotes error terms that we ignore here. Here we used that $X_{xy}(t) \in \{0,1\}$. For the first term of \eqref{LG_outline}, we abbreviate $S_{xy} \deq \ind{x \neq y}$ and use the splitting $\theta S_{xy} - \theta X_{xy}(t) = (\theta S_{xy} - z_{t,xy}) - (\theta X_{xy}(t) - z_{t,xy})$. For the second term of \eqref{LG_outline}, we approximate $1 - X_{xy}(t)$ with its expectation, $1 - p_t$, using the averaging property of the sum over $x,y$. (Note, however, that since $(G^2_t)_{xx} (G_t)_{yy}$ is not independent of the family $X_{xy}(t)$, this averaging property is nontrivial and requires care.) This yields
\begin{equation}
\cal L \ang{G_t} = - \ang{G_t (\theta S - z_t) G_t} + \ang{G_t} + \frac{1 - p_t}{(1 - p)p} \ang{G_t^2} \ang{G_t} + \cdots\,,
\end{equation}
where in the second term we also neglected the contribution of the diagonal terms $x = y$. Hence, the generator terms are given by
\begin{multline*}
\cal L \phi_t(X(t)) + \partial_t \phi_t(X(t))
\\
=
- \ang{G_t (\theta S - z_t) G_t} + \ang{G_t} + \frac{1 - p_t}{(1 - p)p} \ang{G_t^2} \ang{G_t}
+ \ang{G_t [\partial_t z_t] G_t}
- \partial_t  \ang{M_t(z_t)}
+ \cdots\,.
\end{multline*}
In order to close the argument, we have to express the right-hand side in terms of some propagator term times the original observable $\phi_t(X(t)) = \ang{G_t(z_t) - M_t(z_t)}$ itself. This leads to the conditions of the \emph{Bernoulli characteristic flow}
\begin{subequations} \label{M_z_intro}
\begin{align} \label{M_z_intro1}
\partial_t M_t(z_t) &= M_t(z_t)
\\
\partial_t z_t &= \theta S - z_t - \frac{1 - p_t}{(1 - p)p} \ang{M_t(z_t)}\,,
\end{align}
\end{subequations}
under which we have
\begin{equation} \label{L_cancel_intro}
\cal L \phi_t + \partial_t \phi_t = \pbb{1 + \frac{1 - p_t}{(1 - p)p} \ang{G_t^2}} \phi_t + \cdots\,.
\end{equation}
The conditions \eqref{M_z_intro}, together with the terminal value $z_T = z \mathbf{1}$, uniquely characterize $M_t$ and $z_t$ as the solution of a time-dependent matrix Dyson equation, \eqref{eq:timeMDE} below.

Steps 2 and 3 are also proven by resolvent expansion, which we do not go into in this section. They are combined to show $\cal M_t \approx 0$, using a standard martingale concentration estimate, which is a Burkholder-Davis-Gundy-type inequality for jump processes from \eqref{BDG} below. It states that
\begin{equation} \label{BDG_intro}
\proba{\exists t \geq 0 :  |\Mart_t| \geq a, \langle \Mart \rangle_t \leq b^2} \leq 4 \exp{\left[-\frac{a^2}{2(aK + b^2)}\right]}\,,
\end{equation}
where $K$ is an upper bound for the jump sizes of $\cal M_t$. The estimate \eqref{BDG_intro} is in general sharp. We note the transition from subgaussian tails for $aK \lesssim b^2$ to subexponential tails for $aK \gtrsim b^2$, as in for instance Bennett's inequality. In particular, for continuous martingales ($K = 0$), such as Brownian motion, the tails are always subgaussian. In order to obtain estimates with very high probability, we require the exponent on the right-hand side of \eqref{BDG_intro} to be at least $D \log N$. This places a restriction on $a$, $b$, and $K$, which is one of the main reasons why our local law requires the lower bound $Np \gtrsim (\log N)^2$. This restriction arises in the regime of subexponential tails. See Remark \ref{rem:lowerbound_p} below for a more details discussion.

Once $\cal M_t \approx 0$ is known, we can plug \eqref{L_cancel_intro} into \eqref{Dynkin} to obtain
\begin{equation} \label{phi_t_sc_intro}
\phi_t = \int_0^t \dd s \, \pbb{1 + \frac{1 - p_s}{(1 - p)p} \ang{G_s^2}} \phi_s + \cal E(t)\,,
\end{equation}
where $\cal E$ is a collective error term. This can be solved using a (stochastic) Grönwall inequality to show that $\phi_T$ is small, as desired.

It turns out that the full isotropic local law with optimal error bounds requires a priori bounds from an entrywise local law. Thus, we first prove a combined entrywise and average local law. One central error term arising from the resolvent expansion (see \eqref{eq:lambda0MM} below) cannot be controlled with sufficient accuracy in this framework, however. Instead, it has to be handled using isotropic resummation, using that $\absb{\sum_x G_{ax}}$ is typically much smaller than $\sum_{x} \abs{G_{ax}}$. Because of this, our entrywise local law tracks resolvent entries in the standard basis as well as the special vector $\bm e$, which allows us to bootstrap the isotropic resummation in \eqref{eq:lambda0MM}.

As is already apparent from the stochastic propagator in \eqref{phi_t_sc_intro}, in order to control first-order observables such as \eqref{observables_phi2}, we need control on second-order observables depending on $G^2$. In fact, the (customary) trivial bound $|\langle G^2 \rangle | \le |\langle \im G \rangle|/\eta$ is not affordable as a bound on the propagator, due to an additional $p$-dependent error term in the local laws (see Section \ref{subsec:edge} below for a detailed discussion). Hence, we have to perform the entire flow argument not just for the first-order observables \eqref{observables_phi} but for their second-order versions, where $(\theta X - z_t)^{-1}$ is replaced with $(\theta X - z_t)^{-2}$. Together, the first- and second-order observables form a system whose flow equations are closed up to error terms that can be estimated in terms of these observables\footnote{The first paper to use second-order observables in the proof of a single resolvent local law using the method of characteristics is \cite{cipolloni2025maximum}.}. We note that, a posteriori, second-order observables can be easily controlled in terms of first-order quantities using Cauchy's integral formula $G^2(z) = \frac{1}{2 \pi \ii} \oint \frac{G(w)}{(w - z)^2} \, \dd w$, but because our flow argument operators at a fixed time-dependent spectral parameter $z_t$, such non-local combinations of $G$ are not available.

In fact, the strategy outlined above actually does not quite work, because the essentially optimal bound on the jump sizes in Step 2 is too large for \eqref{BDG_intro} to be effective. The reason is that, in proving the isotropic (or entrywise) local law using the observable $\scalar{\bm v}{(G - M) \bm w}$, most jumps are small but a few are large. This largeness arises from the appearance of too many diagonal entries of $G$ in the resolvent expansion. Hence, we in fact have to modify the Bernoulli flow, depending on the choice of the vectors $\bm v$ and $\bm w$, whereby a suitably chosen set of edges, whose switching would yield a too large jump, is \emph{frozen} in the dynamics. Throughout the flow, we then switch them on \emph{statically} as appropriate, using a resolvent expansion. This set of edges has to be sufficiently small to ensure that we can still control the error arising from switching them on. For the entrywise case, where we control $G_{ab} - M_{ab}$, we simply freeze the single edge $ab$. For general $\bm v, \bm w$, the set is constructed at the beginning of Section \ref{sec:isolaw} below.

\subsection*{Acknowledgments} We are grateful to Amirali Hannani and Hong-Quan Tran for numerous insightful discussions. We would like to thank Paul Bourgade and Jiaoyang Huang for sharing their preprint \cite{BourHua2026} with us and for helpful explanations of their result. We acknowledge funding from the European Research Council (ERC) and the Swiss State Secretariat for Education, Research and Innovation (SERI) through the consolidator grant ProbQuant, as well as funding from the Swiss National Science Foundation through the NCCR SwissMAP grant.

\section{Local law and proof of Theorem \ref{thm:deloc}}
The proof of Theorem \ref{thm:deloc} is based on a local law -- a concentration estimate for the resolvent of the adjacency matrix,
\begin{equation}
G(z) \deq (H- z)^{-1} \,, \quad z \in \C \setminus \R \,. 
\end{equation}
The deterministic approximation of $G(z)$ is given by the unique solution $M(z)$ of the \emph{Matrix Dyson Equation} (MDE) \cite{erdos2019matrix, ajanki2019stability}
\begin{equation} \label{eq:MDE}
- \frac{1}{M(z)} = z - p \theta \Smat + \langle M(z) \rangle \quad \text{with} \quad \im M(z) \, \im z > 0 \,, 
\end{equation}
where $\Smat$ denotes the adjacency matrix of the complete graph, which has entries
\begin{equation*}
\Smat_{xy} \deq 1 - \delta_{xy}\,.
\end{equation*}
The solution $M(z)$ of the MDE \eqref{eq:MDE} can be thought of being close to the identity times the Stieltjes transform of the semicircle law, up to a small shift and a low-rank perturbation; see Remark~\ref{rem:simplified_sc} below. However, this approximation is in general not affordable within the sharp bounds of Theorem~\ref{thm:lolaw} below\footnote{In Remark \ref{rem:simplified_sc} below, we give a simplified version of our local law that is expressed in terms of the semicircle law.}. The solution $M(z)$ satisfies the following estimates, which are proved in Appendix \ref{app:technical}. 
\begin{lemma}[Estimates on $M$] \label{lem:Mprop}
For $|z| \lesssim 1$ and $\bm v, \bm w \in \R^N$, we have
\begin{equation} \label{eq:Mprop}
|M_{\bm v \bm w}(z)| \asymp |\langle \bm v, (\mathbf{1} - \Pi_{\bm e})\bm w \rangle| + \frac{|\langle \bm v, \Pi_{\bm e} \bm w \rangle|}{\sqrt{Np}} \,. 
\end{equation}
Moreover,
\begin{equation} \label{eq:ImMprop}
|(\im M(z))_{\bm v \bm w}| \lesssim  |\langle \bm v, (\mathbf{1} - \Pi_{\bm e})\bm w \rangle| + \frac{|\langle \bm v, \Pi_{\bm e} \bm w \rangle|}{Np} \lesssim \vertiii{\bm v}_{\bm e} \, \vertiii{\bm w}_{\bm e} \,. 
\end{equation}
\end{lemma}

The local law holds on the spectral domain
  \begin{equation} \label{eq:domain}
\mathbb{D} \equiv \mathbb{D}(\kappa, C) \deq [-2 + \kappa, 2 - \kappa] \times \ii [C N^{-1} \log N, 1] \subset \HH\,, 
  \end{equation}
  where $\HH$ denotes the complex upper half-plane and $\kappa, C > 0$ are constants. We always use the notation
  \begin{equation*}
\eta \equiv \eta(z) \deq \im z
  \end{equation*}
  for the imaginary part of $z$. The following result is the main technical achievement of this paper.
  
\begin{theorem}[Local law] \label{thm:lolaw}
The following holds under the assumptions of Theorem \ref{thm:deloc}.
\begin{subequations}
\begin{itemize}
\item[(i)]  \emph{The average local law:}
	\begin{equation} \label{eq:avLL}
		\P \pBB{ \bigcap_{z \in \mathbb{D} } \hBB{ \left| \big\langle G(z) - M(z)\big\rangle \right| \leq C \frac{1}{\sqrt{N \eta}}\pBB{\sqrt{\frac{ \log N}{N \eta}} + \sqrt{\frac{\log N}{N p}}  }} } \geq 1 - N^{-D} \,. 
	\end{equation}
	\item[(ii)] \emph{The isotropic local law}: For all deterministic $\bm v, \bm w \in \R^N$
	\begin{equation} \label{eq:isoLL}
		\P \pBB{  \bigcap_{z \in \mathbb{D} } \hBB{ \left| \big( G(z) - M(z)\big)_{\bm v \bm w} \right| \leq C \,  \vertiii{\bm v}_{\bm e} \,  \vertiii{\bm w}_{\bm e} \,  \pBB{\sqrt{\frac{\log N}{N \eta}} + \sqrt{\frac{\log N}{Np} } }}} \geq  1 - N^{-D}\,.
	\end{equation}
\end{itemize}
	\end{subequations}
\end{theorem}

\begin{remark}
To the best of our knowledge, the optimal factor $\frac{\sqrt{\log N}}{N \eta}$ in the average law \eqref{eq:avLL} has previously been obtained only for $\beta$-ensembles \cite{bourgade2022optimal}, in particular in no independent Hermitian ensemble other than GOE, let alone sparse matrices. For Wigner matrices (and non-Hermitian i.i.d.\ matrices), the bound $\frac{(\log N)^{1/2 + \epsilon}}{N \eta}$ was obtained in \cite{cipolloni2025maximum}. See also Appendix \ref{sec:GOE_ll}, where we give a simple Brownian characteristic flow argument, inspired by our main proof, that yields the optimal factor $\frac{\sqrt{\log N}}{N \eta}$ for GOE.
\end{remark}

For off-diagonal resolvent entries, we have to following improved estimate.

\begin{remark}[Off-diagonal local law] \label{rmk:offdiag}
In the entrywise case $\bm v = \bm e_a$ and $\bm w = \bm e_b$, the following improved version of the local law \eqref{eq:isoLL} holds: 
		\begin{equation} \label{eq:offidag}
		\P \pBB{ \bigcap_{a,b \in [N]}\bigcap_{z \in \mathbb{D}} \hBB{ \left| \big( G(z) - M(z)\big)_{ab} \right| \leq C \,   \pBB{\sqrt{\frac{\log N}{N \eta}} + \sqrt{\frac{1 + \delta_{ab}\log N}{Np} } }}} \geq  1 - N^{-D} \,. 
	\end{equation}
 We refer to Section \ref{subsec:offdiag} below for the proof.
\end{remark}

We believe all of the estimates in \eqref{eq:avLL}, \eqref{eq:isoLL}, and \eqref{eq:offidag} to be optimal up to the constant $C$; see Section \ref{sec:optimal} below.

Next, we introduce a notion of stochastic domination\footnote{Our notion $\preceq$ is related to the commonly used notion $\prec$ introduced in \cite{EKY2}, with the difference that the bound $N^\epsilon$ is replaced by a constant $C$.}, denoted by $\preceq$, that will prove useful throughout the rest of this paper.

\begin{definition}
Let $X \equiv X_N(u)$ and $Y \equiv Y_N(u)$ be nonnegative random variables indexed by $N$ and possibly some parameter $u$. We write
\begin{equation*}
X \preceq Y
\end{equation*}
to mean that for any $D > 0$ there exists $C > 0$ such that
\begin{equation*}
\sup_u \P\pb{X_N(u) > C Y_N(u)} \leq N^{-D}\,.
\end{equation*}
\end{definition}

We can simplify Theorem \ref{thm:lolaw} and state it in terms of the traditional semicircle density \eqref{eq:rhodef}.

\begin{remark}[Simplified local law] \label{rem:simplified_sc}
Recall \eqref{eq:rhodef} and denote by 
\begin{equation} \label{def_m_sc}
m_{\mathrm{sc}}(z) \deq \int_{-2}^2 \dd s \, \frac{\rho_s}{s - z}
\end{equation}
the Stieltjes transform of the semicircular density.
Using elementary properties of the solution $M$ to the MDE \eqref{eq:MDE} (see, in particular, \eqref{eq:Mmsc}--\eqref{eq:Mspecdec} in Appendix \ref{app:technical}), under the assumptions of Theorem \ref{thm:lolaw}, we deduce the following simplified local laws for all $z \in \mathbb{D}$.
\begin{itemize}
\item[(i)] The average local law: 
\begin{equation*}
\big| \langle G(z) - m_{\rm sc}(z + p \theta)\rangle \big| \preceq \frac{1}{\sqrt{N \eta}} \left( \sqrt{\frac{\log N}{N \eta}} + \sqrt{\frac{\log N}{Np}}\right) \,. 
\end{equation*}
\item[(ii)] The entrywise local law: for $a,b \in [N]$,
\begin{equation*}
\big| G_{ab}(z) - \delta_{ab} m_{\rm sc}(z)\big| \preceq  \sqrt{\frac{\log N}{N \eta}} + \sqrt{ \frac{1 + \delta_{ab}\log N}{N p}} \,. 
\end{equation*}
\item[(iii)] The isotropic local law (without $\bm e$-improvement): for $\Vert \bm v \Vert , \Vert \bm w \Vert \leq 1$,
\begin{equation*}
\big| G_{\bm v \bm w}(z) - \langle \bm v, \bm w \rangle m_{\rm sc}(z)\big| \preceq  \sqrt{\frac{\log N}{N \eta}} + \sqrt{\frac{\log N}{Np}}\,. 
\end{equation*}
\item[(iv)] The $\bm e$-local law:
\begin{equation*}
\left|G_{\bm e \bm e}(z) - \sqrt{\frac{1-p}{Np}} - \frac{1-p}{Np} (z + m_{\rm sc}(z)) \right| \preceq \frac{1}{Np}\left(\sqrt{\frac{\log N}{N \eta}} + \sqrt{ \frac{\log N}{N p}}\right) \,. 
\end{equation*}
\end{itemize}
\end{remark}

We now give the proof of Theorem \ref{thm:deloc} based on the local law. 

\begin{proof}[Proof of Theorem \ref{thm:deloc}]
Deducing delocalization from an entrywise or isotropic local law is standard (see, e.g., \cite[Theorem 2.10]{BenyachKnowles2017}). For the reader's convenience, we recall the simple argument here. First, note that the local law \eqref{eq:isoLL} obviously also holds when replacing $G$ and $M$ by $\im G$ and $\im M$, respectively. Denote by $\lambda_1, \dots, \lambda_N$ the eigenvalues of $H$, with associated normalized eigenvectors $\bm u_i$. Take an eigenvalue $\lambda_i \in [-2+\kappa, 2-\kappa]$ and spectral parameter $z_i = \lambda_i + \ii \eta  \in \mathbb{D}$ with $\eta = CN^{-1} \log N$. Then, for $\bm v \in \R^N$, we deduce from Lemma \ref{lem:Mprop} and Theorem \ref{thm:lolaw} that 
\begin{equation}
	\begin{split}
\frac{\langle \bm v, \bm u_i \rangle^2}{\eta} &\leq \sum_j \frac{\langle \bm v, \bm u_j \rangle^2 \, \eta}{(\lambda_i - \lambda_j)^2 + \eta^2} = (\im G(z_i))_{\bm v \bm v} \\
&\leq (\im M(z_i))_{\bm v \bm v} + 2 C^{1/2}\vertiii{\bm v}_{\bm e}^2  \leq C_\kappa \vertiii{\bm v}_{\bm e}^2
	\end{split}
\end{equation}
with probability $\geq 1 - N^{-D}$,
for a constant $C_\kappa$ depending only on $\kappa$. 
We deduce that
\begin{equation}
\proba{\langle \bm v, \bm u_i \rangle^2 \leq C_\kappa C \, \frac{ \log N}{N} \vertiii{\bm v}_{\bm e}^2} \geq 1 - N^{-D} \,. 
\end{equation}
The claim now follows from a union bound. 
\end{proof}

\begin{remark}[Proof of Remark \ref{rem:digraphs}] \label{rem:digraphs_proof}
	We now briefly explain how our proof of Theorem \ref{thm:deloc} has to be adapted in order to conclude the proof of the claimed optimal delocalization bound for the left- and right-eigenvectors of the Erd\H{o}s-Rényi digraph in Remark \ref{rem:digraphs}. We refer to  \cite[Corollary~2.6]{AEK2018} as well as \cite{cipolloni2025optimal, alt2021spectral, alt2021inhomogeneous, CEHS24, campbell2025spectral} for similar arguments and will hence be rather brief. 
	
	As usual, the argument proceeds by considering the \emph{Hermitization}
	\begin{equation*}
\widetilde{H}^w := \begin{pmatrix}
0 & \widehat{H} - w \\
\widehat{H}^* - \overline{w} & 0
\end{pmatrix} \in \C^{2N \times 2N} \qquad \text{with} \qquad w \in \C 
	\end{equation*}
	of the non-Hermitian adjacency matrix $\widehat{H}$. By following the proof of our local law in Sections \ref{sec:Bernoulli}--\ref{sec:isolaw} with  only minor straightforward adjustments, we can show the following: Uniformly in $|w| \le 1 - \kappa$ and $\eta \in [C N^{-1} \log N, 1]$, the resolvent 
	\begin{equation*}
G^w(\eta) := \big(\widetilde{H}^w - \ii \eta\big)^{-1}
	\end{equation*}
	concentrates around a deterministic matrix $M^w(z) \in \C^{2N \times 2N}$ given as the solution to 
	\begin{equation*}
- \frac{1}{M^w(\eta)} = \ii \eta - p \theta \widetilde{\Smat} + \begin{pmatrix}
0 & w \mathbf{1} \\ \overline{w} \mathbf{1} & 0 
\end{pmatrix} + \langle M^w(\eta) \rangle \quad \text{where} \quad \widetilde{\Smat} := \begin{pmatrix}
0 & \Smat \\ \Smat & 0
\end{pmatrix} \in \C^{2N \times 2N} \,, 
	\end{equation*}
	exactly as in Theorem \ref{thm:lolaw}. The only difference is that in the definition of the anisotropic norm \eqref{eq:normdef} the orthogonal projection $\Pi_{\bm e}$ is replaced by
$
\pb{
\begin{smallmatrix}
\Pi_{\bm e} & 0
\\
0 & \Pi_{\bm e}
\end{smallmatrix}
}
$	
(a rank-two orthogonal projection). 
	
	Hence, completely analogously to the Hermitian setting, we obtain optimal delocalization for the eigenvectors of the Hermitization $\widetilde{H}^w$. To conclude the same statement fo the true left and right non-Hermitian eigenvectors of $\widehat{H}$, we observe that $0 \in \spec(\widetilde{H}^\lambda)$ if and only if $\lambda \in \spec(\widehat{H})$. Therefore, by uniformity of the local law in $|w| \le 1- \kappa$, we infer \eqref{eq:directed_deloc} (see \cite[Proof of Theorem~2.2]{cipolloni2025optimal} for more details). 
\end{remark}

\subsection{On the optimality of the local law} \label{sec:optimal}
We conclude this section with a discussion on the optimality of our estimates from Theorem \ref{thm:lolaw} and Remark \ref{rmk:offdiag}. Indeed, up to the precise dependence of the constant $C$ on $D$, we believe the estimates in the average local law from \eqref{eq:avLL}, the isotropic law (with $\bm e$-improvement) from \eqref{eq:isoLL}, and the off-diagonal local law from \eqref{eq:offidag} to be optimal. In the dense regime\footnote{As explained in Remark \ref{rem:p_range}, the regime $p \asymp 1$ is strictly speaking not allowed by our assumptions, but we believe that our method could extend to such values of $p$ with suitable adjustments.}, where $p \asymp 1$, \eqref{eq:avLL} and \eqref{eq:isoLL} reduce to
\begin{equation*}
\left| \big\langle G(z) - M(z)\big\rangle \right| \lesssim \frac{\sqrt{\log N}}{N \eta}\,, \qquad
\left| \big( G(z) - M(z)\big)_{\bm v \bm w} \right| \lesssim \vertiii{\bm v}_{\bm e} \,  \vertiii{\bm w}_{\bm e} \sqrt{\frac{\log N}{N\eta}}
\end{equation*}
with very high probability. Leaving the $\bm e$-improvement in the factor $\vertiii{\bm v}_{\bm e} \,  \vertiii{\bm w}_{\bm e}$ aside, these bounds are optimal because $N \eta \big\langle G(z) - M(z)\big\rangle$ and $\sqrt{N \eta} G(z) - M(z)\big)_{\bm v \bm v}$ are known to be approximately Gaussian random variables of order one. Note the factor $\sqrt{\log N}$ instead of $\log N$ in the bound for $\ang{G(z)}$ (compare e.g.\ to \cite[Theorem 3.1]{cipolloni2025optimal}); in  Appendix \ref{sec:GOE_ll} below, we explain how a simplified version of our argument for the Brownian flow immediately yields an average local law for GOE with the optimal $\sqrt{\log N}$ factor, even down to $\eta \asymp \sqrt{\log N}/N$.

The effect of sparsity, when $p \ll 1$, is manifested by the factor $\sqrt{\frac{\log N}{N p}}$, which is natural since the local law and delocalization are known \cite{ADK20} to break down on the critical scale $Np \asymp \log N$. Moreover, we claim that $\sqrt{\frac{\log N}{N p}}$ is in general the optimal power of this factor in \eqref{eq:isoLL}. Indeed, for the diagonal entry $G_{aa}(z)$, this error arises from the fluctuation of the vertex degrees. This follows from \cite[Theorem 4.2]{ADK20}, where it is proved that for any $\eta \geq N^{-1 + c}$ we have $\abs{G_{aa}(z) - m_{\alpha_a}(z)} \lesssim \pb{\frac{\log N}{(Np)^2}}^{1/3}$, where
\begin{equation*}
m_{\alpha}(z) \deq -\frac{1}{z + \alpha m_{\mathrm{sc}}(z)}
\end{equation*}
and $\alpha_a \deq \frac{1}{Np} \sum_{x\in [N]} A_{ax}$ is the normalized degree. By normal approximation, the law of $\alpha_a$ is approximately $1 + (Np)^{-1/2} \cal N(0,1)$, which implies that $m_{\alpha_a}(z)$, and hence $G_{aa}(z)$, concentrates with very high probability only on the scale $\sqrt{\frac{\log N}{N p}}$ around $m_{\mathrm{sc}}(z)$, as claimed. 

For off-diagonal entries of $G(z)$, on the other hand, the stronger estimate \eqref{eq:offidag} is optimal. Indeed, for $a \neq b$, we use a consequence of the Schur complement formula (see e.g.\ \cite[Eq.\ (5.9)]{BenyachKnowles2017}),
\begin{equation*}
G_{ab} = - G_{aa} G_{bb}^{(a)} \pbb{H_{ab} - \sum_{x,y \notin \{a,b\}} H_{ax} G_{xy}^{(ab)} H_{yb}}\,,
\end{equation*}
where upper indices indicate removed vertices. The first term in parentheses is independent of the second, and it can be either $0$ or $\theta \asymp \sqrt{\frac{1}{Np}}$, with probability $1 - p$ and $p$, respectively. This shows the optimality of the factor $\sqrt{\frac{1}{Np}}$ for $a \neq b$ in \eqref{eq:offidag}.

The average local law has an error that is improved by the usual factor $(N \eta)^{-1/2}$ as compared to the isotropic law. As explained above, in the dense regime, this estimate is optimal up to the constant $C$. In the sparse regime $p \ll 1$, it is also known to be optimal \cite{ADK20, lee2018local}. However, as first observed in \cite{lee2018local}, as soon as $p \ll N^{-1/3}$, the accuracy average local law can be improved provided one replaces $\ang{M(z)}$ with a more accurate deterministic approximation, e.g.\ $\E\ang{G(z)}$, which differs from $M(z)$ by order $\frac{1}{Np}$. In fact, in \cite{he2020bulk} it was proved in the regime $\eta \geq N^{-1 + c}$ and $Np \geq N^c$ that $\frac{1}{\zeta} (\ang{G(z)} - \E \ang{G(z)})$ with $\zeta \deq \frac{1}{N\eta} + \frac{1}{\sqrt{N} \sqrt{Np}}$ is approximately Gaussian of order one.

Finally, the $\bm e$-improvement for $G_{\bm e \bm e}$ yields an extra factor $\frac{1}{Np}$ in the error bound of the isotropic law. This is because the top eigenvector $\bm u_1$ is almost parallel to $\bm e$, while the others are almost orthogonal to it; see Remark \ref{rem:deloc_optimal}. A precise form of the local law for $G_{\bm e \bm e}$ is given in Remark \ref{rem:simplified_sc}.

As for the assumptions on $p$, they are the same as in Theorem \ref{thm:deloc}, and we refer to Remark \ref{rem:p_range} above and Remark \ref{rem:lowerbound_p} below.

\subsection{Extending the local law to the edge} \label{subsec:edge}
In this section we discuss the nontrivial obstructions arising when attempting to extend our local law towards the spectral edge. To illustrate the issues, we start discussing the \emph{Brownian flow} described in Section \ref{subsec:overview} and ignore the second term on the right-hand side of \eqref{dG_ito}. Abbreviating $\phi_t := \langle G_t(z_t) - M_t(z_t) \rangle$, the structure of \eqref{dG_ito} is given by (see also Appendix \ref{sec:GOE_ll})
\begin{equation} \label{eq:phiODE}
\dd \phi_t = \mathcal{E}_1(t) \phi_t \, \dd t + \dd \mathcal{E}_2(t) \,. 
\end{equation}
Hence, in order to prove smallness of $\phi_T$, based on smallness of $\phi_0$ (propagating the local law to smaller spectral scales), one has to control the \emph{generator term} $\mathcal{E}_1(t)$ and the (stochastic) \emph{error term} $\dd \mathcal{E}_2(t)$. In fact, (formally) the solution to \eqref{eq:phiODE} is given by
\begin{equation} \label{eq:solving}
\phi_t = \mathcal{P}_{0,t} \phi_0 + \int_{0}^t  \mathcal{P}_{s,t} \, \dd \mathcal{E}_2(s)
\end{equation}
with $\mathcal{P}_{s,t}$ being the \emph{propagator} given by
\begin{equation}
\mathcal{P}_{s,t} := \exp\left( \int_{s}^t \mathcal{E}_1(r) \dd r \right) \,. 
\end{equation}
As one proves self-consistently (see again Appendix \ref{sec:GOE_ll} below), the main contribution to the propagator is obtained by the deterministic approximation to $\mathcal{E}_1(r)$, given by $\langle M'_r(z_r) \rangle := \big\langle \tfrac{\dd }{\dd w} M_r(w) \big\vert_{w = z_r}\big\rangle$, the pointwise $z$-derivative of $\langle M \rangle$, the Stieltjes transform of the density of states. In the spectral bulk $\langle M \rangle$ is smooth, i.e. $|\langle M'_r(z_r) \rangle| \sim 1$ and thus the propagator is harmless, $\mathcal{P}_{s,t} \sim 1$. At the spectral edge, however, $\langle M'_r(z_r) \rangle$ blows up (due to non-smoothness of the density of states) and behaves as\footnote{We note that $\eta_r^{-1} \langle \im M_r(z_r) \rangle$ is always a valid upper bound on the generator term. This follows by estimating $\langle G^2 \rangle \le \eta^{-1} \langle \im G\rangle $ by a Cauchy-Schwarz inequality and a Ward identity, and then self-consistently using a single-resolvent law. The more refined control involving $\langle M'\rangle$ requires proving two-resolvent concentration, $\langle G^2 \rangle \approx \langle M'\rangle$, as well (see also the discussion in Section \ref{subsec:LLproof} below). } 
\begin{equation*}
|\langle M'_r(z_r) \rangle| \sim \frac{\langle \im M_r(z_r) \rangle}{\eta_r}  \quad \text{with} \quad \eta_r := \im z_r \,. 
\end{equation*}
Therefore, using $\tfrac{\dd }{\dd t} (\log \eta_t) = - \tfrac{\langle \im M_r(z_r) \rangle}{\eta_r}$ as a consequence of \eqref{Brownflow}, we obtain 
\begin{equation} \label{propbound}
\int_{s}^t \frac{\langle \im M_r(z_r) \rangle}{\eta_r} \dd r = - \log\left(\frac{\eta_t}{\eta_s}\right) \qquad \text{and thus} \qquad \mathcal{P}_{s,t} \lesssim \frac{\eta_s}{\eta_t} \,. 
\end{equation}
In particular, at the edge, the propagator effectively replaces a full $\eta$-power at a larger spectral scale $\eta_s$ by a full $\eta$-power at a smaller spectral scale $\eta_t$. Therefore, in order to self-consistently derive a local law at the spectral edge, the error term $\dd \mathcal{E}_2(t)$ and the initial condition $\phi_0$ have to feature a full $\eta$-power and thus be compatible with the sharp propagator bound derived in \eqref{propbound}. 

The above discussion carries over to the Bernoulli flow setting after suitable adjustments. The crucial difference to the Brownian flow, however, is that the \emph{average local law does not include a full $\eta$-power}\footnote{We remark that, since a $\sqrt{\eta}$ is present in \eqref{eq:avLL}, we could extend the domain of validity of our local in such a way that $\mathcal{P}_{s,t} \lesssim \sqrt{\eta_s/\eta_t}$ is ensured. This would amount to replacing $\kappa \asymp 1$ in \eqref{eq:domain} by a suitable quantity $\ll 1$.}; see \eqref{eq:avLL}. We stress that this is not a consequence of suboptimal estimates; indeed, as discussed in Section~\ref{sec:optimal}, we believe our bounds to be optimal. As indicated in Remark \ref{rmk:edge} above, a possible way to overcome this issue is to prove concentration of $G$ around a suitably explicit random approximation insteaf of the deterministic $M$, entirely absorbing the sparsity-dependent error terms in the local law; see again Section \ref{sec:optimal}.

\section{Bernoulli flow: Proof of Theorem \ref{thm:lolaw}} \label{sec:Bernoulli}

\subsection{Definition of the Bernoulli flow} \label{subsec:Berflowdef}

Denote by $\Gamma$ the space of adjacency matrices, i.e.\ functions $X \col \{(x,y) \in [N]^2 \col x < y\} \to \{0,1\}$. It can be naturally identified with symmetric matrices in $\{0,1\}^{N \times N}$ with zero diagonal.

The \emph{Bernoulli flow} consists of two ingredients.
\begin{enumerate}[label=(\roman*)]
\item
The \emph{Bernoulli process} $(X(t))_{0 \leq t \leq T}$, a Markov jump process on the space $\Gamma$.
\item
The \emph{Bernoulli characteristic flow} $(z_t)_{0 \leq t \leq T}$, a deterministic matrix-valued flow for the spectral parameter $z \in \HH \mathbf{1}+ \R \Smat$.
\end{enumerate}

These ingredients are combined in the \emph{time-dependent resolvent}
\begin{equation} \label{eq:Gtzt}
G_t(z_t) \deq \frac{1}{\theta X(t) - z_t}\,, \qquad t \in [0,T]\,,
\end{equation}
which we propagate from time $t=0$ to the target time $T$, at which $X(T) \stackrel{\dd}{=} A(p)$ and $\im z_T \asymp (\log N) / N$.

To define the Bernoulli process, for $X \in \Gamma$ and $x < y$ we introduce the notation $X^{xy,1}$ to denote the matrix obtained from $X$ by replacing the $xy$-entry by $1$, which has entries
\begin{equation*}
(X^{xy,1})_{ab} \deq
\begin{cases}
X_{ab} & \text{if } \{a,b\} \neq \{x,y\}
\\
1 & \text{if } \{a,b\} = \{x,y\}\,.
\end{cases}
\end{equation*}
The Bernoulli process $X(t)$ is the right-continuous Markov jump process with initial condition $X(0) = 0 \in \Gamma$ and generator
\begin{equation} \label{eq:generator}
\Lcal\phi (X) \deq \sum_{x< y}\big( \phi(X^{xy, 1}) - \phi(X) \big)\,, \qquad \phi \col \Gamma \to \C\,.
\end{equation}
That is, each entry of $X(t)$ starts from $0$ and jumps to $1$ at unit rate, independently of the others. Consequently,
\begin{equation} \label{eq:pt}
X_{xy}(t) \stackrel{\dd}{=} \op{Bernoulli}(p_t) \quad \text{with} \quad p_t \deq 1 - \ee^{-t} \,.
\end{equation} 
Hence, $X(t)$ is the adjacency matrix of an Erd\H{o}s-Rényi graph with edge probability $p_t$:
\begin{equation} \label{eq:ERasflow}
X(t) \stackrel{\dd}{=} A(p_t) \,. 
\end{equation}
Given the target edge probability $p \equiv p(N)$, we choose the \emph{terminal time} $T \equiv T(N)$  of the Markov process by requiring that $p = p_T = 1- \ee^{-T}$, i.e.
\begin{equation} \label{eq:targettime}
T = \log \left(\frac{1}{1 - p}\right) = p + O(p^2) \quad \text{for} \quad p \to 1\,. 
\end{equation}

Next, we define the Bernoulli characteristic flow and derive its basic properties.
For a target spectral parameter $z \in \HH$ we consider the terminal value problem
\begin{equation} \label{eq:Bernflow}
\partial_t z_t = - z_t + \theta \Smat - \frac{1 - p_t}{1 - p} \frac{m_t(z_t)}{p} \,, \qquad z_T = z \mathbf{1}\,, \quad t \in [0,T]
\end{equation}
on $\HH \mathbf{1}+ \R \Smat$, which we call the \emph{Bernoulli characteristic flow}. Here, we introduced the shorthand notation
\begin{equation*}
m_t(z_t) \deq \langle M_t(z_t) \rangle \in \mathbb{H}\,,
\end{equation*}
where $M_t(z_t)$ is the unique solution with positive imaginary part to the time-dependent MDE: 
\begin{equation} \label{eq:timeMDE}
- \frac{1}{M_t(z_t)} = z_t - \frac{p_t}{\sqrt{N p(1 - p)}} \Smat + \frac{p_t (1 - p_t)}{p (1-p)} \langle M_t(z_t) \rangle \,. 
\end{equation}
In the following lemma, we collect some elementary properties of \eqref{eq:Bernflow}--\eqref{eq:timeMDE}, which are proved in Appendix \ref{app:technical}.
\begin{lemma}[Properties of the Bernoulli characteristic flow] \label{lem:BerProp}
\leavevmode
\begin{itemize}
\item[(a)] The terminal value problem \eqref{eq:Bernflow} has a unique solution $[0,T] \ni t \mapsto z_t \in \HH \mathbf{1} + \R \Smat$. 
\item[(b)] The solution to \eqref{eq:timeMDE} satisfies $M_t(z_t) = \ee^t M_0(z_0)$ for all $t \in [0,T]$. 
\item[(c)] The imaginary part $\eta_t$ of the solution $z_t$ to \eqref{eq:Bernflow} is a multiple of the identity, i.e.\ with a slight abuse of notation, $\eta_t = \eta_t \mathbf{1}$. Moreover, $t \mapsto \eta_t$ is explicitly given by 
\begin{equation} \label{eq:etaexplicit}
\eta_t = \eta_0 \ee^{-t} - \frac{ \im m_0(z_0) }{p (1 - p)} (1 - \ee^{-t})
\end{equation}
and is hence a strictly monotonically decreasing function.  
\item[(d)] For any $s,t \in [0,T]$ with $s \leq t$ we have the integration rules
\begin{align}
\int_{s}^{t} \frac{1 - p_t}{1 - p} \frac{\im m_r(z_r) }{p\,  \eta_r} \, \dd r &\leq \log \left(\frac{\eta_s}{\eta_t}\right) \quad \text{and} \label{eq:logint} \\ 
 \int_{s}^{t} \frac{1 - p_t}{1 - p} \frac{\im  m_r(z_r)}{p\,  \eta_r^\alpha} \, \dd r &\leq \frac{1}{\alpha - 1} \left(\frac{1}{\eta_t^{\alpha - 1}} - \frac{1}{\eta_s^{\alpha - 1}}\right) \quad \text{for} \quad \alpha >1 \label{eq:intrule}\,. 
\end{align}
\end{itemize}
\end{lemma}

In our proof, we always restrict ourselves to the bulk, meaning that the real part of $z_T = z$ lies in $[-2 + \kappa, 2 - \kappa]$ (recall \eqref{eq:domain}). Using Lemma \ref{lem:BerProp} (b) and \eqref{eq:Mmsc} below, this implies that, for any $r \in [0,T]$,
\begin{equation*}
\im m_r(z_r) \asymp \im m_T(z_T) = \im \ang{M(z)} = \im m_{\rm sc}(z + p \theta) + O\pbb{\frac{1}{N}} \geq c_\kappa\,,
\end{equation*}
where the last step follows by elementary properties of $m_{\rm sc}$. Thus we get the simplified versions of \eqref{eq:logint}--\eqref{eq:intrule}, reading
\begin{equation} \label{eq:intrulepractice}
\int_{s}^{t}  \frac{1}{p\,  \eta_r} \, \dd r \lesssim \log N \qquad \text{and} \qquad  \int_{s}^{t} \frac{1}{p\,  \eta_r^\alpha} \, \dd r \lesssim \frac{1}{\eta_t^{\alpha - 1}} \qquad \text{for} \quad \alpha > 1 \,. 
\end{equation}
Additionally, we will often use that, by \eqref{eq:targettime}, we have that $T \lesssim p$, and hence, trivially, $\int_{s}^{t} \dd r \, p^{-1} \lesssim 1$.

\subsection{Proof of Theorem \ref{thm:lolaw}} \label{subsec:LLproof}
The goal of this section is to give a proof of Theorem \ref{thm:lolaw}~(i) and the \emph{entrywise} part Theorem~\ref{thm:lolaw}~(ii), i.e.~for standard basis vectors $\bm v = \bm e_a$, $\bm w = \bm e_b$. To do so, we follow the evolution of four (families of) observables along the Bernoulli process in tandem with the Bernoulli characteristic flow. Before we start, for $a,b \in [N]$, we define the set of vectors
\begin{equation} \label{eq:Vab}
\mathcal{V}_{ab} \deq \{\bm e_a, \bm e_b, \bm e\} \qquad \text{where} \qquad \bm e \deq N^{-1/2}(1,\dots,1)^* \in \R^N \,. 
\end{equation}
The reason for including $\bm e$ in the set $\Vab$ of vectors is that one place in the argument we need to employ \emph{isotropic resummation}, by which $\bm e$ naturally emerges (see \eqref{eq:lambda0MM} below). 

The observables are given by $\phi^{{\rm av},1}_t$ and  $\phi^{{\rm av},2}_t$ in the average case, and by $\phi^{{\rm ent},1}_{\bm a \bm b, t}$ and $\phi^{{\rm ent},2}_{\bm a \bm b, t}$, for $\bm a, \bm b \in \mathcal{V}_{ab} $ with $a,b \in [N]$, in the entrywise case. More precisely, the \emph{average observables} $\phi^{{\rm av},1}_t$ and  $\phi^{{\rm av},2}_t$ are defined as 
\begin{equation} \label{eq:avobs}
\phi^{{\rm av},1}_{t}(X(t)) \deq \langle G_t(z_t) - M_t(z_t) \rangle \quad \text{and} \quad \phi^{{\rm av},2}_{t}(X(t)) \deq \langle G^2_t(z_t) - M'_t(z_t) \rangle
\end{equation}
where we denoted 
\begin{equation} \label{eq:Mprime}
M_t'(z_t) \deq \frac{\dd }{\dd w} M_t(w) \bigg|_{w = z_t} \,. 
\end{equation}
For $a,b \in [N]$ and $\bm a, \bm b \in \Vab$, the \emph{entrywise observables} $\phi^{{\rm ent}, 1}_{\bm a \bm b, t}$ and $\phi^{{\rm ent}, 2}_{\bm a \bm b, t}$ are defined as
\begin{equation} \label{eq:entobs}
\phi^{{\rm ent}, 1}_{\bm a \bm b, t}(X(t)) \deq \big(G^{ab}_t(z_t) - M_t(z_t)\big)_{\bm a \bm b} \quad \text{and} \quad \phi^{{\rm ent},2}_{\bm a \bm b, t}( X(t)) \deq \left(\big(G^{ab}_t(z_t)\big)^2 - M'_t(z_t)\right)_{\bm a \bm b} \,, 
\end{equation}
where $G^{ab}_t(z_t) = \big(\theta X^{ab}(t) - z_t\big)^{-1}$ is the resolvent of the random matrix $X^{ab}(t)$, which has its $(a,b)$ and $(b,a)$ entry set to zero. Note that in the diagonal case, $a=b$, the resolvent $G^{aa}$ agrees with $G$, since $X^{aa} \equiv 0$. Moreover, the Bernoulli process $X(t)$ and the Bernoulli characteristic flow $z_t$ from Section \ref{subsec:Berflowdef} satisfy $G_0(z_0) = M_0(z_0) = - 1 / z_0$, from which we deduce that
\begin{equation} \label{eq:initialcond}
\phi^{{\rm av}, i}_{0}(X(0)) = 0 \qquad \text{and} \qquad \phi^{{\rm ent}, i}_{\bm a \bm b,0}(X(0)) = 0 \quad \text{for} \quad a,b \in [N]\,, \ \bm a, \bm b \in \Vab \quad \text{and} \quad  i \in [2] \,. 
\end{equation}

We stress that control on the two-resolvent observables $\phi^{{\rm av},2}_{t}$ and $\phi^{{\rm ent}, 2}_{\bm a \bm b, t}$ can \emph{a posteriori} be derived from bounds on the corresponding single-resolvent observables using Cauchy's integral formula 
\begin{equation} \label{eq:contour}
G^2(z) = \frac{1}{2 \pi \ii} \oint_{\gamma} \frac{G(w)}{(z-w)^2} \dd w
\end{equation}
for a suitable contour $\gamma \subset \C$ encircling $z$, e.g., $\gamma = \{w : |w-z| = |\im z|/2\}$, and an analogous formula for the deterministic approximation. Within the proof, however, such an argument involving an integral representation of $G^2$ would be rendered circular, as it necessarily involves other spectral parameters, thus losing locality of the argument. For this reason, we have to treat $\phi^{{\rm av},2}_{t}$ and $\phi^{{\rm ent}, 2}_{\bm a \bm b, t}$ as genuinely new observables along the proof. 

In the remainder of the paper, we will often use the shorthand notations
\begin{equation} \label{eq:shorthands}
\hspace{-3mm}G_t \deq G_t(z_t)\,, \quad G_t^{ab} \deq G_t^{ab}(z_t)\,, \quad M_t \deq M_t(z_t)\,, \quad \phi_t^{{\rm av}, i} \deq  \phi_t^{{\rm av}, i}(X(t))\,, \quad  \phi_{\bm a \bm b, t}^{{\rm ent}, i} \deq  \phi_{\bm a \bm b, t}^{{\rm ent}, i}(X(t)) 
\end{equation}
for $a,b \in [N]$ and $i \in [2]$. 
That is, in particular, we omit the spectral parameter as an argument of a resolvent, and drop the argument $X(t)$ of the observables defined above. 
\subsubsection{Stopping time} \label{subsubsec:stoptime}
We introduce a stopping time $\tau$, defined as the minimal time at which $|\phi_t^{{\rm av}, i}|$ or $|\phi_{\bm a \bm b, t}^{{\rm ent}, i}|$ for $i \in [2]$ is larger than a suitably chosen threshold. By constructing four martingales associated with the observables \eqref{eq:avobs}--\eqref{eq:entobs}, we then show that our observables are actually much \emph{smaller} than the threshold level of the stopping time, from which we conclude that $\tau = T$. Given an $N$-independent constant $\Cf> 0$, to be chosen in \eqref{eq:Cchoice} below, we define the stopping time
\begin{equation} \label{eq:stoptime}
\tau = \tau(\Cf) \deq \inf \left\{  t \in [0,T] \, : \, \max_{i \in [2]} \frac{\big| \phi^{{{\rm av}}, i}_{t} \big|}{ \Lambda_{{\rm av}}^{(i)}(t)} + 
\max_{i \in [2]}\max_{a,b \in [N]} \max_{\bm a, \bm b \in \Vab}\,\frac{ \big| \phi_{\bm a \bm b, t}^{{{\rm ent}}, i} \big| }{ \Lambda_{{\rm ent}}^{(i)}(t)} \geq \Cf 
\right\}\,,
\end{equation}
where we introduced the shorthand notations
\begin{alignat}{3}
\Lambda_{{\rm av}}^{(1)}(t) &\deq \frac{{ (\log N)^{1/2}}}{N \eta_t} + \sqrt{\frac{\log N}{(N \eta_t) (N p)}} \,, \qquad &&\Lambda_{{\rm av}}^{(2)}(t) &&\deq \sqrt{\frac{\log N}{N \eta_t^3}} \label{eq:Lambdaav}\\
\Lambda_{{\rm ent}}^{(1)}(t) &\deq \sqrt{\frac{ \log N}{N \eta_t}} + \sqrt{\frac{\log N}{N p}} \,, \qquad  \qquad  &&\Lambda_{{\rm ent}}^{(2)}(t) &&\deq \frac{1}{\eta_t}\left(\sqrt{\frac{ \log N}{N \eta_t}} + \sqrt{\frac{\log N}{N p}} \right)\label{eq:Lambdaent}
\end{alignat}
for the average and entrywise time-dependent control parameters, respectively. 

Away from the jumps of $X(t)$, all quantities in the definition of $\tau$ are continuous. Combined with the fact that the first jump occurs almost surely at a positive time, we conclude that $\tau > 0$ almost surely.

\subsubsection{Martingale construction} \label{subsubsec:Mart}
We now construct Dynkin martingales associated with the average and entrywise observables. For $i \in [2]$, these are given by
\begin{equation} \label{eq:Mav}
\Mart^{{\rm av}, i}_t \deq \phi^{{\rm av}, 	i}_t\big(X(t)\big)  - \phi^{{\rm av}, 	i}_0\big(X(0)\big) - \int_{0}^{t} \dd s \ \Big( \big(\Lcal \phi^{{\rm av}, 	i}_s\big)\big(X(s)\big) + \partial_s \phi^{{\rm av}, 	i}_s\big(X(s)\big)   \Big) 
\end{equation}
and 
\begin{equation} \label{eq:Ment}
	\Mart^{{\rm ent}, i}_{\bm a \bm b,t} \deq \phi^{{\rm ent}, i}_{\bm a \bm b,t}\big(X(t)\big)  - \phi^{{\rm ent}, i}_{\bm a \bm b,0}\big(X(0)\big) - \int_{0}^{t} \dd s \ \Big( \big(\Lcal \phi^{{\rm ent}, i}_{\bm a\bm b,s}\big)\big(X(s)\big) + \partial_s \phi^{{\rm ent}, i}_{\bm a\bm b,s}\big(X(s)\big)   \Big) 
\end{equation}
for any $a,b \in [N]$ and $\bm a, \bm b \in \Vab$. In Appendix \ref{app:CDC}, we prove the following result (see Lemmas \ref{lem:dynkin_mart} and \ref{lem:quadr_var} as well as Remark \ref{rem:time_dep_phi}). 
\begin{lemma}[Associated martingales and predictable quadratic variations] \label{lem:martingales}
For $i \in [2]$, let $\Mart_t =  \Mart^{{\rm av}, i}_t$ or $\Mart_t = \Mart^{{\rm ent}, i}_{\bm a \bm b,t}$, for any $a,b \in [N]$ and $\bm a, \bm b \in \Vab$, and let $\phi_t= \phi^{{\rm av}, i}_t $ or  $ \phi_t = \phi^{{\rm ent}, i}_{\bm a \bm b,t}$, respectively. Then $\big(\Mart_{t} \big)_{t \geq 0}$ is a mean-zero càdlàg martingale, whose predictable quadratic variation is given by
\begin{equation} \label{eq:QVformula}
	\begin{split}
\langle \Mart\rangle_t &= \int_0^t \dd s \ \Big( \big(\Lcal \abs{\phi_s}^2 \big)\big(X(s)\big) - 2 \re \phi_s\big(X(s)\big) \big(\Lcal \phi_s\big)\big(X(s)\big)\Big)  \\
&= \int_0^t \dd s \ \sum_{x < y} \, \absb{\phi_s\big(X^{xy,1}(s)\big) - \phi_s\big(X(s)\big)}^2 \,.
	\end{split}
\end{equation}
(Recall the notation $X^{xy,1}$ introduced below \eqref{eq:generator}.)
\end{lemma}
\subsubsection{Stochastic Grönwall estimates and proof of Theorem \ref{thm:lolaw}}
After having constructed the stopping time in Section~\ref{subsubsec:stoptime} and the martingales associated with the average and entrywise observables in Section~\ref{subsubsec:Mart}, in the present section, we give the proofs of the average and entrywise local laws in Theorem~\ref{thm:lolaw} by showing that $\tau =T$. The key to achieving this is the following \emph{stochastic Grönwall estimates}. 
\begin{proposition}[Stochastic Grönwall estimates] \label{prop:Gron}
Adopt the notations and conventions from above, and assume that 
\begin{equation}
Np \geq (1+\mathfrak{C})^2 (\log N)^2 \quad \text{and} \quad N \eta_T \geq (1+\mathfrak{C})^2 \log N\,.
\end{equation}
Then there exists a constant $C_\kappa$, depending only on $\kappa$ from \eqref{eq:domain}, such that for any $c \geq 10 \Cf^{-1/2}$,
\begin{subequations} \label{eq:avGron}
\begin{equation} \label{eq:avGron1}
\proba{\exists t \in [0,T]: \big| \phi^{{\rm av},1}_{t\wedge \tau} \big| \geq C_\kappa\int_{0}^{t \wedge \tau}  \dd s \, \frac{1}{p}\left(1 + \Cf \Lambda_{{\rm av}}^{(2)}(s) \right) \, \big| \phi^{{\rm av},1}_s \big| +  (1+10c \Cf) \Lambda_{{\rm av}}^{(1)}(t \wedge \tau)} \leq  N^{-c\Cf/C_\kappa}
\end{equation}
and
\begin{equation} \label{eq:avGron2}
		\proba{\exists t \in [0,T] : 	\big| \phi^{{\rm av},2}_{t\wedge \tau} \big| \geq C_\kappa \int_{0}^{t \wedge \tau} \dd s \, \frac{1}{p} \left(1 + \Cf \Lambda_{{\rm av}}^{(2)}(s) \right) \, \big| \phi^{{\rm av},2}_s \big| +   (1+10 c \Cf)\Lambda_{{\rm av}}^{(2)}(t \wedge \tau)} \leq  N^{-c \Cf/C_\kappa}
\end{equation}
\end{subequations}
for the average observables, and 
\begin{subequations} \label{eq:entGron}
\begin{multline} \label{eq:entGron1}
\proba{	\exists t \in [0,T], \ a,b \in [N], \ \bm a, \bm b \in \Vab : \big| \phi^{{\rm ent}, 1}_{\bm a \bm b, t\wedge \tau} \big| \geq C_\kappa \int_{0}^{t \wedge \tau} \dd s \, 	\big| \phi^{{\rm ent}, 1}_{\bm a \bm b, s} \big| \, + (1+ 10 c \Cf)\Lambda_{{\rm ent}}^{(1)}(t \wedge \tau)} \\
  \leq  N^{- c \Cf/C_\kappa+10}
\end{multline}
and
\begin{multline} \label{eq:entGron2}
\proba{	\exists t \in [0,T], \ a,b \in [N], \ \bm a, \bm b \in \Vab : \big| \phi^{{\rm ent}, 2}_{\bm a \bm b, t\wedge \tau} \big| \geq C_\kappa \int_{0}^{t \wedge \tau} \dd s \, 	\big| \phi^{{\rm ent}, 2}_{\bm a \bm b, s} \big| \, +  (1+10 c \Cf)\Lambda_{{\rm ent}}^{(2)}(t \wedge \tau)} \\
\leq N^{- c \Cf/C_\kappa+10}
\end{multline}
\end{subequations}
for the entrywise observables. 
\end{proposition}
The proof of Proposition \ref{prop:Gron} is given in Section \ref{sec:Gronproof} below. 

\begin{proof}[Proof of the average and entrywise local laws in Theorem \ref{thm:lolaw}~(i)+(ii)]
	We apply Grönwall's inequality to each of the four statements in \eqref{eq:avGron}--\eqref{eq:entGron}, involving \eqref{eq:intrulepractice}. This implies that there exists a constant $C_\kappa$, depending only on $\kappa$, such that for any $c$ as in Proposition \ref{prop:Gron}, we have, for $i \in [2]$, 
	\begin{equation*}
		\begin{split}
\proba{\exists t \in [0,T] : \big|\phi^{{\rm av}, i}_{t \wedge \tau}\big| \geq \exp(C_\kappa)(1+10c \Cf) \Lambda_{{\rm av}}^{(i)}(t \wedge \tau)} &\leq N^{-c \Cf/C_\kappa} \,, \\
\proba{	\exists t \in [0,T], \ a,b \in [N], \ \bm a, \bm b \in \Vab : \big| \phi^{{\rm ent}, i}_{\bm a \bm b, t\wedge \tau} \big| \geq \exp(C_\kappa)  (1+ 10 c \Cf)\Lambda_{{\rm ent}}^{(i)}(t \wedge \tau)} &\leq  N^{- c \Cf/C_\kappa+10} \,.
		\end{split}
	\end{equation*}
	Hence, by choosing 
	\begin{equation} \label{eq:Cchoice}
 \Cf \geq 4 \exp(C_\kappa)  \quad \text{and} \quad c = \exp(-C_\kappa)/40
	\end{equation}
	we find that, for $i \in [2]$,
		\begin{equation*}
		\begin{split}
			\proba{\exists t \in [0,T] : \big|\phi^{{\rm av}, i}_{t \wedge \tau}\big| \geq \frac{\Cf }{2} \Lambda_{{\rm av}}^{(i)}(t \wedge \tau)} &\leq N^{-\Cf/C_\kappa'} \,, \\
			\proba{	\exists t \in [0,T], \ a,b \in [N], \ \bm a, \bm b \in \Vab : \big| \phi^{{\rm ent}, i}_{\bm a \bm b, t\wedge \tau} \big| \geq \frac{\Cf}{2}\Lambda_{{\rm ent}}^{(i)}(t \wedge \tau)} &\leq  N^{- \Cf/C_\kappa'+10} \,,
		\end{split}
	\end{equation*}
	where we abbreviated $C_\kappa' = 40 C_\kappa \exp(C_\kappa)$. Therefore, by definition of the stopping time \eqref{eq:stoptime}, we conclude
	\begin{equation}
\proba{\tau(\Cf) <T} \leq N^{- \Cf/C_\kappa'+11}\,,
	\end{equation}
	i.e.~we have proven the local laws to hold for fixed spectral parameters. 
To deduce the uniform statements in \eqref{eq:avLL}--\eqref{eq:isoLL}, we apply a simple grid argument. Additionally, since the statement in \eqref{eq:isoLL} is formulated for the actual resolvent $G$ and not the modified $G^{ab}$ that has the $(a,b)$ and $(b,a)$ entries set to zero, we need to apply a simple resolvent exansion formulated in Lemma~\ref{lem:resolventexpand} to obtain the desired statement for $G$. This is formalized in Lemma~\ref{lem:entforG}~(i) below. 

This concludes the proof of Theorem \ref{thm:lolaw}~(i) and the purely entrywise part of Theorem \ref{thm:lolaw}~(ii); the extension to general vectors is discussed in Section \ref{sec:isolaw} below. 
\end{proof}

\section{Proof of Proposition \ref{prop:Gron}} \label{sec:Gronproof}
The goal of this section is to give the proof of Proposition \ref{prop:Gron}. The argument is divided in three main lemmas, formulated below, that control the predictable quadratic variation of the martingale terms \eqref{eq:Mav}--\eqref{eq:Ment} (Lemma \ref{lem:martest}), their respective jump sizes (Lemma \ref{lem:jump}), and finally the generator terms in \eqref{eq:Mav}--\eqref{eq:Ment} (Lemma \ref{lem:generate}). In the final lemma, we crucially use a cancellation among the two terms in 
\begin{equation}
\Lcal \phi_s + \partial_s \phi_s
\end{equation}
for all considered observables $\phi_s$, which arises from our choice of the Bernoulli characteristic flow \eqref{eq:Bernflow}. 

We now formulate the above mentioned three lemmas precisely. Their proofs are given in Section~\ref{sec:lemproofs}. 
\begin{lemma}[Quadratic variation of martingale terms] \label{lem:martest}
Using the assumptions and notations from above, there exists a constant $C_\kappa > 0$ such that for $i \in [2]$, and uniformly in $t \in [0,T]$ and $a,b \in [N]$ and $\bm a, \bm b \in \Vab$, 
\begin{equation}
\langle \Mart^{{\rm av}, i}\rangle_{t \wedge \tau} \leq \frac{C_\kappa}{\log N} \big(\Lambda_{{\rm av}}^{(i)}(t \wedge \tau)\big)^2 \quad \text{and} \quad \langle\Mart^{{\rm ent},i}_{\bm a \bm b}\rangle_{t \wedge \tau} \leq \frac{C_\kappa}{\log N} \big(\Lambda_{{\rm ent}}^{(i)}(t \wedge \tau)\big)^2 \,. 
\end{equation}
\end{lemma}

\begin{lemma}[Bound on jump sizes] \label{lem:jump}
Using the assumptions and notations from above, there exists a constant $C_\kappa > 0$ such that for $i \in [2]$, and uniformly in $t \in [0,\tau]$ and $a,b \in [N]$ and $\bm a, \bm b \in \Vab$, 
\begin{equation} \label{jump_sizes}
\left|\Mart^{{\rm av}, i}_{t} - \Mart^{{\rm av}, i}_{t-}\right| \leq \frac{C_\kappa}{\sqrt{Np \log N}} \Cf\Lambda_{{\rm av}}^{(i)}(t)\quad \text{and} \quad \left|\Mart^{{\rm ent},i}_{\bm a \bm b, t} - \Mart^{{\rm ent},i}_{\bm a \bm b, t-}\right| \leq \frac{C_\kappa}{\sqrt{Np}} \Cf\Lambda_{{\rm ent}}^{(i)}(t) \,. 
\end{equation}
\end{lemma}

\begin{lemma}[Generator terms] \label{lem:generate}
Using the assumptions and notations from above, there exists a constant $C_\kappa > 0$ such that, uniformly for $t \in [0,T]$, 
\begin{align}
\left|\int_{0}^{t\wedge \tau} \dd s \ \big( \Lcal \phi^{{\rm av}, 	1}_s + \partial_s \phi^{{\rm av}, 	1}_s  \big) \right| &\leq C_\kappa \int_{0}^{t \wedge \tau} \dd s \, \frac{1}{p}\left(1 + \Cf \Lambda_{{\rm av}}^{(2)}(s) \right) \, \big| \phi^{{\rm av},1}_s \big| + \Lambda_{{\rm av}}^{(1)}(t \wedge \tau)\\
\left|\int_{0}^{t\wedge \tau} \dd s \ \big( \Lcal \phi^{{\rm av}, 	2}_s + \partial_s \phi^{{\rm av}, 	2}_s  \big) \right| &\leq C_\kappa \int_{0}^{t \wedge \tau} \dd s \, \frac{1}{p}\left(1 + \Cf \Lambda_{{\rm av}}^{(2)}(s) \right) \, \big| \phi^{{\rm av},2}_s \big| + \Lambda_{{\rm av}}^{(2)}(t \wedge \tau)
\end{align}
for the average generator terms, and, for $i \in [2]$, uniformly in $t \in [0,T]$, $a,b \in [N]$ and $\bm a, \bm b \in \Vab$, 
\begin{equation}
	\left|\int_{0}^{t\wedge \tau} \dd s \ \big( \Lcal \phi^{{\rm ent}, 	i}_{\bm a \bm b, s} + \partial_s \phi^{{\rm ent}, 	i}_{\bm a \bm b, s}  \big) \right| \leq C_\kappa \int_{0}^{t \wedge \tau} \dd s \,  \big| \phi^{{\rm ent},i}_{\bm a \bm b, s} \big| + \Lambda_{{\rm ent}}^{(i)}(t \wedge \tau)
\end{equation}
all with very high probability. 
\end{lemma}

Armed with Lemmas \ref{lem:martest}, \ref{lem:jump}, and \ref{lem:generate}, we can readily conclude the proof of Proposition \ref{prop:Gron}.

\begin{proof}[Proof of Proposition \ref{prop:Gron}]
To prove Proposition \ref{prop:Gron}, we combine Lemmas \ref{lem:martest} and \ref{lem:jump} with a jump process version of the BDG inequality, formulated in Proposition \ref{prop:DBG} below (see also Remark \ref{rem:time_dep_phi}). 
More precisely, to apply Proposition \ref{prop:DBG} in its time-homogeneous formulation, we rescale the martingales by their natural time-dependent sizes. Hence, there exists a constant $C_\kappa > 0$, depending only on $\kappa$, such that, for any $c \geq 10 \mathfrak{C}^{-1/2}$, and $i \in [2]$, 
\begin{align}
&\proba{\exists t \in [0,T]: \left|\Mart^{{\rm av}, i}_{t\wedge \tau}\right| \geq 10 c \,  \Cf\Lambda_{{\rm av}}^{(i)}(t\wedge \tau)} \leq  N^{-c \Cf/C_\kappa} \label{eq:Mavbound} \\
&\proba{ \exists t \in [0,T], \ a,b \in [N], \ \bm a, \bm b \in \Vab : 	\left|\Mart^{{\rm ent},i}_{\bm a \bm b, t\wedge \tau}\right| \geq 10 c \,  \Cf\Lambda_{{\rm ent}}^{(i)}(t\wedge \tau)}\leq N^{-c \Cf/C_\kappa+10} \label{eq:Mentbound}
\end{align}
where for the second (entrywise) bound we additionally employed a union bound over $a,b \in [N]$ and $\bm a, \bm b \in \Vab$. Combining \eqref{eq:Mavbound}--\eqref{eq:Mentbound} with Lemma \ref{lem:generate} and \eqref{eq:Mav}--\eqref{eq:Ment}, additionally using \eqref{eq:initialcond}, we conclude the proof of Proposition \ref{prop:Gron}. 
\end{proof}

\begin{remark}[Origin of the lower bound on $p$] \label{rem:lowerbound_p}
We comment on the importance of the lower bound $Np \geq C (\log N)^2$ in our proof. It is used in several crucial steps of our argument. Perhaps the most prominent one is in the estimate of Green function entries in terms of the parameter $\Lambda_{{\rm ent}}^{(1)}$ from \eqref{eq:Lambdaav}. We use the Burkholder-Davis-Gundy-type inequality from Proposition \ref{prop:DBG} below to estimate the size of Dynkin martingale $\Mart^{{\rm ent}, i}_{\bm a \bm b,t}$ from \eqref{eq:Ment} in terms of its quadratic variation. Since our entire argument is predicated on very high probability bounds, we require the right-hand side of \eqref{BDG} to be of oder $\ee^{-D \log N}$, which in particular requires $a/K \geq 4 D \log N$. Our upper bound on the jump size of $\Mart^{{\rm ent}, i}_{\bm a \bm b,t}$ is $K = C (Np)^{-1/2} \Lambda_{{\rm ent}}^{(1)}$ (see \eqref{jump_sizes}), which essentially arises by resolvent expansion and is sharp. Since our goal in \eqref{BDG} is the bound $a = \Lambda_{{\rm ent}}^{(1)}$, we therefore immediately arrive at the condition $a/K = \frac{1}{C}\sqrt{Np}$, and hence $Np \geq (4CD \log N)^2$.
\end{remark}

\section{Resolvent expansions: Proofs of Lemmas \ref{lem:martest}, \ref{lem:jump}, and \ref{lem:generate}} \label{sec:lemproofs}

The goal of this section is to give the proofs of Lemmas \ref{lem:martest}, \ref{lem:jump}, and \ref{lem:generate} in Sections \ref{subsec:martestproof}, \ref{subsec:jumpproof}, and~\ref{subsec:generateproof}, respectively. The \emph{key tool} throughout the entire section are simple \emph{resolvent expansions}, formalized in Lemma \ref{lem:resolventexpand}, and additionally using the Cauchy-Schwarz inequality in summations and the \emph{Ward identity}
\begin{equation} \label{Ward}
G(z)G^*(z) = \frac{\im G(z)}{\im z}\,,
\end{equation}
which follows from the resolvent identity.
Further, recall that our observables in \eqref{eq:avobs} and \eqref{eq:entobs} are designed to control $G$ in the average case and the modified $G^{ab}$ in the entrywise case. Hence, on top of the above mentioned tools, we will use that we have average laws for $G^{ab}$ and isotropic laws for $G$ as well. This is formalized in the following lemma. 

\begin{lemma}[Going from $G$ to $G^{ab}$ and from $G^{ab}$ to $G$] \label{lem:entforG}
	Let $t  \in [0, \tau]$, i.e.~suppose that the local law bounds collected in \eqref{eq:stoptime} hold. Then, at the expense of replacing $\Cf$ by $2 \Cf$, we have the following:  
	\begin{itemize}
		\item[(i)] For any $a, b \in [N]$ and $\bm a, \bm b \in \Vab$, the entrywise laws for $(G_s^{ab})_{\bm a \bm b}$ and $\big((G_s^{ab})^2\big)_{\bm a \bm b}$ also hold for $(G_s)_{\bm a \bm b}$ and $(G_s^2)_{\bm a \bm b}$.  
		\item[(ii)] For any $a, b , a', b' \in [N]$ and $\bm a, \bm b \in \Vab$, the entrywise laws for $(G_s^{ab})_{\bm a \bm b}$ and $\big((G_s^{ab})^2\big)_{\bm a \bm b}$ also hold for $(G_s^{a'b'})_{\bm a \bm b}$ and $\big((G_s^{a'b'})^2\big)_{\bm a \bm b}$. 
\item[(iii)] For any $a, b \in [N]$, the average laws for $\langle G_s \rangle$ and $\langle G_s^2 \rangle$ also hold for $\langle G_s^{ab}\rangle$ and $\langle (G_s^{ab})^2 \rangle$. 
	\end{itemize}
\end{lemma}

The proof of Lemma \ref{lem:entforG} is based on simple resolvent expansions and we provide a brief sketch in Appendix~\ref{app:technical}. 

Throughout the whole section, we will use the notation $\lesssim $ to absorb constants $C_\kappa > 0$, depending only on the bulk parameter $\kappa$.  
\subsection{Proof of Lemma \ref{lem:martest}} \label{subsec:martestproof}
In this section we give the proof of Lemma \ref{lem:martest}, estimating the four predictable quadratic variations separately. 

\subsubsection{Estimates for \texorpdfstring{$\Mart^{{\rm ent},1}_{\bm a \bm b, t\wedge \tau}$}{$\Mart^{{\rm ent},1}_{a b, t\wedge \tau}$}} \label{subsubsec:Martent1estimates}
We start by considering the predictable quadratic variation of $\Mart^{{\rm ent},1}_{\bm a \bm b, t\wedge \tau}$. As $(a,b)$ remains fixed throughout this section, we write $\tG_t \equiv G^{ab}_t(z_t)$ for short (recall the notation introduced below \eqref{eq:entobs}) and will occasionally even omit the time dependence, i.e.~write $\tG \equiv \tG_t$. 

To control $\langle \Mart^{{\rm ent},1}_{\bm a \bm b}\rangle_{t \wedge \tau}$, we perform a standard \emph{resolvent expansion}, cast in the following lemma. Its proof is a straightforward computation and so omitted. 
\begin{lemma}[Resolvent expansion] \label{lem:resolventexpand}
For any $x,y \in [N]$ with $x \neq y$, we have
\begin{equation}
	\tG^{xy,1}  = \sum_{\ell=0}^m \big(\theta (X_{xy}-1)\big)^\ell \big(\tG \Delta_{xy}\big)^\ell \tG + \big(\theta (X_{xy}-1)\big)^{m+1} \big(\tG \Delta_{xy}\big)^m \tG^{xy,1}
\end{equation}
valid for any fixed $m \in \N_0$, where we introduced the shorthand notation
\begin{equation}
	\Delta_{xy} = \Delta_{yx}= \Smat^{(xy)} + \Smat^{(yx)} \quad \text{with} \quad \big(\Smat^{(xy)}\big)_{ij} = \delta_{xi} \delta_{yj} (1 - \delta_{xy})
\end{equation}
i.e.~$\Smat^{(xy)} \in \R^{N \times N}$ is the matrix with all zeros except at the $(x,y)$ entry (unless $x=y$). 

The statement holds verbatim when replacing $\tG \to G$. 
\end{lemma}
Armed with Lemma \ref{lem:resolventexpand}, we can readily estimate the predictable quadratic variation of $\Mart^{{\rm ent},1}_{\bm a \bm b, t\wedge \tau}$. In fact, expanding to third order (i.e.~taking $m=2$ in Lemma \ref{lem:resolventexpand}), we find
\begin{equation} \label{eq:Martent1}
\langle \Mart^{{\rm ent},1}_{\bm a \bm b}\rangle_{t \wedge \tau} \lesssim  \int_{0}^{t \wedge \tau} \dd s \, \sum_{x ,y} \big| \big(\tG^{xy,1}_{s}\big)_{\bm a \bm b} - \big(\tG_{s}\big)_{\bm a \bm b} \big|^2 \leq \Rcal^{(1)}_{\bm a \bm b}(t \wedge \tau) +  \Rcal^{(2)}_{\bm a \bm b}(t \wedge \tau) + \Rcal^{(3)}_{\bm a \bm b}(t \wedge \tau)
\end{equation}
where we denoted
\begin{align}
\Rcal^{(1)}_{\bm a \bm b}(t \wedge \tau) &\deq \theta^2 \int_{0}^{t \wedge \tau} \dd s \,   \sum_{x ,y}  \big| \big(\tG_s \Delta_{xy} \tG_s\big)_{\bm a \bm b} \big|^2\label{eq:R1}\,, \\
\Rcal^{(2)}_{\bm a \bm b}(t \wedge \tau) &\deq \theta^4 \int_{0}^{t \wedge \tau} \dd s \,  \sum_{x ,y} \big| \big(\tG_s \Delta_{xy} \tG_s \Delta_{xy} \tG_s\big)_{\bm a \bm b}\big|^2\label{eq:R2}\,, \\
\Rcal^{(3)}_{\bm a \bm b}(t \wedge \tau)&\deq\theta^6 \int_{0}^{t \wedge \tau} \dd s \, \sum_{x ,y} \big| \big(\tG_s \Delta_{xy} \tG_s \Delta_{xy} \tG_s \Delta_{xy} \tG^{xy,1}_s\big)_{\bm a \bm b}\big|^2 \,. \label{eq:R3}
\end{align}

These three terms \eqref{eq:R1}, \eqref{eq:R2}, and \eqref{eq:R3} shall now be estimated separately. We start with \eqref{eq:R1}, for which we have
\begin{equation} \label{eq:R1est}
	\begin{split}
\Rcal^{(1)}_{\bm a \bm b}(t \wedge \tau) &\lesssim \frac{1}{Np }\int_{0}^{t \wedge \tau} \dd s \,   \sum_{x ,y} \big|\big(\tG_s\big)_{\bm a x} \big(\tG_s\big)_{y\bm b}\big|^2 \lesssim  \frac{1}{Np }\int_{0}^{t \wedge \tau} \dd s \,  \frac{(\im \tG_s)_{\bm a \bm a} (\im \tG_s)_{\bm b \bm b}}{\eta_s^2} \\
&\lesssim \frac{1}{Np }\int_{0}^{t \wedge \tau} \dd s \frac{1}{\eta_s^2} \lesssim \frac{1}{N \eta_{t \wedge \tau}} \lesssim \frac{1}{\log N} \big(\Lambda_{{\rm ent}}^{(1)}(t \wedge \tau)\big)^2
	\end{split}
\end{equation}
uniformly in $a,b \in [N]$ and $\bm a, \bm b \in \mathcal{V}_{ab}$.
Here, in the second step, we employed the Ward identity (see \eqref{Ward})
\begin{equation*}
\sum_{x} |G(z)_{\bm ax}|^2 = \frac{(\im G(z))_{\bm a \bm a}}{\im z}\,.
\end{equation*}
To go to the second line, we then used the definition of the stopping time, which allows us to bound
\begin{equation} \label{eq:Gest}
|(\im G_s)_{\bm a \bm a}| \leq |(\im M_s)_{\bm a \bm a}| + \Cf \Lambda_{{\rm ent}}^{(1)}(s) \lesssim 1\quad \text{for all} \quad a \in [N] \quad \text{and} \quad \bm a \in \mathcal{V}_{aa} \,,
\end{equation}
i.e.~for all standard basis vectors and also for $ \bm a = \bm e = N^{-1/2} (1, \dots, 1)^*$. 
In the penultimate step in \eqref{eq:R1est}, we then employed \eqref{eq:intrulepractice} for $\alpha = 2$. The ultimate bound simply follows by the definition \eqref{eq:Lambdaent} of $\Lambda_{{\rm ent}}$.

Next, we turn to bounding $\Rcal^{(2)}_{\bm a \bm b}(t \wedge \tau)$ from \eqref{eq:R2}. We split this term in two parts, diagonal and off-diagonal, 
$$ 	\Rcal^{(2)}_{\bm a \bm b}(t \wedge \tau) \lesssim   		\Rcal^{(2, \rm D)}_{\bm a \bm b}(t \wedge \tau) + 	\Rcal^{(2, \rm OD)}_{\bm a \bm b}(t \wedge \tau)$$ 
defined as
\begin{equation} \label{eq:R2DOD}
	\begin{split}
\Rcal^{(2, \rm D)}_{\bm a \bm b}(t \wedge \tau)  &= \theta^4 \int_{0}^{t \wedge \tau} \dd s \,  \sum_{x ,y} \big| \big(\tG_s\big)_{\bm a x}  \big(\tG_s\big)_{yy}  \big(\tG_s\big)_{x\bm b}\big|^2 \quad \text{and} \\
\Rcal^{(2, \rm OD)}_{\bm a \bm b}(t \wedge \tau)  &= \theta^4 \int_{0}^{t \wedge \tau} \dd s \,  \sum_{x ,y} \big| \big(\tG_s\big)_{\bm a x}  \big(\tG_s\big)_{yx}  \big(\tG_s\big)_{y\bm b}\big|^2
	\end{split}
\end{equation}
respectively. We start by estimating $\Rcal^{(2, \rm D)}_{\bm a \bm b}(t \wedge \tau)$, which can be bounded as
\begin{equation} \label{eq:R2Dest}
	\begin{split}
\Rcal^{(2, \rm D)}_{\bm a \bm b}(t \wedge \tau) &\lesssim \frac{1}{(Np)^2} \int_{0}^{t \wedge \tau} \hspace{-2mm}\dd s \, \sum_{x ,y} \left(\big| (M_s)_{\bm ax} \big(\tG_s\big)_{yy}  \big(\tG_s\big)_{x\bm b}\big|^2 +  \big( \Cf \Lambda_{{\rm ent}}^{(1)}(s)\big)^2 \big| \big(\tG_s\big)_{yy}  \big(\tG_s\big)_{x\bm b}\big|^2\right) \\
&\lesssim  \frac{1}{Np^2} \int_{0}^{t \wedge \tau} \hspace{-2mm}\dd s \, \sum_{x } \left(\big|(M_s)_{\bm a x}\big|^2 + \big( \Cf \Lambda_{{\rm ent}}^{(1)}(s)\big)^2 \big|\big(\tG_s\big)_{x\bm b}\big|^2\right) \\
& \lesssim \frac{1}{Np} + \frac{\Cf^2}{N p^2} \int_0^{t \wedge \tau} \dd s \, \frac{\big(\Lambda_{{\rm ent}}^{(1)}(s)\big)^2}{\eta_s} \lesssim \frac{1}{\log N} \big(\Lambda_{{\rm ent}}^{(1)}(t \wedge \tau)\big)^2 \,. 
	\end{split}
\end{equation}
Here, in the first step, we used that, for $s \in [0, t \wedge \tau]$,
\begin{equation} \label{eq:G-Mest}
\big|\big(\tG_s\big)_{\bm ax} - (M_s)_{\bm ax}\big| \leq \Cf \Lambda_{{\rm ent}}^{(1)}(s)\,. 
\end{equation}
To go to the next line, we analogously bounded $\big(\tG_s\big)_{yy}$ and also $\big(\tG_s\big)_{x\bm b}$ in the first term. In the penultimate step, we then carried out the remaining $x$-summations, used the Ward identity in the second term and applied \eqref{eq:Gest}. Finally, in the last step, we employed the integral rules from \eqref{eq:intrulepractice} to get that 
\begin{equation}
\int_0^{t \wedge \tau} \dd s \, \frac{\big(\Lambda_{{\rm ent}}(s)\big)^2}{p \eta_s} \lesssim \frac{\log N}{N \eta_{t \wedge \tau}} + \frac{(\log N)^2}{Np}  \,, 
\end{equation}
and additionally used that $ Np \geq \Cf^2 (\log N)^2$. Next, we control $\Rcal^{(2, \rm OD)}_{\bm a \bm b}(t \wedge \tau)$, which, similarly to \eqref{eq:R2Dest}, can be bounded as
\begin{equation} \label{eq:R2ODest}
	\begin{split}
		\Rcal^{(2, \rm OD)}_{\bm a \bm b}(t \wedge \tau) \lesssim \frac{1}{(Np)^2} \int_{0}^{t \wedge \tau} \hspace{-2mm}\dd s \, \sum_{x ,y} \big| \big(\tG_s\big)_{\bm a x}  \big(\tG_s\big)_{y\bm b}\big|^2  \lesssim\frac{1}{(Np)^2} \int_{0}^{t \wedge \tau} \hspace{-2mm}\dd s \, \frac{1}{\eta_s^2}   \lesssim \frac{1}{\log N} \big(\Lambda_{{\rm ent}}^{(1)}(t \wedge \tau)\big)^2 \,. 
	\end{split}
\end{equation}

Finally, we turn to estimating $\Rcal^{(3)}_{\bm a \bm b}(t \wedge \tau)$ from \eqref{eq:R3}. 
We focus on one exemplary $x,y$ index constellation in \eqref{eq:R3}, arising from writing out $\Delta_{xy}$ defined in Lemma \ref{lem:resolventexpand}. For this term we have
\begin{align}
&\mspace{-20mu}\frac{1}{(Np)^3}\int_{0}^{t \wedge \tau} \dd s \, \sum_{x ,y} \big| \big(\tG_s\big)_{\bm ax}  \big(\tG_s\big)_{yy} \big(\tG_s \big)_{xx}  \big(\tG^{xy,1}_s\big)_{y\bm b}\big|^2 
\notag \\
& \lesssim \frac{1}{(Np)^3}\int_{0}^{t \wedge \tau} \dd s \, \big(\im \tG_s\big)_{\bm a \bm a} \left(\frac{\big(\im \tG_s\big)_{\bm b \bm b}}{\eta_s^2} + \frac{1}{p \eta_s}\right) 
\notag
\\
& \lesssim  \frac{1}{(Np)^3}\int_{0}^{t \wedge \tau} \dd s \left(\frac{1 }{\eta_s^2} + \frac{1}{p \eta_s}\right) 
\notag \\
&\lesssim \frac{1}{(Np)^2} \left(\frac{1}{N \eta_{t \wedge \tau}} + \frac{\log N}{Np}\right)
\notag \\ \label{eq:R3example}
&\lesssim \frac{1}{\log N} \big(\Lambda_{{\rm ent}}^{(1)}(t \wedge \tau)\big)^2 \,. 
\end{align}
Here, in the first step, we employed yet another resolvent expansion to write
\begin{equation} \label{eq:resexpaux}
\tG_{y\bm b}^{xy,1} = \tG_{y\bm b} - \theta (1 - X_{xy}) \big(\tG_{yx}  \tG_{y\bm b}^{xy,1} + \tG_{yy} \tG_{x\bm b}^{xy,1}\big) \,,
\end{equation}
which implies that
\begin{equation} \label{eq:resexp}
\big| \tG_{y\bm b}^{xy,1} \big|  \lesssim \big| \tG_{y\bm b}\big| + \theta 
\end{equation}
by application of \eqref{eq:G-Mest} and Lemma \ref{lem:entforG}, and two (resp.~one) Ward identities for the two terms corresponding to \eqref{eq:resexp}. To go to the second line, we employed \eqref{eq:Gest}  and finally used \eqref{eq:intrulepractice} together with \eqref{eq:Lambdaent} in the last two steps.  
All other index constellations can be handled similarly and we hence find that 
\begin{equation} \label{eq:R3est}
\Rcal^{(3)}_{\bm a \bm b}(t \wedge \tau) \lesssim  \frac{1}{\log N} \big(\Lambda_{{\rm ent}}^{(1)}(t \wedge \tau)\big)^2 \,. 
\end{equation}

Collecting the above estimates \eqref{eq:Martent1}, \eqref{eq:R1est}, \eqref{eq:R2Dest}, \eqref{eq:R2ODest}, and \eqref{eq:R3est}, we conclude that 
\begin{equation}
\langle \Mart^{{\rm ent},1}_{\bm a \bm b}\rangle_{t \wedge \tau} \lesssim \frac{1}{\log N} \big(\Lambda_{{\rm ent}}^{(1)}(t \wedge \tau)\big)^2 \,. 
\end{equation}

\subsubsection{Estimates for \texorpdfstring{$\Mart^{{\rm ent},2}_{\bm a \bm b, t\wedge \tau}$}{$\Mart^{{\rm ent},2}_{a b, t\wedge \tau}$}}
In this section, we estimate the predictable quadratic variation of $\Mart^{{\rm ent},2}_{\bm a \bm b, t\wedge \tau}$, using the same abbreviating notations as in the previous section. By a first order resolvent expansion (i.e.~applying Lemma \ref{lem:resolventexpand} with $m=0$), we find
\begin{equation} \label{eq:Martent2}
	\begin{split}
	\langle \Mart^{{\rm ent},2}_{\bm a \bm b}\rangle_{t \wedge \tau} &\lesssim  \int_{0}^{t \wedge \tau}  \dd s \, \sum_{x ,y} \big| \big((\tG^{xy,1}_{s})^2\big)_{\bm a \bm b} - \big(\tG_{s}^2\big)_{\bm a \bm b} \big|^2 \\
	&\lesssim \theta^2 \int_{0}^{t \wedge \tau} \dd s \,   \sum_{x ,y}  \left(\big| \big(\tG_s \Delta_{xy} \tG_s^{xy,1} \tG_s\big)_{\bm a \bm b} \big|^2 + \big| \big(\tG_s \tG_s \Delta_{xy} \tG_s^{xy,1}\big)_{\bm a \bm b} \big|^2\right) \\
	& \quad + \theta^4 \int_{0}^{t \wedge \tau} \dd s \,   \sum_{x ,y} \big| \big(\tG_s \Delta_{xy} \tG_s^{xy,1} \tG_s \Delta_{xy} \tG_{s}^{xy,1}\big)_{\bm a \bm b} \big|^2 \,. 
	\end{split}
\end{equation}
To control the first term on the second line of \eqref{eq:Martent2}, we use that, analogously to \eqref{eq:resexpaux}--\eqref{eq:resexp}, 
\begin{equation} \label{eq:resexp2}
\big| \big(\tG^{xy,1}_s \tG_s\big)_{y\bm b}\big| \lesssim \big| \big(\tG_s \tG_s\big)_{y\bm b}\big| + \theta \eta_s^{-1} \,. 
\end{equation}
Hence, using \eqref{eq:resexp2} and \eqref{eq:resexp} for the first and second term in the second line of \eqref{eq:Martent2}, and involving \eqref{eq:intrulepractice}, we deduce
\begin{equation} \label{eq:seclineMartent2}
	\begin{split}
\text{sec.~line of \eqref{eq:Martent2}} &\lesssim \frac{1}{Np} \int_{0}^{t \wedge \tau} \dd s \,   \sum_{x ,y}\left( \big| (\tG_s)_{\bm a x}  (\tG_s^2)_{y\bm b}\big|^2 + \frac{1}{Np} \big| (\tG_s^2)_{y\bm b}\big|^2  + \frac{1}{Np \eta_s^2 }\big| (\tG_s)_{\bm a x} \big|^2 \right) \\
&\lesssim \frac{1}{Np} \int_{0}^{t \wedge \tau} \dd s \left(\frac{1}{\eta_s^4} +  \frac{1}{p \eta_s^3}\right) \lesssim \frac{1}{N \eta_{t \wedge \tau}^3} + \frac{1}{N p \eta_{t \wedge \tau}^2} \lesssim \frac{1}{\log N} \big( \Lambda_{{\rm ent}}^{(2)}(t \wedge \tau)\big)^2\,,
	\end{split}
\end{equation}
where we additionally used that $(|\tG_s|^4)_{\bm b \bm b} \lesssim \eta_s^{-3}$ as follows by simple norm bounds, the Ward identity, and the definition of the stopping time, similarly to \eqref{eq:Gest}. 

In estimating the third line of \eqref{eq:Martent2}, we focus on one particular index constellation arising from writing out the $\Delta_{xy}$'s, similarly to \eqref{eq:R3example}. For this term we have
\begin{equation*}
	\begin{split}
\frac{1}{(Np)^2} \int_{0}^{t \wedge \tau} \dd s \, \big| (\tG_s)_{\bm ax} (\tG_s^{xy,1} \tG_s)_{yy} (\tG_s^{xy,1})_{x\bm b} \big|^2 &\lesssim \frac{1}{(Np)^2} \int_{0}^{t \wedge \tau} \dd s \,\frac{1}{\eta_s^2} \sum_{x,y} \big| (\tG_s)_{\bm a x} \big|^2 \\
&\lesssim \int_{0}^{t \wedge \tau} \dd s \,\frac{1}{N p^2\eta_s^3} \lesssim \frac{1}{Np \eta_{t \wedge \tau}^2} \lesssim \frac{\big(\Lambda_{{\rm ent}}^{(2)}(t \wedge \tau)\big)^2}{\log N}  \,. 
	\end{split}
\end{equation*}
By collecting the above estimates, we conclude that
\begin{equation*}
	\langle \Mart^{{\rm ent},2}_{\bm a \bm b}\rangle_{t \wedge \tau} \lesssim \frac{1}{\log N} \big( \Lambda_{{\rm ent}}^{(2)}(t \wedge \tau)\big)^2 \,.
\end{equation*}

\subsubsection{Estimates for $\Mart^{{\rm av},1}_{t\wedge \tau}$}
In this section, we turn to estimating the predictable quadratic variation of $\Mart^{{\rm av},1}_{t\wedge \tau}$. In order to control resolvent entries of $G_s$ (instead of $\tG_s$), we employ Lemma~\ref{lem:entforG} without further mention. Then, by application of Lemma \ref{lem:resolventexpand} (with $G$ instead of $\tG$) for $m=0$, we obtain
\begin{equation} \label{eq:Martav1}
	\begin{split}
		\langle \Mart^{{\rm av},1}\rangle_{t \wedge \tau} &\lesssim  \int_{0}^{t \wedge \tau}  \dd s \, \sum_{x ,y} \big| \langle G^{xy,1}_{s}\rangle - \langle G_{s}\rangle \big|^2 \lesssim \frac{\theta^2}{N^2} \int_{0}^{t \wedge \tau} \dd s \,   \sum_{x ,y}  \big| \big(G_s^{xy,1}G_s\big)_{xy} \big|^2  \\
		&\lesssim \frac{1}{N^3 p}\int_{0}^{t \wedge \tau} \dd s \,   \sum_{x ,y} \left(\big| (G^2_s)_{xy}\big|^2 + \frac{1}{Np\eta_s^2} \right) \lesssim \int_{0}^{t \wedge \tau} \dd s \, \left(\frac{1}{N^2 p \eta_s^3} + \frac{1}{N^2 p^2 \eta_s^2}\right) \\
		& \lesssim \frac{1}{(N \eta_{t \wedge \tau})^2} + \frac{1}{(N \eta_{t \wedge \tau}) (N p)} \lesssim \frac{1}{\log N} \big(\Lambda_{{\rm av}}^{(1)}(t\wedge \tau)\big)^2 \,. 
	\end{split}
\end{equation}
To go to the second line, we used that, analogously to \eqref{eq:resexp2}, 
\begin{equation} \label{eq:resexp22}
	\big| \big(G^{xy,1}_s G_s\big)_{xy}\big| \lesssim \big| \big(G_s^2\big)_{xy}\big| + \theta \eta_s^{-1} \,. 
\end{equation}
Next, we carried out the $x,y$-summations and bounded $\langle |G_s|^4\rangle \leq \langle \im G_s \rangle/\eta_s^3 \lesssim \eta_s^{-3}$ using norm bounds and the Ward identity. The penultimate step simply follows from \eqref{eq:intrulepractice}. 

\subsubsection{Estimates for $\Mart^{{\rm av},2}_{t\wedge \tau}$} Finally, we estimate the predictable quadratic variation of $\Mart^{{\rm av},2}_{t\wedge \tau}$. As in the previous section, in order to control resolvent entries $G_s$ (instead of $\tG_s$), we employ Lemma~\ref{lem:entforG} without further mention. Then, by application of Lemma \ref{lem:resolventexpand} (with $G$ instead of $\tG$) for $m=0$, we obtain
\begin{equation} \label{eq:Martav2}
	\begin{split}
		\langle \Mart^{{\rm av},2}\rangle_{t \wedge \tau} &\lesssim  \int_{0}^{t \wedge \tau}  \dd s \, \sum_{x ,y} \big| \langle (G^{xy,1}_{s})^2\rangle - \langle G_{s}^2\rangle \big|^2	\\ & \lesssim \frac{\theta^2}{N^2} \int_{0}^{t \wedge \tau} \dd s \,   \sum_{x ,y}  \left(\big| \big(G_s^{xy,1} G_s^2\big)_{xy} \big|^2 \right) \\
		& \quad + \frac{\theta^4}{N^2} \int_{0}^{t \wedge \tau} \dd s \,   \sum_{x ,y} \left(\big| \big( G_s^{xy,1} G_s\big)_{xy} \big|^4 + \big| \big( G_s^{xy,1} G_s\big)_{xx} \big|^2 \big| \big( G_s^{xy,1} G_s\big)_{yy} \big|^2\right)\,. 
	\end{split}
\end{equation}
To control the second line of \eqref{eq:Martav2}, we use that, analogously to \eqref{eq:resexp2} and \eqref{eq:resexp22}, 
\begin{equation*}
	\big| \big(G^{xy,1}_s G_s^2\big)_{xy}\big| \lesssim \big| \big(G_s^3\big)_{xy}\big| + \theta \eta_s^{-2} 
\end{equation*}
to deduce 
\begin{equation*}
\text{second line of \eqref{eq:Martav2}} \lesssim \frac{1}{N^3 p} \int_{0}^{t \wedge \tau}  \dd s \, \sum_{x ,y} \left(\big| \big(G_s^3\big)_{xy}\big|^2 + \frac{1}{Np \eta_s^4}\right) \lesssim \frac{1}{\log N} \big( \Lambda_{{\rm av}}^{(2)}(t \wedge \tau)\big)^2 \,, 
\end{equation*}
similarly to \eqref{eq:seclineMartent2}, additionally using that $\langle |G_s|^6\rangle \leq \langle \im G_s \rangle/\eta_s^5 \lesssim \eta_s^{-5}$. 

For the last line of \eqref{eq:Martav2}, we employ \eqref{eq:resexp22} and the analogous statement for the matrix entries $xx$ and $yy$, to get that 
\begin{equation*}
	\begin{split}
\text{third~line of \eqref{eq:Martav2}} &\lesssim \frac{1}{N^4 p^2} \int_{0}^{t \wedge \tau}  \dd s \, \sum_{x ,y} \left(\big|(G_s^2)_{xy}\big|^4 + \big|(G_s^2)_{xx}\big|^2 \, \big|(G_s^2)_{yy}\big|^2 + \frac{1}{(Np)^2 \eta_s^4}\right) \\
&\lesssim \int_{0}^{t \wedge \tau}  \dd s \, \left( \frac{1}{N^3 p^2 \eta_s^5} + \frac{1}{N^2 p^2 \eta_s^4} + \frac{1}{N^4 p^4 \eta_s^4}\right)\lesssim \frac{1}{\log N} \big( \Lambda_{{\rm av}}^{(2)}(t \wedge \tau)\big)^2 \,.
	\end{split}
\end{equation*}
Here, to go to the second line, we bounded $|(G_s^2)_{xy}| + |(G_s^2)_{xx}| \lesssim \eta_s^{-1}$ using the Cauchy-Schwarz inequality and the entrywise single resolvent bound \eqref{eq:Gest}. In the last step, we finally employed \eqref{eq:intrulepractice}. 

By collecting the above estimates, we conclude that
\begin{equation*}
	\langle \Mart^{{\rm av},2}\rangle_{t \wedge \tau} \lesssim \frac{1}{\log N} \big( \Lambda_{{\rm av}}^{(2)}(t \wedge \tau)\big)^2 \,.
\end{equation*}
We have hence controlled all four predictable quadratic variations and thus completed the proof of Lemma \ref{lem:martest}. \qed

\subsection{Proof of Lemma \ref{lem:jump}} \label{subsec:jumpproof}
The goal of this section is to give the proof of Lemma \ref{lem:jump}, i.e.~control the size of jumps of the martingales \eqref{eq:Mav}--\eqref{eq:Ment}. Note that, by continuity, the generator terms in \eqref{eq:Mav}--\eqref{eq:Ment} (the integrals) do \emph{not} contribute to the jump sizes and we thus have that, for each of the four martingales,
\begin{equation} \label{eq:jumpgen}
 \left| \Mart_s - \Mart_{s-} \right| \leq \max_{x \neq y} \left| \phi_s\big(X^{xy,1}(s)\big) - \phi_s\big(X(s)\big) \right| \quad \text{for every} \quad s \in [0, \tau] \,. 
\end{equation}

As in the previous section, we begin with the estimates for $\Mart_{\bm a \bm b,s}^{{\rm ent}, 1}$.  In this case, using \eqref{eq:jumpgen} and the shorthand  notation $\tG = G^{ab}$, we have
\begin{equation} \label{eq:ent1jump}
 \left| \Mart_{\bm a \bm b,s}^{{\rm ent}, 1} - \Mart_{\bm a \bm b,s-}^{{\rm ent}, 1} \right| \leq \max_{x \neq y}^{ab} \left| \big(\tG^{xy,1}_s\big)_{\bm a \bm b} - \big(\tG_s\big)_{\bm a \bm b}  \right|\,,
\end{equation}
where we introduced the shorthand notation $\max_{x \neq y}^{ab}$ to mean the maximum over all $x, y \in [N]$ such that $x \neq y$ and $\{x,y\} \neq \{a,b\}$ as sets. This last constraint is imposed by considering the resolvent $\tG$ of the random matrix that has its $(a,b)$ and $(b,a)$ entry set to zero. 

To control \eqref{eq:ent1jump}, we employ a resolvent expansion as in Lemma \ref{lem:resolventexpand} to order $m=1$ and Lemma \ref{lem:entforG} together with Lemma \ref{lem:Mprop}, to find that 
\begin{equation*}
	\begin{split}
\eqref{eq:ent1jump} &\lesssim \frac{1}{\sqrt{Np}} \max_{x \neq y}^{ab} \left|\big(\tG_s\big)_{\bm a x} \big(\tG_s\big)_{y \bm b}\right| + \frac{1}{Np} \\
&\lesssim \frac{1}{\sqrt{Np}} \max_{x \neq y}^{ab} \left[ \Big(|(M_s)_{\bm ax}| + \Cf \Lambda_{{\rm ent}}^{(1)}(s) \Big) \, \Big(|(M_s)_{y \bm b}| + \Cf \Lambda_{{\rm ent}}^{(1)}(s) \Big)\right] + \frac{1}{Np} \\
&\lesssim \frac{1}{\sqrt{Np}} \max_{x \neq y}^{ab}\Big[|(M_s)_{\bm ax}| |(M_s)_{y \bm b}| \Big] + \frac{1}{\sqrt{Np}}  \Cf \Lambda_{{\rm ent}}^{(1)}(s) \lesssim \frac{1}{\sqrt{Np}}  \Cf \Lambda_{{\rm ent}}^{(1)}(s) \,. 
	\end{split}
\end{equation*}
Here, in the last step, we used that, due to the constraints imposed by $\max_{x \neq y}^{ab}$ and Lemma \ref{lem:Mprop} (together with Lemma \ref{lem:BerProp}~(ii)), 
\begin{equation} \label{eq:Mtrickbound}
\max_{x \neq y}^{ab}\Big[|(M_s)_{\bm ax}| |(M_s)_{y \bm b}| \Big] \lesssim \frac{1}{\sqrt{Np}} \,. 
\end{equation}
We point out that controlling \eqref{eq:ent1jump} by $1/\sqrt{Np}$ only could have been achieved using the trivial bound where the right-hand side of \eqref{eq:Mtrickbound} is replaced by $1$. The improvement in \eqref{eq:Mtrickbound} is achieved due to the constraints imposed by $\max_{x \neq y}^{ab}$. This is the reason for considering $\tG = G^{ab}$ instead of $G$ in the entrywise law. Overall, we now have that
\begin{equation} \label{eq:jump1ent}
 \left| \Mart_{\bm a \bm b,s}^{{\rm ent}, 1} - \Mart_{\bm a \bm b,s-}^{{\rm ent}, 1} \right| \lesssim \frac{1}{\sqrt{Np}}  \Cf \Lambda_{{\rm ent}}^{(1)}(s) \quad \text{for every} \quad s \in [0, \tau] \,. 
\end{equation}

We continue by estimating $\Mart_{\bm a \bm b,s}^{{\rm ent}, 2}$.  Again, using \eqref{eq:jumpgen} and the shorthand  notation $\tG = G^{ab}$, we have 
\begin{equation} \label{eq:ent2jump}
	\left| \Mart_{\bm a \bm b,s}^{{\rm ent}, 2} - \Mart_{\bm a \bm b,s-}^{{\rm ent}, 2} \right| \leq \max_{x \neq y}^{ab} \left| \big((\tG^{xy,1}_s)^2\big)_{\bm a \bm b} - \big(\tG_s^2\big)_{\bm a \bm b}  \right|\,.
\end{equation}
In order to control \eqref{eq:ent2jump}, we will use the following lemma, whose proof is given in Appendix \ref{app:technical}. 
\begin{lemma}[Bound on $M'$] \label{lem:M2bound}
	For $z_T \in \mathbb{D}$, we have that, uniformly in $s \in [0,T]$, the derivative $M_s' = M_s'(z_s)$ from \eqref{eq:Mprime} satisfies
	\begin{equation}
		\big| (M_s')_{\bm v \bm w}\big| \leq C_\kappa \left(\big|\langle \bm v, (\mathbf{1} - \Pi_{\bm e}) \bm w \rangle\big| + \frac{\big| \langle \bm v, \Pi_{\bm e} \bm w \rangle \big|}{Np}\right)\,, \quad \text{for any} \quad \bm v, \bm w \in \R^N \,, 
	\end{equation}
	where $C_\kappa$ is a constant depending only on $\kappa$. 
\end{lemma}

Similarly to \eqref{eq:ent1jump}, we now use a resolvent expansion from Lemma \ref{lem:resolventexpand} to order $m=0$ and Lemma \ref{lem:entforG} together with Lemma \ref{lem:Mprop} to find that
\begin{equation*}
\begin{split}
\eqref{eq:ent2jump} &\lesssim \frac{1}{\sqrt{Np}} \max_{x \neq y}^{ab} \left[\left|\big(\tG_s\big)_{\bm a x} \big(\tG^{xy,1}_s\tG_s\big)_{y \bm b}\right| + \left|\big(\tG_s^2\big)_{\bm a x} \big(\tG^{xy,1}_s\big)_{y \bm b}\right| \right] + \frac{1}{Np} \frac{1}{\eta_s} \\
&\lesssim  \frac{1}{\sqrt{Np}} \max_{x \neq y}^{ab} \left|\big(M_s\big)_{\bm a x} \big(M_s'\big)_{y \bm b}\right|  + \frac{1}{\sqrt{Np}} \Cf \Lambda_{{\rm ent}}^{(2)}(s) \lesssim \frac{1}{\sqrt{Np}} \Cf \Lambda_{{\rm ent}}^{(2)}(s) \,. 
\end{split}
\end{equation*}
Here, in the last step we used that, analogously to \eqref{eq:Mtrickbound}, 
\begin{equation} \label{eq:Mtrickbound2}
\max_{x \neq y}^{ab} \left|\big(M_s\big)_{\bm a x} \big(M_s'\big)_{y \bm b}\right|  \lesssim \frac{1}{\sqrt{Np}}
\end{equation}
due to the constraints imposed by $\max_{x \neq y}^{ab} $ and with the aid of Lemma \ref{lem:M2bound}. Thus,
\begin{equation} \label{eq:jump2ent}
	\left| \Mart_{\bm a \bm b,s}^{{\rm ent}, 2} - \Mart_{\bm a \bm b,s-}^{{\rm ent}, 2} \right| \lesssim \frac{1}{\sqrt{Np}}  \Cf \Lambda_{{\rm ent}}^{(2)}(s) \quad \text{for every} \quad s \in [0, \tau] \,. 
\end{equation}

Next, we bound the jump size of $\Mart_s^{{\rm ent}, 1}$, for which we have that, by a simple resolvent expansion, i.e.~applying Lemma \ref{lem:resolventexpand} to order $m=0$, 
\begin{equation} \label{eq:av1jump}
	\begin{split}
	\left| \Mart_{s}^{{\rm av}, 1} - \Mart_{s-}^{{\rm av}, 1} \right|  &\leq \max_{x \neq y}\left| \langle G^{xy,1}_s\rangle - \langle G_s \rangle  \right| \\ &\lesssim \frac{1}{\sqrt{Np}}\frac{1}{N}\max_{x \neq y} \left| \big(G_s^{xy,1}G_s\big)_{xy} \right|  \lesssim \frac{1}{\sqrt{Np}} \frac{1}{N \eta_s} \lesssim \frac{1}{\sqrt{Np}} \frac{1}{\sqrt{\log N}} \Cf \Lambda_{{\rm av}}^{(1)}(s)\,,
	\end{split}
\end{equation}
where to go to the second line, we wrote out the average trace. Afterwards, in the penultimate step, we employed the Cauchy-Schwarz inequality together with the Ward identity and Lemma \ref{lem:entforG}. 

Finally, we bound the jump size of $\Mart_s^{{\rm ent}, 2}$, i.e.
\begin{equation}\label{eq:av2jump}
		\left| \Mart_{s}^{{\rm av}, 2} - \Mart_{s-}^{{\rm av}, 2} \right|  \leq \max_{x \neq y} \absb{ \langle (G^{xy,1}_s)^2\rangle - \langle G_s^2 \rangle} \,.
\end{equation}
Analogously to \eqref{eq:av1jump}, we have
\begin{equation*} 
	\begin{split}
	\eqref{eq:av2jump}	\lesssim \, &\frac{1}{\sqrt{Np}}\frac{1}{N}\max_{x \neq y} \left| \big(G_s^{xy,1}G_s^2\big)_{xy} \right| 
		\\
	&\qquad 	+ \frac{1}{Np} \frac{1}{N} \max_{x \neq y} \left[\left| \big(G_s^{xy,1} G_s\big)_{xx}\big(G_s^{xy,1} G_s\big)_{yy} \right| + \left| \big(G_s^{xy,1} G_s\big)_{xy}\big(G_s^{xy,1} G_s\big)_{xy} \right| \right] \\
		\lesssim \, &\frac{1}{\sqrt{Np}} \frac{1}{N \eta_s^2} \lesssim \frac{1}{\sqrt{Np}} \frac{1}{\sqrt{\log N}} \Cf \Lambda_{{\rm av}}^{(2)}(s)
	\end{split}
\end{equation*}
and hence
\begin{equation} \label{eq:jump2av}
		\left| \Mart_{s}^{{\rm av}, 2} - \Mart_{s-}^{{\rm av}, 2} \right| \lesssim \frac{1}{\sqrt{Np}} \frac{1}{\sqrt{\log N}} \Cf \Lambda_{{\rm av}}^{(2)}(s) \quad \text{for every} \quad s \in [0, \tau] \,. 
\end{equation}
Collecting the estimates \eqref{eq:jump1ent}, \eqref{eq:jump2ent}, \eqref{eq:av1jump}, and \eqref{eq:jump2av}, this concludes the proof of Lemma \ref{lem:jump}. \qed

\subsection{Proof of Lemma \ref{lem:generate}} \label{subsec:generateproof}
As in the previous subsections, we discuss the generator terms 
\begin{equation} \label{eq:generatorgeneral}
\int_{0}^{t \wedge \tau} \dd s\, \big(\Lcal \phi_s + \partial_s \phi_s\big)
\end{equation}
of the four different observables separately. We will start with the generator terms associated with $\phi_{ t}^{{\rm av}, 1}$, as it is the most illustrative one.

\subsubsection{Generator terms for $\phi_{ t}^{{\rm av}, 1}$}
For $\phi_{ t}^{{\rm av}, 1}$, the integrand in \eqref{eq:generatorgeneral} is given by
\begin{equation} \label{eq:avGen}
\sum_{x < y} \big(\big\langle G_s^{xy,1} \big\rangle  - \big\langle G_s \big\rangle\big) + \big\langle G_s \big\{\partial_s z_s\big\} G_s \big\rangle -  \langle M_s \rangle
\end{equation}
where we employed Lemma \ref{lem:BerProp}~(b) for the $M_s$-term. To control \eqref{eq:avGen} we employ a resolvent expansion for the first term in \eqref{eq:avGen}, just as the one formulated in Lemma \ref{lem:resolventexpand}, to order $m=2$, to find that
\begin{equation} \label{eq:avGenexpand}
	\begin{split}
\eqref{eq:avGen} =  & \, N^{-1}\sum_{x \neq y}  \theta( X_{xy}(s) - 1) (G_s^2)_{yx}  + N^{-1} \sum_{x,y} (G_s^2)_{yx}\big( \theta \Smat_{xy} - (z_s)_{xy}\big) - \langle M_s \rangle \\
& + N^{-1} \sum_{x \neq y} \theta^2 \big(X_{xy}(s) - 1\big)^2 (G_s^2)_{xx} (G_s)_{yy} - \langle G_s^2 \rangle \frac{1 - p_s}{1 -p}\frac{m_s}{p} \\
& + N^{-1} \sum_{x \neq y} \theta^2 \big(X_{xy}(s) - 1\big)^2 (G_s^2)_{xy} (G_s)_{xy} + \Rcal^{(3)}(s)\,,
	\end{split}
\end{equation}
where we inserted our choice of Bernoulli characteristic flow $\partial_s z_s$ from \eqref{eq:Bernflow} and denoted the third order resolvent expansion term by $\Rcal^{(3)}(s)$. 

The first line of \eqref{eq:avGenexpand} corresponds to the first order resolvent expansion terms, that, by involving the defining resolvent equation, $G_s = (\theta X(s) - z_s)^{-1}$, can easily be seen  to evaluate to $\langle G_s  - M_s\rangle$. The second line of \eqref{eq:avGenexpand} corresponds the main diagonal term of the second order resolvent expansion, which carries the main cancellation achieved by the Bernoulli characteristic flow \eqref{eq:Bernflow}.
The terms in the third line of \eqref{eq:avGenexpand} will be shown to be negligible. More precisely, our goal is to show that 
\begin{equation} \label{eq:Edef}
\eqref{eq:avGen} = \left(1 + \frac{ (1 - p_s)\langle G_s^2 \rangle}{p (1 - p)}\right) \langle G_s - M_s \rangle + \Ecal(s)
\end{equation}
with an error term $\Ecal(s) = \Ecal'(s) + \Ecal''(s)$ that is composed of errors from the second and third line of \eqref{eq:avGenexpand} and defined in \eqref{eq:errordef} and \eqref{eq:errordef2} below, respectively. We now discuss these last two lines of \eqref{eq:avGenexpand}. 

We start with the second line of \eqref{eq:avGenexpand}. Using that $m_s = \big(M_s\big)_{yy}$ for all $y \in [N]$ by vertex transitivity of the complete graph and $\big(X_{xy}(s) - 1\big)^2 = \big(1 - X_{xy}(s)\big)$ since $X_{xy}(s) = 0,1$, we have 
\begin{equation} \label{eq:1av2ndline}
	\begin{split}
\text{second line of \eqref{eq:avGenexpand}} &= \frac{\theta^2}{N} \sum_{x,y} (G_s^2)_{xx}\left[(1 - \delta_{xy}) \big(1 - X_{xy}(s)\big) (G_s)_{yy} - (1 - p_s) (M_s)_{yy}\right]  \\
&=  \frac{ (1-p_s)\langle G_s^2  \rangle}{p (1 - p)} \langle G_s - M_s \rangle+ \Ecal'(s)\,,
	\end{split}
\end{equation}
where the error term $\Ecal'(s)$ can be bounded by a sum of five contributions, 
\begin{equation}\label{eq:errordef}
|\Ecal'(s)| \lesssim \sum_{i=1}^5 \Ecal_i(s)\,,
\end{equation}
which are given by
\begin{align} 
\Ecal_1(s) &\deq \frac{\theta^2}{N} \left| \sum_{x,y} \big(G_s^2 - M'_s\big)_{xx} (X_{xy}(s) - p_s)\big(G_s - M_s\big)_{yy}\right| \label{eq:E1} \\
\Ecal_2(s) &\deq \frac{\theta^2}{N} \left| \sum_{x,y} \big(G_s^2 - M'_s\big)_{xx} (X_{xy}(s) - p_s)\big(M_s\big)_{yy}\right| \label{eq:E2} \\
\Ecal_3(s) &\deq  \frac{\theta^2}{N} \left| \sum_{x,y} \big(M'_s\big)_{xx} (X_{xy}(s)- p_s) \, \big(G_s - M_s\big)_{yy}\right| \label{eq:E3} \\
\Ecal_4(s) &\deq \frac{\theta^2}{N} \left| \sum_{x,y} \big(M'_s\big)_{xx} (X_{xy}(s)- p_s) \, \big(M_s\big)_{yy}\right| \label{eq:E4}  \\
\Ecal_5(s) &\deq  \frac{\theta^2}{N} \left| \sum_x (G_s^2)_{xx} (G_s)_{xx} \right|  \label{eq:E5}
\end{align}
and shall now be estimated separately. 

For the first error term \eqref{eq:E1}, we use Lemma \ref{lem:entforG} and the second relation in \eqref{eq:Bennettpractice}, which follows from Bennett's inequality (formulated in Lemma \ref{lem:Bennett} below -- note that \eqref{eq:Bennettpractice} holds \emph{uniformly} in Bernoulli parameters $p \in (0,1/2]$ and hence \emph{uniformly} in time due to the relation $p_s = 1- \ee^{-s}$ in our setting), and find
\begin{equation} \label{eq:E1est}
	\begin{split}
\Ecal_1(s) &\lesssim \frac{\Cf^2 \big(\Lambda_{{\rm ent}}^{(1)}(s)\big)^2}{\eta_s} \frac{1}{N^2 p}\sum_{x,y} |X_{xy}(s) - p_s| \\
&\preceq  \frac{p_s}{p} \left(1 + \frac{\log N}{N p_s}\right)\frac{\Cf^2\big(\Lambda_{{\rm ent}}^{(1)}(s)\big)^2}{\eta_s} \lesssim \frac{\Cf^2\big(\Lambda_{{\rm ent}}^{(1)}(s)\big)^2}{\eta_s} \,.
	\end{split} 
\end{equation}
For $\Ecal_2$ in \eqref{eq:E2}, we estimate
\begin{equation} \label{eq:E2est}
	\begin{split}
\Ecal_2(s) &\lesssim \frac{1}{N} \sum_x \big|(G_s^2 - M_s')_{xx} \big| \, \left[\frac{1}{Np}\bigg| \sum_y(X_{xy}(s) - p_s) \bigg|\right] \\
&\preceq  \frac{p_s}{p} \left(\sqrt{\frac{\log N}{N p_s}} + \frac{\log N}{Np_s}\right)\frac{\Cf \Lambda_{{\rm ent}}^{(1)}(s)}{\eta_s} \lesssim \frac{\Cf \big(\Lambda_{{\rm ent}}^{(1)}(s)\big)^2}{\eta_s}\,,
	\end{split}
\end{equation}
where we now used the first relation in \eqref{eq:Bennettpractice}, which is possible since $(M_s)_{yy} = \langle M_s \rangle$, Lemma \ref{lem:entforG}, \eqref{eq:Lambdaent}, and the definition of the stopping time \eqref{eq:stoptime}. For $\Ecal_3$ in \eqref{eq:E3}, we similarly find that 
\begin{equation} \label{eq:E3est}
	\Ecal_3(s) \preceq  \frac{p_s}{p} \left(\sqrt{\frac{\log N}{N p_s}} + \frac{\log N}{N p_s}\right)\Cf \Lambda^{(1)}_{{\rm ent}}(s) \lesssim \Cf \big(\Lambda_{{\rm ent}}^{(1)}(s)\big)^2\,,
\end{equation}
additionally using Lemma \ref{lem:M2bound} for the $M'$ term. 
For the error term $\Ecal_4$ in \eqref{eq:E4}, we use that $(M_s')_{xx} = \langle M_s'\rangle$ and $(M_s)_{yy} = \langle M_s\rangle$, to deduce
\begin{equation} \label{eq:E4est}
\Ecal_4(s) \lesssim \frac{1}{N^2 p}\bigg| \sum_{x,y} (X_{xy}(s) - p_s)\bigg| \preceq  \frac{p_s}{p}\left(\sqrt{\frac{\log N}{N^2 p_s}} + \frac{\log N}{N^2 p_s}\right) \lesssim \sqrt{p} \,  \Lambda_{{\rm av}}^{(1)}(s)
\end{equation}
by means of Bennett's inequality from Lemma \ref{lem:Bennett}, but now for $N(N-1)/2 \asymp N^2$ instead of $N$ independent Bernoulli random variables. Note that in \eqref{eq:E4est} we could have put any time $s$ as an argument of the $\Lambda_{{\rm av}}(s)$ control parameter, since we anyway involved only the time-independent parts of $\Lambda_{{\rm av}}^{(1)}$ from \eqref{eq:Lambdaav} in the ultimate estimate in \eqref{eq:E4est}. Finally, we estimate the error term $\Ecal_5$ in \eqref{eq:E5} by adding and subtracting $(M_s')_{xx}$ and $(M_s)_{xx}$ as
\begin{equation} \label{eq:E5est}
\begin{split}
\Ecal_5(s) \lesssim \frac{1}{Np} \left(1 + \Cf \Lambda_{{\rm av}}^{(1)}(s) + \Cf \Lambda_{{\rm av}}^{(2)}(s) + \frac{\Cf^2 \big(\Lambda_{{\rm ent}}^{(1)}(s)\big)^2}{\eta_s}\right) \,. 
\end{split}
\end{equation}

We now turn to controlling the third line of \eqref{eq:avGenexpand}, which is cast in the error term $\Ecal''$ defined as
\begin{equation} \label{eq:errordef2}
\Ecal''(s) \deq N^{-1} \sum_{x \neq y} \theta^2 \big(X_{xy}(s) - 1\big)^2 (G_s^2)_{xy} (G_s)_{xy} + \Rcal^{(3)}(s) \,. 
\end{equation}
The first term in \eqref{eq:errordef2} can be bounded as
\begin{equation} \label{eq:R2Diagav}
\begin{split}
\frac{1}{N^2 p} \sum_{x,y} \big|(G_s^2)_{xy} (G_s)_{xy}\big| &\lesssim \frac{1}{Np} \langle |G_s|^4\rangle^{1/2} \langle |G_s|^2 \rangle^{1/2} \\
&\leq \frac{1}{Np \eta_s^2} \langle \im G_s \rangle \lesssim \frac{1}{Np \eta_s^2} \left(1 + \Cf \Lambda_{{\rm av}}^{(1)}(s)\right) \lesssim \frac{1}{Np \eta_s^2} \,. 
\end{split}
\end{equation}
For the third order resolvent expansion term $\Rcal^{(3)}(s)$, similarly to \eqref{eq:R3example}, we focus on one exemplary index constellation arising from writing out the $\Delta_{xy}$ in the expansion. For this term, we have 
\begin{equation} \label{eq:R3avest}
	\begin{split}
&\frac{1}{N^{5/2} p^{3/2}} \left| \sum_{x,y} \big(G_s^{xy,1} G_s\big)_{xx} (G_s)_{yy} (G_s)_{xy} \right| \\
 \lesssim \; &\frac{1}{Np^{3/2}} \sqrt{\frac{1}{N} \sum_{x,y} |(G_s)_{xy}|^2} \Big(1  + \Cf \Lambda_{{\rm ent}}^{(2)}(s)\Big) \Big(1 + \Cf \Lambda_{{\rm ent}}^{(1)}(s)\Big)   \\
 \lesssim \; & \frac{1}{p \eta_s}\frac{1}{\sqrt{(N \eta_s) (Np)}} \left(1 + \Cf \Lambda_{{\rm av}}^{(1)}(s)\right) \lesssim \frac{1}{N p^{3/2}\eta_s^{3/2} } \,, 
	\end{split}
\end{equation}
as follows by writing the diagonal $G$-factors, schematically, as $M + (G-M)$, involving the definition of the stopping time \eqref{eq:stoptime}, additionally using Lemma \ref{lem:entforG}, and applying the Cauchy-Schwarz inequality together with the Ward identity. Other index constellations can be handled analogously.

Collecting all the above error estimates using \eqref{eq:intrulepractice} and the definitions \eqref{eq:Lambdaav}--\eqref{eq:Lambdaent}, we find that 
\begin{equation}
\int_{0}^{t \wedge \tau} \dd s \, |\Ecal(s)| \leq \Lambda_{{\rm av}}^{(1)}(t \wedge \tau)
\end{equation}
with very high probability.
Here, we additionally used that $p \leq 1/(\log N)^2$ by assumption, which allows to handle the error contributions \eqref{eq:E1est}, \eqref{eq:E2est}, \eqref{eq:E3est}, and \eqref{eq:E4est}.

\subsubsection{Generator terms for $\phi_{ t}^{{\rm av}, 2}$}
For $\phi_{ t}^{{\rm av}, 2}$, the integrand in \eqref{eq:generatorgeneral} is given by
\begin{equation} \label{eq:avGen2}
	\sum_{x < y} \big(\big\langle (G_s^{xy,1})^2 \big\rangle  - \big\langle G_s^2 \big\rangle\big) + 2\big\langle G_s^2 \big\{\partial_s z_s\big\} G_s \big\rangle -  2\langle M_s' \rangle - \frac{1 - p_s}{1 - p} \frac{\langle M_s'\rangle^2}{p}\,,
\end{equation}
where we employed the following lemma, whose proof is given in Appendix \ref{app:technical}, for the $M_s'$-term(s).
\begin{lemma}[Time evolution of $M'$] \label{lem:Mt'}
	We have
	\begin{equation}
		\partial_s  M_s' = \left(2  + \frac{1 - p_s}{1 - p} \frac{\langle M_s'\rangle }{p}\right) M_s'\,.
	\end{equation}
\end{lemma}

 To control \eqref{eq:avGen2} we use a resolvent expansion for the first term in \eqref{eq:avGen2}, just as the one formulated in Lemma \ref{lem:resolventexpand}, to order $m=1$, to find that
\begin{equation} \label{eq:avGenexpand2}
	\begin{split}
		\eqref{eq:avGen2} =  & \; 2N^{-1}\sum_{x \neq y}  \theta( X_{xy}(s) - 1) (G_s^3)_{yx}  + 2N^{-1} \sum_{x,y} (G_s^3)_{yx}\big( \theta \Smat_{xy} - (z_s)_{xy}\big) - 2\langle M_s' \rangle \\
		& +N^{-1} \sum_{x \neq y} \theta^2 \big(X_{xy}(s) - 1\big)^2 (G_s^2)_{xx} (G_s^2)_{yy} - \frac{1 - p_s}{1 - p} \frac{\langle M_s'\rangle^2}{p}\\
		& + 2N^{-1} \sum_{x \neq y} \theta^2 \big(X_{xy}(s) - 1\big)^2 (G_s^3)_{xx} (G_s)_{yy} - 2\langle G_s^3 \rangle \frac{1 - p_s}{1 -p}\frac{\langle M_s\rangle}{p} \\
		& + N^{-1} \sum_{x \neq y} \theta^2 \big(X_{xy}(s) - 1\big)^2 \left[2 (G_s^3)_{xy} (G_s)_{xy} + (G_s^2)_{xy} (G_s^2)_{xy}\right] + \Rcal^{(\geq 3)}(s)\,,
	\end{split}
\end{equation}
where we inserted our choice of Bernoulli characteristic flow $\partial_s z_s$ from \eqref{eq:Bernflow} and collected all terms containing at least a factor $\theta^3$ from the resolvent expansion in $\Rcal^{(\geq 3)}(s)$. 

The first two lines of \eqref{eq:avGenexpand2} contribute the integrand in \eqref{eq:avGron2}, while the last three lines will be error terms. Indeed, to start, similarly to \eqref{eq:avGenexpand}, the first line of \eqref{eq:avGenexpand2} evaluates to 
\begin{equation}
\text{first line of \eqref{eq:avGenexpand2}} = 2\langle G_s^2 - M_s'\rangle \,. 
\end{equation}
The second line of \eqref{eq:avGenexpand2} can be written as 
\begin{equation} \label{eq:avGenexpand22ndline}
	\begin{split}
\text{second line of \eqref{eq:avGenexpand2}} &= \frac{\theta^2}{N}\sum_{x,y} \left[(1 - \delta_{xy}) \big(1 - X_{xy}(s)\big) (G_s^2)_{xx}(G_s^2)_{yy} - (1 - p_s) (M_s')_{xx} (M_s')_{yy}\right]  \\
&= \frac{1 - p_s }{p(1 - p)} \left(2 \langle M_s'\rangle + \langle G_s^2 - M_s'\rangle \right) \langle G_s^2 - M_s'\rangle + \Ecal^{\rm 2nd}(s)
	\end{split}
\end{equation}
for an appropriate error term $\Ecal^{\rm 2nd}(s)$ such that this holds. 

In fact, writing $\Ecal^{\rm 3rd}(s)$ and $\Ecal^{\rm 4th}(s)$ for the third and fourth line of \eqref{eq:avGenexpand2}, it remains to show the following lemma. 
\begin{lemma} \label{lem:av2errors}
Using the above notation, we have
\begin{equation}
\int_{0}^{t \wedge \tau} \dd s \, \left( |\Ecal^{\rm 2nd}(s)| + |\Ecal^{\rm 3rd}(s)| + |\Ecal^{\rm 4th}(s)|\right) \leq \Lambda_{{\rm av}}^{(2)}(t \wedge \tau)
\end{equation}
with very high probability.
\end{lemma}
\begin{proof}
We study each of the three error terms separately. The first one, $\Ecal^{\rm 2nd}(s)$, similarly to \eqref{eq:errordef}, by writing $G^2 = M' + (G^2 - M')$ and $X = p + (X - p)$, can be bounded by a sum of five contributions, 
\begin{equation}\label{eq:errordefav2}
	|\Ecal^{\rm 2nd}(s)| \lesssim \sum_{i=1}^5 \Ecal_i(s)\,,
\end{equation}
which are given by
\begin{align} 
	\Ecal_1(s) &\deq \frac{\theta^2}{N} \left| \sum_{x,y} \big(G_s^2 - M'_s\big)_{xx} (X_{xy}(s) - p_s)\big(G_s^2 - M_s'\big)_{yy}\right| \label{eq:E1av2}\,, \\
	\Ecal_2(s) &\deq \frac{\theta^2}{N} \left| \sum_{x,y} \big(G_s^2 - M'_s\big)_{xx} (X_{xy}(s) - p_s)\big(M_s'\big)_{yy}\right| \label{eq:E2av2}\,, \\
	\Ecal_3(s) &\deq  \frac{\theta^2}{N} \left| \sum_{x,y} \big(M'_s\big)_{xx} (X_{xy}(s)- p_s) \, \big(G_s^2 - M_s'\big)_{yy}\right| \label{eq:E3av2}\,, \\
	\Ecal_4(s) &\deq \frac{\theta^2}{N} \left| \sum_{x,y} \big(M'_s\big)_{xx} (X_{xy}(s)- p_s) \, \big(M_s'\big)_{yy}\right| \label{eq:E4av2}\,,  \\
	\Ecal_5(s) &\deq  \frac{\theta^2}{N} \left| \sum_x (G_s^2)_{xx} (G_s^2)_{xx} \right| \,.  \label{eq:E5av2}
\end{align}
All these five terms can be controlled completely analogously to the corresponding terms in \eqref{eq:E1}--\eqref{eq:E5}; cf.~their estimates in \eqref{eq:E1est}, \eqref{eq:E2est}, \eqref{eq:E3est}, \eqref{eq:E4est}, and \eqref{eq:E5est}, respectively. This yields
\begin{equation} \label{eq:E2nd}
	\int_{0}^{t \wedge \tau} \dd s \, |\Ecal^{\rm 2nd}(s)|  \leq \Lambda_{{\rm av}}^{(2)}(t \wedge \tau)/3
\end{equation}
with very high probability.

For the second one, $\Ecal^{\rm 3rd}(s)$, we perform a similar decomposition as for the second line of \eqref{eq:avGenexpand2} in \eqref{eq:avGenexpand22ndline} and find similar error terms as in \eqref{eq:E1av2}--\eqref{eq:E5av2} above, the main difference being that now the $G^3$-terms will only be estimated by their \emph{naive} size. That is, as we do not track a local law for $G^3$ throughout the argument, i.e.~instead of writing $G^3 = M'' + (G^3 - M'')$, we just bound $|\langle G^3\rangle| \lesssim \langle \im G \rangle/\eta^2 \lesssim \eta^{-2}$ and similarly for $(G^3)_{xx}$. We remark that this is the reason for the weaker average local law $G^2$ compared to the one for $G$; cf.~the two control parameters in \eqref{eq:Lambdaav}. However, \emph{a posteriori}, the $\langle G^2\rangle$ local law can be obtained from the $\langle G \rangle$ local law by contour integration as in \eqref{eq:contour}.

 In this way, we find
\begin{equation} \label{eq:E3rd}
\begin{split}
\big| \Ecal^{\rm 3rd}(s) \big| &\lesssim \frac{\theta^2}{N} \left|\sum_{x ,y} (G_s^3)_{xx} \left[ (1 - \delta_{xy}) \big(1 - X_{xy}(s)\big) (G_s)_{yy} - (1 - p_s) (M_s)_{yy}\right]\right| \\
&\lesssim \frac{|\langle G_s^3\rangle \langle G_s - M_s \rangle|}{p} + \frac{1}{N^2 p} \left|\sum_{x,y}  (G_s^3)_{xx} \big(X_{xy}(s) - p_s\big) (M_s)_{yy}\right|  \\
& \quad + \frac{1}{N^2 p} \left|\sum_{x,y}  (G_s^3)_{xx} \big(X_{xy}(s) - p_s\big) (G_s - M_s)_{yy}\right|+ \frac{1}{N^2 p} \left|\sum_x (G_s^3)_{xx} (G_s)_{xx}\right| \\
&\preceq \frac{\Cf \Lambda_{{\rm av}}^{(1)}(s)}{p \eta_s^2} + \frac{1}{\eta_s^2} \frac{p_s}{p} \left(\sqrt{\frac{\log N}{N p_s}} + \frac{\log N}{N p_s}\right) + \frac{1}{\eta_s^2} \frac{p_s}{p} \left(1 + \frac{\log N}{N p_s}\right) \Cf \Lambda_{{\rm ent}}^{(1)}(s) + \frac{1}{N p \eta_s^2} \,. 
\end{split}
\end{equation}
By integrating \eqref{eq:E3rd} in time, with the aid of \eqref{eq:intrulepractice}, and using $ p \leq 1/\log N$ together with $\eta_{t \wedge \tau} \leq 1$, we obtain
\begin{equation} \label{eq:E3rdfinal}
 	\int_{0}^{t \wedge \tau} \dd s \, |\Ecal^{\rm 3rd}(s)|  \leq \Lambda_{{\rm av}}^{(2)}(t \wedge \tau)/3
\end{equation}
with very high probability.

Finally, we turn to the last error term $\Ecal^{\rm 4th}(s)$. For the first term in the last line of \eqref{eq:avGenexpand2} we employ the Cauchy-Schwarz inequality to find it to be bounded by (a constant times)
\begin{multline*}
\frac{1}{N^2 p} \sum_{x,y} \big|(G_s^2)_{xy}\big|^2 + \frac{1}{N^2 p} \left(\sum_{x ,y}\big|(G_s)_{xy}\big|^2\right)^{1/2}\left(\sum_{x ,y}\big|(G_s^3)_{xy}\big|^2\right)^{1/2} \\
\lesssim  \frac{1}{N p} \left(\langle |G_s|^4\rangle + \langle |G_s|^2 \rangle^{1/2} \langle |G_s|^6\rangle^{1/2} \right) \lesssim \frac{1}{N p \eta_s^3} \left(1 + \Cf \Lambda_{{\rm av}}^{(1)}(s)\right)\lesssim \frac{1}{N p \eta_s^3} \,. 
\end{multline*}
The second term (i.e.~all higher order resolvent expansion terms) can be controlled similarly to \eqref{eq:R3avest} and we hence omit their detailed discussion. Overall, we find that
\begin{equation} \label{eq:E4thfinal}
	\int_{0}^{t \wedge \tau} \dd s \, |\Ecal^{\rm 4th}(s)|  \leq \Lambda_{{\rm av}}^{(2)}(t \wedge \tau)/3
\end{equation}
with very high probability.
Hence, combining \eqref{eq:E2nd}, \eqref{eq:E3rdfinal}, and \eqref{eq:E4thfinal}, we conclude the proof of Lemma \ref{lem:av2errors}. 
\end{proof}

\subsubsection{Generator terms for \texorpdfstring{$\phi_{ \bm a \bm b, t}^{{\rm ent}, 1}$}{$\phi_{a b, t}^{{\rm ent}, 1}$}} \label{subsubsec:genent1} We adopt the notation from Section \ref{subsubsec:Martent1estimates}, i.e.~write $\tG$ instead of $G^{ab}$, since $a,b$ remain fixed throughout the entire section. 

For $\phi_{\bm a \bm b, t}^{{\rm ent}, 1}$, the integrand in \eqref{eq:generatorgeneral} is given by
\begin{equation} \label{eq:entGen}
	\sum_{x < y} \big(\big(\tG_s^{xy,1} \big)_{\bm a \bm b}  - \big( \tG_s \big)_{\bm a \bm b}\big) + \big( \tG_s \big\{\partial_s z_s\big\} \tG_s \big)_{\bm a \bm b} -  (M_s)_{\bm a \bm b}
\end{equation}
where we employed Lemma \ref{lem:BerProp}~(b) for the $M_s$-term. As in the average case, to control \eqref{eq:entGen} we employ a resolvent expansion for the first term in \eqref{eq:entGen}, just as the one formulated in Lemma \ref{lem:resolventexpand}, to order $m=2$, to find that 
\begin{equation} \label{eq:entGenexpand}
	\begin{split}
		\eqref{eq:entGen} =  & \, \sum_{x \neq y}  \theta( X_{xy}(s) - 1) (\tG_s)_{\bm a x} (\tG_s)_{y\bm b} +  \sum_{x,y} (\tG_s)_{\bm a x}\big( \theta \Smat_{xy} - (z_s)_{xy}\big) (\tG_s)_{y\bm b}- (M_s)_{\bm a \bm b}  \\
		& + \sum_{x \neq y} \theta^2 \big(X_{xy}(s) - 1\big)^2 (\tG_s)_{\bm a x} (\tG_s)_{yy} (\tG_s)_{x\bm b} -  (\tG_s^2)_{\bm a \bm b}  \frac{1 - p_s}{1 -p}\frac{m_s}{p} \\
		& + \sum_{x \neq y} \theta^2 \big(X_{xy}(s) - 1\big)^2 (\tG_s)_{\bm a y} (\tG_s)_{xy} (\tG_s)_{x\bm b} + \Rcal^{(3)}(s)\,,
	\end{split}
\end{equation}
where we inserted our choice of Bernoulli characteristic flow $\partial_s z_s$ from \eqref{eq:Bernflow} and denoted the third order resolvent expansion term by $\Rcal^{(3)}(s)$. 

Just as in the previous sections, the first line of \eqref{eq:entGenexpand} evaluates to $\big(\tG_s - M_s\big)_{\bm a \bm b}$ and we turn to the second and third line of \eqref{eq:entGenexpand}, which shall be shown to be negligible. Indeed, similarly to \eqref{eq:1av2ndline}, we write 
\begin{equation} \label{eq:1ent2ndline}
	\begin{split}
		\text{second line of \eqref{eq:entGenexpand}} = \theta^2 \sum_{x,y} (\tG_s)_{\bm a x} (\tG_s)_{x\bm b}\left[(1 - \delta_{xy}) \big(1 - X_{xy}(s)\big) (\tG_s)_{yy} - (1 - p_s) (M_s)_{yy}\right] 
	\end{split}
\end{equation}
and bound \eqref{eq:1ent2ndline} by a sum of four contributions, 
\begin{equation}\label{eq:errordefent1}
	|\eqref{eq:1ent2ndline}| \lesssim \sum_{i=1}^4 \Ecal_i(s)\,,
\end{equation}
which are given by
\begin{align} 
	\Ecal_1(s) &\deq \frac{1}{Np} \left| \sum_{x,y} (\tG_s)_{\bm ax} (\tG_s)_{x\bm b} \big(\tG_s - M_s\big)_{yy}\right| \label{eq:E1ent1} \\
	\Ecal_2(s) &\deq \frac{1}{Np}\left| \sum_{x,y} (\tG_s)_{\bm a x} (\tG_s)_{x\bm b} (X_{xy}(s) - p_s)\big(M_s\big)_{yy}\right| \label{eq:E2ent1} \\
	\Ecal_3(s) &\deq  \frac{1}{Np} \left| \sum_{x,y} (\tG_s)_{\bm a x} (\tG_s)_{x\bm b} (X_{xy}(s)- p_s) \, \big(\tG_s - M_s\big)_{yy}\right| \label{eq:E3ent1} \\
	\Ecal_4(s) &\deq  \frac{1}{Np} \left| \sum_x (\tG_s)_{\bm a x} (\tG_s)_{x\bm b} (\tG_s)_{xx} \right| \,.  \label{eq:E4ent1}
\end{align}

We now control each of these four contributions now separately. By the Cauchy-Schwarz inequality together with the Ward identity and Lemma \ref{lem:entforG}, the first one can be bounded as
\begin{equation} \label{eq:E1estent}
\Ecal_1(s) \lesssim \frac{1}{p \eta_s}  \Cf \Lambda_{{\rm av}}^{(1)}(s) \,. 
\end{equation}
For the second contribution, using \eqref{eq:Bennettpractice}, we estimate 
\begin{equation} \label{eq:E2estent}
\Ecal_2(s) \lesssim \sum_x \big| (\tG_s)_{\bm a x} (\tG_s)_{x\bm b}\big| \frac{1}{Np} \left|\sum_y (X_{xy}(s) - p_s)\right| \preceq \frac{1}{\eta_s} \frac{p_s}{p} \left(\sqrt{\frac{\log N}{N p_s}} + \frac{\log N}{N p_s}\right) \lesssim \frac{1}{\eta_s} \Lambda_{{\rm ent}}^{(1)}(s) \,. 
\end{equation}
The third contribution, using again \eqref{eq:Bennettpractice}, can be bounded as
\begin{equation} \label{eq:E3estent}
	\begin{split}
\Ecal_3(s) &\lesssim \Cf \Lambda_{{\rm ent}}^{(1)}(s) \sum_x \big| (\tG_s)_{\bm a x} (\tG_s)_{x\bm b}\big| \frac{1}{Np}\sum_y \big|(X_{xy}(s) - p_s)\big| \\
&\preceq \Cf \frac{\Lambda_{{\rm ent}}^{(1)}(s)}{\eta_s}  \frac{p_s}{p} \left(1 + \frac{\log N}{N p_s}\right) \lesssim \frac{1}{\eta_s} \Cf \Lambda_{{\rm ent}}^{(1)}(s) \,. 
	\end{split}
\end{equation}
Finally, the fourth contribution admits the bound
\begin{equation} \label{eq:E4estent}
\Ecal_4(s) \lesssim \frac{1}{Np} \sum_x \big| (\tG_s)_{\bm a x} (\tG_s)_{x\bm b}\big|  \lesssim \frac{1}{Np \eta_s} \,. 
\end{equation}

Then, combining the bounds \eqref{eq:E1estent}, \eqref{eq:E2estent}, \eqref{eq:E3estent}, and \eqref{eq:E4estent}, we conclude that
\begin{equation}
\int_{0}^{t \wedge \tau} \dd s \, \sum_{i=1}^{4} \Ecal_i(s) \leq \Lambda_{{\rm ent}}^{(1)}(s)
\end{equation}
with very high probability,
where we additionally used \eqref{eq:intrulepractice} and the fact that, by assumption, $p \leq 1/(\log N)^2$. 

We now estimate the last line of \eqref{eq:entGenexpand}. By the Cauchy-Schwarz inequality and three Ward identities, the first term admits the bound
\begin{equation} \label{eq:2ndOD}
	\begin{split}
\hspace{-3mm}\frac{1}{Np}\sum_{x , y}  \big|  (\tG_s)_{\bm a y} (\tG_s)_{xy} (\tG_s)_{x\bm b} \big| &\lesssim \frac{1}{Np} \left(\sum_{x,y} \big| (\tG_s)_{xy}\big|^2\right)^{1/2} \left(\sum_x \big|(\tG_s)_{x\bm b}\big|^2\right)^{1/2}   \left(\sum_y \big|(\tG_s)_{y\bm a}\big|^2\right)^{1/2} \\
&\lesssim \frac{1}{N^{1/2} p} \langle |\tG_s|^2\rangle^{1/2} \big((|\tG_s|^2)_{\bm b \bm b}\big)^{1/2} \big((|\tG_s|^2)_{\bm a \bm a}\big)^{1/2} \lesssim \frac{1}{N^{1/2} p\eta_s^{3/2}} \,. 
	\end{split}
\end{equation}
For the third order resolvent expansion term $\Rcal^{(3)}(s)$, similarly to \eqref{eq:R3avest}, we focus on one particular index constellation arising from writing out the $\Delta_{xy}$'s in the expansion, namely
\begin{equation} \label{eq:3rdorder}
	\begin{split}
\frac{1}{(Np)^{3/2}} \left| \sum_{x\neq y} (1 - X_{xy}(s))(\tG_s)_{\bm ax} (\tG_s)_{yy} (\tG_s)_{xx} (\tG^{xy,1}_s)_{y\bm b}\right|\,.
	\end{split}
\end{equation}
All other terms without two diagonal terms $G_{yy} G_{xx}$ created can be handled in a much simpler way and will hence be omitted. Now, for \eqref{eq:3rdorder}, we employ a resolvent expansion, similarly to \eqref{eq:resexpaux}, of the form 
\begin{equation} \label{eq:complicatedresexp}
	\begin{split}
\tG_{y\bm b}^{xy,1} = \tG_{y\bm b} &- \theta (1 - X_{xy}) \big(\tG_{yx}  \tG_{y\bm b}^{xy,1} + \tG_{yy} \tG_{x\bm b}\big) \\
 &+ \theta^2 (1 - X_{xy}) \big( \tG_{yy}\tG_{xx} \tG^{xy,1}_{y \bm b} + \tG_{yy}\tG_{xy} \tG^{xy,1}_{x \bm b} + \tG_{yx}\tG_{yx} \tG^{xy,1}_{y \bm b} + \tG_{yx}\tG_{yy} \tG^{xy,1}_{x \bm b}\big) \,.
	\end{split}
\end{equation}
Moreover, we separate the two contributions arising from $1 - X_{xy}(s)$. For the term stemming from $1$,  we insert \eqref{eq:complicatedresexp} and discuss the terms order by order in $\theta$. 

For the leading term with $\theta^0$, we add and subtract the respective $M$-terms for the diagonal resolvent entries. First, for the $M \times M$-term, we have
\begin{equation}  \label{eq:lambda0MM}
	\begin{split}
		\frac{1}{(Np)^{3/2}} \left| \sum_{x\neq y} (\tG_s)_{\bm ax} (M_s)_{yy} (M_s)_{xx} (\tG_s)_{y\bm b}\right| &\lesssim \frac{1}{N^{1/2}p^{3/2}} \big|(\tG_s)_{\bm a \bm e} (\tG_s)_{\bm e \bm b}\big| + \frac{1}{(Np)^{3/2}} \big|(\tG_s^2)_{\bm a \bm b}\big| \\
		& \lesssim \frac{1}{N^{1/2}p^{3/2}} \left(1 + \Cf \Lambda_{{\rm ent}}^{(1)}(s)\right) \left(1 + \frac{1}{N \eta_s}\right)\,,
	\end{split}
\end{equation}
where we lifted the restriction in the summation and, using $(M_s)_{xx} = \langle M_s \rangle$, employed \emph{isotropic resummation} in the form $\sum_x A_{\bm a x} = N^{1/2} A_{\bm a \bm e}$. Additionally, we used the Cauchy-Schwarz inequality together with the Ward identity. Next, we consider an $M \times (G-M)$-type term, which we bound as
\begin{equation} \label{eq:Lambda0MG-M}
\begin{split}
		&\frac{1}{(Np)^{3/2}} \left| \sum_{x\neq y} (\tG_s)_{\bm ax} (\tG - M_s)_{yy} (M_s)_{xx} (\tG_s)_{y\bm b}\right| \\
		\lesssim &\; \frac{\Cf \Lambda_{{\rm ent}}^{(1)}(s)}{(N\eta_s)^{1/2}p^{3/2}} \left(\big|(\tG_s)_{\bm a \bm e} \big(\im \tG_s\big)_{\bm b \bm b}^{1/2}\big|  + \frac{1}{N  \eta_s^{1/2}} \big|(\im \tG_s)_{\bm a\bm a} (\im \tG_s)_{\bm b\bm b}\big|^{1/2}\right) \\
		 \lesssim & \; \Cf \Lambda_{{\rm ent}}^{(1)}(s) \left(\frac{1}{(N\eta_s)^{1/2}p^{3/2}} + \frac{1}{(Np)^{3/2} \eta_s}\right) \,. 
\end{split}
\end{equation}
 Finally, for the $(G-M)\times (G-M)$-type term, we have the bound
\begin{equation} \label{eq:lambda0G-M2}
\begin{split}
	\frac{1}{(Np)^{3/2}} \left| \sum_{x\neq y} (\tG_s)_{\bm ax} (\tG - M_s)_{yy} (\tG - M_s)_{xx} (\tG_s)_{y\bm b}\right| 
\lesssim \frac{\big(\Cf \Lambda_{{\rm ent}}^{(1)}(s)\big)^2}{N^{1/2}\eta_s p^{3/2}}  \,. 
\end{split}
\end{equation}

Next, we turn to terms arising from the resolvent expansion \eqref{eq:complicatedresexp} of order $\theta$, which, by means of the Cauchy-Schwarz inequality and the Ward identity, can be bounded as
\begin{equation} \label{eq:lambda1}
\begin{split}
\frac{1}{(Np)^{2}}  \sum_{x, y} \left(\left|(\tG_s)_{\bm ax}  (\tG_s)_{yx} \right| + \left|(\tG_s)_{\bm ax}  (\tG_s)_{x\bm b}\right|\right) \lesssim \frac{1}{N \eta_s p^2} \,. 
\end{split}
\end{equation}
Lastly, we control terms of order $\theta^2$ from \eqref{eq:complicatedresexp} as
\begin{equation} \label{eq:lambda2}
\begin{split}
\frac{1}{(Np)^{5/2}} \sum_{x,y}\left(\left|(\tG_s)_{\bm ax}  (\tG_s^{xy,1})_{y\bm b} \right| + \left|(\tG_s)_{\bm ax}  (\tG_s)_{xy}\right|+ \left|(\tG_s)_{\bm ax}  (\tG_s)_{yx}\right|\right) \lesssim \frac{1}{N^{3/2} \eta_s p^{5/2}}\,,
\end{split}
\end{equation}
where we used the Cauchy-Schwarz inequality, the Ward identity, and additionally, as a consequence of \eqref{eq:complicatedresexp}, that
\begin{equation} \label{eq:xyupexpand}
\big|\tG_{y \bm b}^{xy,1}\big| \lesssim \big|\tG_{y \bm b}\big| +\theta  \big|\tG_{x \bm b}\big| + \theta^2 \big|\tG_{xy}\big|  + \theta^2 \big|\tG_{yx}\big|\,. 
\end{equation}

Now, we turn to terms stemming from $X_{xy}$ in \eqref{eq:3rdorder}. We pull absolute values inside the summation and employ \eqref{eq:xyupexpand} (even disregarding the $\theta$ factors) to control $(\tG^{xy,1}_s)_{y \bm b}$. In this way, we find the $X_{xy}$-contribution to be bounded by
\begin{multline} \label{eq:Xxy}
\frac{1}{(Np)^{3/2}} \sum_{x,y} X_{xy}(s)\left(\left| (\tG_s)_{\bm ax} (\tG_s)_{y \bm b} \right|  + \left| (\tG_s)_{\bm ax} (\tG_s)_{x \bm b} \right| + \left| (\tG_s)_{\bm ax} (\tG_s)_{xy} \right| \right) \\
\lesssim \frac{1}{N^{1/2}\eta_s p} \sqrt{\frac{1}{N^2p} \sum_{x,y} X_{xy}(s)} + \frac{1}{(Np)^{1/2}} \sum_x \left| (\tG_s)_{\bm ax} (\tG_s)_{x \bm b} \right| \frac{1}{Np} \sum_{y} X_{xy}(s)  \preceq \frac{1}{N^{1/2}\eta_s p} \,. 
\end{multline}
To go to the second line, we employed the Cauchy-Schwarz inequality and the Ward identity for the first and third term. In the last step, we then used \eqref{eq:Bennettpractice} and, again, the Cauchy-Schwarz inequality and the Ward identity for the second term. 

Finally, combining \eqref{eq:2ndOD} with \eqref{eq:3rdorder} and \eqref{eq:lambda0MM}, \eqref{eq:Lambda0MG-M}, \eqref{eq:lambda0G-M2}, \eqref{eq:lambda1}, \eqref{eq:lambda2}, and \eqref{eq:Xxy}, we conclude that 
\begin{equation}
\int_{0}^{t \wedge \tau} \dd s \, [\text{last line of \eqref{eq:entGenexpand}} ] \leq \Lambda_{{\rm ent}}(t \wedge \tau)
\end{equation}
with very high probability,
by means of \eqref{eq:intrulepractice} and additionally using that $(\log N)^2 \lesssim Np \leq N (\log N)^{-2}$. Additionally, in order to integrate the first term on the right-hand side of \eqref{eq:Lambda0MG-M} in time, we use that, as a consequence of \eqref{eq:Bernflow}, 
\begin{equation*}
\int_{0}^{t} \frac{1}{\eta_s^{1/2}p} \, \dd s \lesssim -\int_{\eta_0}^{\eta_t} \frac{\dd \eta}{\eta^{1/2}} \lesssim \sqrt{\eta_0} - \sqrt{\eta_t} \lesssim 1\,,
\end{equation*}
where the last step follows from $\eta_0 \lesssim  1$ as a consequence of \eqref{eq:etaexplicit} together with $\eta_T \leq 1$ and $T \leq 1$.

This concludes the treatment of the generator terms for $\phi_{\bm a \bm b, t}^{\rm ent, 1}$. 

\subsubsection{Generator terms for \texorpdfstring{$\phi_{\bm a \bm b, t}^{{\rm ent}, 2}$}{$\phi_{a b, t}^{{\rm ent}, 2}$}}
For $\phi_{\bm a \bm b, t}^{{\rm ent}, 2}$, the integrand in \eqref{eq:generatorgeneral} is given by
\begin{equation} \label{eq:entGen2}
	\sum_{x < y} \big(\big((\tG_s^{xy,1})^2 \big)_{\bm a \bm b}  - \big( \tG_s^2 \big)_{\bm a \bm b}\big) + \big( \tG_s \big\{\partial_s z_s\big\} \tG_s^2 \big)_{\bm a \bm b} + \big( \tG_s^2 \big\{\partial_s z_s\big\} \tG_s \big)_{\bm a \bm b}-  2(M'_s)_{\bm a \bm b} - \frac{1 - p_s}{1 - p} \frac{\langle M_s'\rangle}{p} (M_s')_{\bm a \bm b}\,,
\end{equation}
where we employed Lemma \ref{lem:Mt'} for the $M_s$-term. As in the average case \eqref{eq:avGen2}, to control \eqref{eq:entGen2} we employ a resolvent expansion for the first term in \eqref{eq:entGen2}, just as the one formulated in Lemma \ref{lem:resolventexpand}, to order $m=1$, to find that 		\eqref{eq:entGen2} equals
\begin{equation} \label{eq:entGenexpand2}
	\begin{split}
 &  \sum_{x \neq y}  \theta( X_{xy}(s) - 1) \big((\tG_s^2)_{\bm a x} (\tG_s)_{y\bm b} + (\tG_s)_{\bm a x} (\tG_s^2)_{y\bm b} \big)  \\
		+ &  \sum_{x,y} \big((\tG_s^2)_{\bm a x}\big( \theta \Smat_{xy} - (z_s)_{xy}\big) (\tG_s)_{y\bm b} + (\tG_s)_{\bm a x}\big( \theta \Smat_{xy} - (z_s)_{xy}\big) (\tG_s^2)_{y\bm b}\big)- 2(M_s')_{\bm a \bm b}  \\
		 + & \sum_{x \neq y} \theta^2 \big(X_{xy}(s) - 1\big)^2 (\tG_s)_{\bm a x} (\tG_s^2)_{yy} (\tG_s)_{x\bm b} -  (\tG_s^2)_{\bm a \bm b}  \frac{1 - p_s}{1 -p}\frac{\langle M_s'\rangle}{p} \\
				 + & \sum_{x \neq y} \theta^2 \big(X_{xy}(s) - 1\big)^2 \big((\tG_s)_{\bm a x} (\tG_s)_{yy} (\tG_s^2)_{x\bm b} + (\tG_s^2)_{\bm a x} (\tG_s)_{yy} (\tG_s)_{x\bm b}\big) -  2(\tG_s^3)_{\bm a \bm b}  \frac{1 - p_s}{1 -p}\frac{\langle M_s\rangle}{p} \\
		+ &  \sum_{x \neq y} \theta^2 \big(X_{xy}(s) - 1\big)^2 \big( (\tG_s)_{\bm a x} (\tG_s)_{yx} (\tG_s^2)_{y\bm b} + (\tG_s^2)_{\bm a x} (\tG_s)_{yx} (\tG_s)_{y\bm b}  + (\tG_s)_{\bm a x} (\tG_s^2)_{yx} (\tG_s)_{y\bm b}\big) \\
		+& \,  \Rcal^{(\geq 3)}(s)\,,
	\end{split}
\end{equation}
where we inserted our choice of Bernoulli characteristic flow $\partial_s z_s$ from \eqref{eq:Bernflow} and collected all terms containing at least a factor $\theta^3$ from the resolvent expansion in $\Rcal^{(\geq 3)}(s)$.

As in the other three cases, the first two lines of \eqref{eq:entGenexpand2} simply equal $2(\tG_s^2 - M_s')_{\bm a \bm b}$, constituting the integrand in \eqref{eq:entGron2}. The last four lines of \eqref{eq:entGenexpand2} are in fact  error terms, i.e.
\begin{equation}
\int_{0}^{t \wedge \tau} \dd s \, [\text{last four lines of \eqref{eq:entGenexpand2}}] \leq \Lambda_{{\rm ent}}^{(2)}(t \wedge \tau)
\end{equation}
with very high probability.
This can be shown following very similar estimates as the ones presented for controlling the generator terms of $\phi^{\rm av, 1}_t$, $\phi^{\rm av, 2}_t$, and $\phi_{ \bm a \bm b, t}^{\rm ent, 1}$, and is hence omitted for brevity. 

This concludes the discussion of the generator terms for all four observables and hence the whole proof of Lemma \ref{lem:generate}. \qed

\section{General isotropic local law} \label{sec:isolaw}
In this section, we give the proof of the general isotropic law in Theorem \ref{thm:lolaw}~(ii) and explain the improvement for off-diagonal resolvent entries as formulated in Remark \ref{rmk:offdiag}.
\subsection{Proof of the general isotropic law} \label{subsec:isoproof}
Fix vectors $\widetilde{\bm v}, \widetilde{\bm w} \in \R^N$.
Analogously to the entrywise law, we introduce the set of vectors 
\begin{equation}
	\mathcal{V} = \mathcal{V}_{\widetilde{\bm v} \widetilde{\bm w}} = \big\{ \widetilde{\bm v}, \widetilde{\bm w}, \bm e\big\} \cup \big\{ \bm e_a : a \in [N]\big\}
\end{equation}
and define, for any $\bm v, \bm w \in \mathcal{V}$, the \emph{isotropic observables}
\begin{equation} \label{eq:tGiso}
	\phi_{\bm v \bm w, t}^{{\rm iso}}(X(t)) \deq \big(\tG_t(z_t) - M_t(z_t)\big)_{\bm v \bm w}\,,
\end{equation}
where $\tG_t(z_t) = \big(\theta \widetilde{X}(t) - z_t\big)^{-1}$ is the resolvent of the random matrix $\widetilde{X}(t)$ which has its entries within $\mathcal{I}$ set to zero, i.e.
\begin{equation} \label{eq:jumpsto0}
	\widetilde{X}_{xy}(t) = \widetilde{X}_{yx}(t) = 0 \quad \text{for all} \quad x,y \in \mathcal{I} \quad \text{and all times} \quad t \in [0,T] \,. 
\end{equation}
Here, the set of vertices $\mathcal{I}$ is defined as  with
\begin{equation} \label{eq:Idef}
\mathcal{I} \deq \mathcal{I}_{\bm v} \cup \mathcal{I}_{\bm w}\,, \qquad \mathcal{I}_{{\bm v}} \deq \left\{ x \in [N] : |\langle \Pi_{\bm e}^\perp{\bm v}, \bm e_x \rangle|  \geq (\log N)^{-5/16} \Vert \Pi_{\bm e}^\perp \bm v \Vert \right\} \subset [N]. 
\end{equation}
and analogously for $\mathcal{I}_{\bm w}$. Note that we always have that 
\begin{equation}
|\mathcal{I}| \leq 2 (\log N)^{5/8} \,. 
\end{equation}

\subsubsection{Stopping time, martingale, and Grönwall estimate: Proof of the isotropic law}

Armed with these preparatory ingredients, given an $N$-independent constant $\mathfrak{C}_{\rm iso}$, that has to be chosen in the proof of Proposition \ref{prop:Groniso} below, we define the stopping time 
\begin{equation}
\tau = \tau (\mathfrak{C}_{\rm iso})\deq \inf \left\{ t \in [0,T] : \max_{\bm v, \bm w \in \mathcal{V}}\frac{\big|\phi_{\bm v \bm w, t}^{\rm iso}\big|}{ \vertiii{\bm v}_{\bm e} \, \vertiii{\bm w}_{\bm e} \, \Lambda_{{\rm iso}}(t)} \geq \mathfrak{C}_{\rm iso} \right\}\,,
\end{equation}
where we introduced the shorthand notation
\begin{equation} \label{eq:Lambdaiso}
\Lambda_{{\rm iso}}(t) \deq \sqrt{\frac{\log N}{N \eta_t}} + \sqrt{\frac{\log N}{Np}} \,. 
\end{equation}
The naturally associated martingale to the isotropic observable is given by
\begin{equation}
\Mart_{\bm v \bm w, t}^{\rm iso} \deq \phi_{\bm v \bm w, t}^{{\rm iso}}(X(t))  - \phi_{\bm v \bm w, 0}^{{\rm iso}}(X(0))  - \int_{0}^{t} \dd s \, \Big( \big(\Lcal \phi_{\bm v \bm w, s}^{{\rm iso}} \big)\big(X(s)\big) + \partial_s \phi_{\bm v \bm w, s}^{{\rm iso}}\big(X(s)\big) \Big) \,. 
\end{equation}
We have the following isotropic stochastic Grönwall estimate lying the basis for the proof of the isotropic local law, completely analogously to Proposition \ref{prop:Gron}. 
\begin{proposition}[Isotropic stochastic Grönwall estimate] \label{prop:Groniso}
Adopt the notations and conventions from above, and assume that 
\begin{equation} \label{eq:isocond}
	Np \geq (1+\mathfrak{C}_{\rm iso})^2 (\log N)^2 \quad \text{and} \quad N \eta_T \geq (1+\mathfrak{C}_{\rm iso})^2 \log N \,. 
\end{equation}
Then there exists a constant $C_\kappa$, depending only on $\kappa$ from \eqref{eq:domain}, such that for any $c \geq 10 \mathfrak{C}_{\rm iso}^{-1/2}$,
\begin{equation*}
	\begin{split}
&\proba{	\exists t \in [0,T], \ \bm v, \bm w \in \mathcal{V} : \big| \phi^{{\rm iso}}_{\bm v \bm w, t\wedge \tau} \big| \geq C_\kappa \int_{0}^{t \wedge \tau} \dd s \, 	\big| \phi^{{\rm iso}}_{\bm v \bm w, s} \big| \, + (1 + 10 c \mathfrak{C}_{\rm iso}) \vertiii{\bm v}_{\bm e} \, \vertiii{\bm w}_{\bm e} \, \Lambda_{{\rm iso}}(t \wedge \tau)} \\
&\hspace{5.8cm}\leq N^{-c \mathfrak{C}_{\rm iso}/C_\kappa + 10}\,.
	\end{split}
\end{equation*}
\end{proposition}
The proof of Proposition \ref{prop:Groniso} is given in Section \ref{subsubsec:proofGroniso} below. 
\begin{proof}[Proof of the istropic local law in Theorem \ref{thm:lolaw}~(ii)]
Armed with Proposition \ref{prop:Groniso}, the proof of the isotropic local law \eqref{eq:isoLL} follows analogously to that of the average and isotropic laws in Section \ref{sec:Bernoulli}. The main difference is that, analogously in order to eventually go from $\tG$ in \eqref{eq:tGiso} to the usual $G$, one has to replace up to $|\mathcal{I}|^2 \leq 4 (\log N)^{5/4}$ many matrix entries of size (up to) $(Np)^{-1/2}$. However, since $|\mathcal{I}|^2/\sqrt{Np}$ might be large, we cannot simply sum up the errors. 

To overcome this difficulty, we consider, for fixed $z \in \mathbb{D}$, 
\begin{equation*}
G(u) = G(u; z)\deq \big(\theta \big[\widetilde{X}(T) + u(X(T) - \widetilde{X}(T))\big] - z\big)^{-1} \quad \text{for} \quad u \in [0,1]\,,
\end{equation*}
which smoothly interpolates between $\widetilde{G} = G(0)$ and $G = G(1)$. 

Then, by differentiating in $u$, we obtain
\begin{equation} \label{eq:derivident}
\partial_{u} G_{\bm v \bm w}(u) = - \theta \sum_{x,y \in \mathcal{I}} X_{xy}(T) G_{\bm v x}(u) G_{y \bm w}(u)  
\end{equation}
and hence
\begin{equation} \label{eq:partialest}
|\partial_{u} G_{\bm v \bm w}(u)| \preceq \frac{1}{\sqrt{Np}} \max_{x,y \in \mathcal{I}}  |G_{\bm v x}(u) G_{y \bm w}(u) | 
\end{equation}
uniformly in $\bm v, \bm w \in \mathcal{V}$ and $u \in [0,1]$, by application of \eqref{eq:IBennett}, yielding that 
\begin{equation} \label{eq:IXxy}
\sum_{x,y \in \mathcal{I}} X_{xy}(T) \preceq 1 \,. 
\end{equation}

We can then apply Grönwall's lemma to \eqref{eq:partialest} and find that\footnote{To be precise, we first apply Grönwall for the case that $\bm v, \bm w$ are both standard basis vectors. Afterwards, we treat the case that one of the two vectors is not a standard basis vector in $\mathcal{V}$, and finally the case where both $\bm v, \bm w$ are not standard basis vectors. The outcome can be summarized in the following bound \eqref{eq:Gubound}.}
\begin{equation} \label{eq:Gubound}
| G_{\bm v \bm w}(u)| \lesssim  |M_{\bm v \bm w}| + \vertiii{\bm v}_{\bm e} \, \vertiii{\bm w}_{\bm e}
\end{equation}
with very high probability,
uniformly in $u \in [0,1]$ and $\bm v, \bm w \in \mathcal{V}$. 
By integrating \eqref{eq:derivident} and using \eqref{eq:Gubound} together with \eqref{eq:IXxy}, we conclude
\begin{equation*}
\max_{\bm v, \bm w \in \mathcal{I}} \left| G_{\bm v \bm w} - \widetilde{G}_{\bm v \bm w} \right| \preceq \frac{1}{\sqrt{Np}} \vertiii{\bm v}_{\bm e} \, \vertiii{\bm w}_{\bm e}\,,
\end{equation*}
where we additionally used that $|M_{\bm v x} |\lesssim \vertiii{\bm v}_{\bm e}$ and $|M_{y \bm w} |\lesssim \vertiii{\bm w}_{\bm e}$, both uniformly in $x, y \in [N]$.  We have hence proven the isotropic local for $G$ as well. 
\end{proof}

\subsubsection{Resolvent expansions: Proof of Proposition \ref{prop:Groniso}} \label{subsubsec:proofGroniso}
As in Section \ref{sec:Gronproof}, the proof of Proposition~\ref{prop:Groniso} is based on three ingredients: An estimate on the quadratic variation of the martingale term, a bound on the jump sizes, and controlling the generator terms. 
\begin{lemma}[Quadratic variation of martingale terms] \label{lem:martestiso}
	Using the assumptions and notations from above, there exists a constant $C_\kappa > 0$ such that for $i \in [2]$, and uniformly in $t \in [0,T]$ and $\bm v, \bm w \in \mathcal{V}$, 
	\begin{equation}
	\langle\Mart^{{\rm iso}}_{\bm v \bm w}\rangle_{t \wedge \tau} \leq \frac{C_\kappa}{\log N} \big(\vertiii{\bm v}_{\bm e} \, \vertiii{\bm w}_{\bm e} \, \Lambda_{{\rm iso}}(t \wedge \tau)\big)^2 \,. 
	\end{equation}
\end{lemma}

\begin{lemma}[Bound on jump sizes] \label{lem:jumpiso}
	Using the assumptions and notations from above, there exists a constant $C_\kappa > 0$ such that, uniformly in $t \in [0,\tau]$ and $\bm v, \bm w \in \mathcal{V}$,
	\begin{equation}
 \left|\Mart^{{\rm iso}}_{\bm v \bm w, t} - \Mart^{{\rm iso}}_{\bm v \bm w, t-}\right| \leq \left(\frac{C_\kappa}{\sqrt{Np}} + \frac{C_\kappa}{(\log N)^{9/8}}\right) \Cf_{\rm iso}\, \vertiii{\bm v}_{\bm e} \, \vertiii{\bm w}_{\bm e}\, \Lambda_{{\rm iso}}(t) \,. 
	\end{equation}
\end{lemma}

\begin{lemma}[Generator terms] \label{lem:generateiso}
	Using the assumptions and notations from above, there exists a constant $C_\kappa > 0$ such that, uniformly for $t \in [0,T]$ and $\bm v, \bm w \in \mathcal{V}$, 
	\begin{equation}
		\left|\int_{0}^{t\wedge \tau} \dd s \ \big( \Lcal \phi^{{\rm iso}}_{\bm v \bm w, s} + \partial_s \phi^{{\rm iso}}_{\bm v \bm w, s}  \big) \right| \leq C_\kappa \int_{0}^{t \wedge \tau} \dd s \,  \big| \phi^{{\rm iso}}_{\bm v \bm w, s} \big| + \vertiii{\bm v}_{\bm e} \, \vertiii{\bm w}_{\bm e} \, \Lambda_{{\rm iso}}(t \wedge \tau)
	\end{equation}
with probability at least $1 - N^{-\Cf_{\rm iso}/C_\kappa}$
\end{lemma}
Armed with the above three lemmas and the jump process version of the BDG inequality from Proposition \ref{prop:DBG} and Remark \ref{rem:time_dep_phi}, the proof of Proposition \ref{prop:Groniso} is completely analogous to that of Proposition \ref{prop:Gron}, choosing $\Cf_{\rm iso}$ similarly to \eqref{eq:Cchoice}. 
\qed
\\[1mm]

It thus remains to give the proofs of Lemmas \ref{lem:martestiso}, \ref{lem:jumpiso}, and \ref{lem:generateiso}.
\begin{proof}[Proof of Lemma \ref{lem:martestiso}]
We follow the argument starting from \eqref{eq:Martent1}, based on resolvent expansions (see Lemma \ref{lem:resolventexpand}). The proof provided there holds (almost) line by line also for general isotropic vectors $\bm v, \bm w$, employing straightforward changes accounting for the non-isotropic norms $\vertiii{\bm v}_{\bm e}$ and $\vertiii{\bm w}_{\bm e}$, in particular using Lemma \ref{lem:Mprop} and the fact that $|M_{\bm v x} |\lesssim \vertiii{\bm v}_{\bm e}$ and $|M_{y \bm w} |\lesssim \vertiii{\bm w}_{\bm e}$, both uniformly in $x, y \in [N]$.
This concludes the proof of Lemma~\ref{lem:martestiso}. 
\end{proof}

\begin{proof}[Proof of Lemma \ref{lem:jumpiso}]
The proof closely follows the arguments presented in Section \ref{subsec:jumpproof} employing obvious adjustments involving the non-isotropic norms of $\bm v$ and $\bm w$, again, in particular using the bounds $|M_{\bm v x} |\lesssim \vertiii{\bm v}_{\bm e}$ and $|M_{y \bm w} |\lesssim \vertiii{\bm w}_{\bm e}$ uniformly in $x,y \in [N]$. 

Besides these straightforward changes, the only nontrivial adjustment is that $\max_{x \neq y}^{ab}$ from \eqref{eq:Mtrickbound} is replaced by a maximum over all $x \neq y$ such that $x,y \notin \mathcal{I}$ (recall the definition of $\mathcal{I}$ from \eqref{eq:Idef}). By construction of the set $\mathcal{I}$ and involving Lemma \ref{lem:Mprop} together with the Cauchy-Schwarz inequality and using $\Vert \Pi_{\bm e} \bm e_x \Vert = N^{-1/2}$, we thus infer
\begin{equation*}
	\begin{split}
	\max_{\substack{x \neq y: \\
		x ,y\notin \mathcal{I}}}\Big[|(M_s)_{\bm v x}| |(M_s)_{y \bm w}| \Big] &\lesssim \left( \max_{x \notin \mathcal{I}} |\langle \Pi_{\bm e}^\perp \bm v,  \bm e_x \rangle|  + \frac{\Vert \Pi_{\bm e} \bm v \Vert}{Np^{1/2}}\right) \, \left( \max_{y \notin \mathcal{I}} |\langle \Pi_{\bm e}^\perp \bm w, \bm e_y \rangle|  + \frac{\Vert \Pi_{\bm e} \bm w \Vert}{Np^{1/2}}\right) \\
	&\lesssim \left(\frac{\Vert \Pi_{\bm e}^\perp \bm v \Vert}{(\log N)^{5/16}} + \frac{\Vert \Pi_{\bm e} \bm v \Vert}{Np^{1/2}}\right) \left(\frac{\Vert \Pi_{\bm e}^\perp \bm w \Vert}{(\log N)^{5/16}} + \frac{\Vert \Pi_{\bm e} \bm w \Vert}{Np^{1/2}}\right) \lesssim \frac{\vertiii{\bm v}_{\bm e} \, \vertiii{\bm w}_{\bm e}}{(\log N)^{5/8}}
	\end{split}
\end{equation*}
from which, following the argument in Section \ref{subsec:jumpproof}, we conclude Lemma \ref{lem:jumpiso}. 
\end{proof}

\begin{proof}[Proof of Lemma \ref{lem:generateiso}]
We follow the argument given in Section \ref{subsubsec:genent1}. Again, the proof carries over immediately for general vectors $\bm v, \bm w$.
The only difference is that in the analog of \eqref{eq:E1ent1} (estimating \eqref{eq:E1estent}) we employ an \emph{average} local law for $\tG$, which follows by a simple resolvent expansion replacing (up to) $|\mathcal{I}|^2 \leq 4 (\log N)^{5/4}$ matrix entries having size $(Np)^{-1/2} \lesssim (\log N)^{-1}$ each. More precisely, replacing one matrix entry $(x,y) \in \mathcal{I}^2$ amounts to an error 
\begin{equation*}
|\langle \tG_s \rangle - \langle \tG_s^{xy,1} \rangle| \lesssim \frac{1}{\sqrt{Np}} \frac{1}{N \eta_s} \lesssim \frac{1}{(\log N)^{3/2}} \frac{(\log N)^{1/2}}{N \eta_s}
\end{equation*}
as shown in \eqref{eq:av1jump}. Summing this up $|\mathcal{I}|^2$ many times provides the desired average law for $\tG$. 

In total, we find that 
\begin{equation*}
\left|  \int_{0}^{t\wedge \tau} \dd s \, \Big( \big(\Lcal \phi_{\bm v \bm w, s}^{{\rm iso}} \big)\big(X(s)\big) + \partial_s \phi_{\bm v \bm w, s}^{{\rm iso}}\big(X(s)\big) \Big) \right| \leq \int_{0}^{t \wedge \tau} \dd s \, \big|\phi_{\bm v \bm w, s}^{\rm iso}(s)\big| + \vertiii{\bm v}_{\bm e} \, \vertiii{\bm w}_{\bm e} \,\Lambda_{{\rm iso}}(t \wedge \tau) 
\end{equation*}
with probability at least $1 - N^{-\Cf_{\rm iso}/C_\kappa}$. This concludes the proof of Lemma \ref{lem:generateiso}. 
\end{proof}

\subsection{Off-diagonal law} \label{subsec:offdiag}
The improvement for the off-diagonal entrywise law can be obtained following the same arguments as before: One defines a stopping time $\tau$ capturing the additional smallness due to off-diagonality, and then deduces a stochastic Grönwall estimate proving that $\tau = T$, for which we can import the already proven average law. As before, the argument rests on three technical ingredients: an estimate on the quadratic variation of the martingale term, a bound on the jump sizes, and controlling the generator terms. Following the proofs of these ingredients presented in Section \ref{sec:lemproofs} for $\phi_{ \bm a \bm b, t}^{{\rm ent}, 1}$ (and the associated martingale), we see that the only difference occurs in the diagonal contribution of the second order resolvent expansion term in the martingale estimate \eqref{eq:R2Dest}. In fact, instead of the bound \eqref{eq:R2Dest}, we estimate (ignoring the tilde for notational simplicity) for $a \neq b$: 
\begin{multline*}
\frac{1}{(Np)^2} \int_{0}^{t \wedge \tau} \dd s \, \sum_{x,y} \big| (G_s)_{ax} (G_s)_{yy} (G_s)_{xb} \big|^2 \\  \lesssim \frac{1}{Np^2} \int_{0}^{t \wedge \tau} \dd s \, \left(\frac{\log N}{N \eta_s} + \frac{1}{Np}\right)\left[\sum_{x \neq a } |(G_s)_{xb}|^2 + 1\right] \lesssim \frac{\log N}{Np} \left(\frac{\log N}{N \eta_{t \wedge \tau}} + \frac{1}{Np}\right)\,,
\end{multline*}
where we used that $|M_{xy}| \lesssim \delta_{xy} + N^{-1}$ by Lemma \ref{lem:Mprop} and employed \eqref{eq:intrulepractice}. This improves upon \eqref{eq:R2Dest} in the $1/(Np)$-term in the desired way and we conclude our discussion of the off-diagonal entrywise law.

\section{Local spectral statistics: Proof of Theorem \ref{thm:GOE}} \label{sec:proof_GOE}

In this section we prove Theorem \ref{thm:GOE}. The main inputs are our local law, Theorem \ref{thm:lolaw}, as well as the recent characterization of the $\mathrm{Sine}_1$ point process via loop equations by Bourgade and Huang \cite{BourHua2026}. Our presentation largely follows \cite[Section 2.3]{BourHua2026} with the key differences being that in our sparse setup the error terms have to be handled much more carefully in order to reach the scale $Np \geq (\log N)^{2 + \kappa}$.

\subsection{Overview of the proof} \label{sec:GOE_overview}
We begin with an overview of the argument, explaining the new ideas needed to reach the scale $Np \geq (\log N)^{2 + \kappa}$. The criterion of \cite{BourHua2026} requires the verficiation of loop equations at the microscopic scale $\im z = \eta \asymp \frac{1}{N}$. As in \cite[Section 2.3]{BourHua2026}, we do this using a third-order cumulant expansion. For the lowest order loop equation ($k = 1$ in Proposition \ref{prop:loop} below), the main error term from a third-order cumulant expansion applied to $\E[\ang{G}]$ is
\begin{equation} \label{goe_sketch_error}
\frac{1}{N^2 (Np)^{1/2}} \sum_{x,y} \E[G_{xy} G_{xx} G_{yy}]\,,
\end{equation}
where all resolvents are evaluated at $z$ with imaginary part $\eta \asymp \frac{1}{N}$. We may estimate the Green function entries on the right-hand side using the fact that the map $\eta \mapsto \eta \im G(E + \ii \eta)$ is operator increasing, from which we can deduce that
\begin{equation} \label{G_est_large_eta}
\abs{G_{xy}(z)} \lesssim \frac{\eta + t}{\eta} \max_{u \in \{x,y\}} \im G_{uu}(z + \ii t)\,.
\end{equation}
By Theorem \ref{thm:lolaw}, $\im G_{uu}(z + \ii t) \lesssim 1$ for $t \gtrsim \frac{\log N}{N}$, and therefore we can estimate the right-hand side of \eqref{goe_sketch_error} by
\begin{equation*}
\frac{(\log N)^3}{(Np)^{1/2}}\,.
\end{equation*}
This simple-minded approach therefore requires that $Np \gg (\log N)^6$.

What we failed to exploit above is that in the main result of \cite{BourHua2026}, the error term is allowed to contain a factor $\E [1 + \abs{\ang{G}}^2]$. We therefore have to estimate as many factors of $G$ in terms of $\ang{G}$ as possible, in order to avoid the wasteful estimate \eqref{G_est_large_eta} that yields factors of $\log N$. Although \eqref{goe_sketch_error} does not a priori contain any factor $\ang{G}$, we can easily generate one by estimating the sum using Cauchy-Schwarz and the applying the Ward identity. This yields the bound
\begin{multline} \label{error_cumul_outline}
\frac{1}{N^2 (Np)^{1/2}} \E \qBB{\pbb{\sum_{x,y} \abs{G_{xy}}^2}^{1/2} \sum_x \abs{G_{xx}}^2}
\\
\lesssim \frac{1}{(Np)^{1/2}} \E \qBB{(\im \ang{G})^{1/2} \frac{1}{N} \sum_x \abs{G_{xx}}^2} \lesssim \frac{(\log N)^2}{(Np)^{1/2}} \E[(\im \ang{G})^{1/2}]\,,
\end{multline}
where in the last step we used \eqref{G_est_large_eta} to estimate $\frac{1}{N} \sum_x \abs{G_{xx}}^2 \lesssim (\log N)^2$. This results in the improved condition $Np \gg (\log N)^4$.

To reach the condition $Np \geq (\log N)^{2 + \kappa}$, we cannot afford to estimate as many factors of $G$ using \eqref{G_est_large_eta}, and instead we need to recreate as many factors of $\ang{G}$ as possible. The off-diagonal factor $G_{xy}$ can be easily dealt with by using the Ward identity, as in the above estimate. The diagonal terms $G_{xx}$ and $G_{yy}$ cannot be handled in this way. Formulas of the kind
\begin{equation*}
G_{xx} = \ang{G} + G_{xx} \ang{HG} - \ang{G} (HG)_{xx}
\end{equation*}
allow one to replace $G_{xx}$ with $\ang{G}$, but on the microscopic scale $\eta \asymp \frac{1}{N}$ the error in this formula is not smaller than the main term, and hence such formulas are not useful in our setting.

Our solution to handle the diagonal terms $G_{xx}$ is to choose $\eta_0 \asymp \frac{\log N}{N}$ and write
\begin{equation*}
\abs{G_{xx}(z)} = \absbb{G_{xx}(z + \ii \eta_0) - \int_0^{\eta_0} \partial_\eta G(z + \ii t) \, \dd t} \lesssim 1 + \int_{0}^{\eta_0} \frac{\im G_{xx} (z + \ii t)}{\eta + t} \, \dd t\,.
\end{equation*}
Since $\im G_{xx}$ is positive, summing over $x$ will yield factors of $\ang{G}$. Thus we estimate
\begin{align*}
\frac{1}{N} \sum_x \abs{G_{xx}}^2 &\lesssim 1 + \frac{1}{N} \sum_x \int_0^{\eta_0} \frac{\dd t_1}{\eta + t_1} \int_0^{\eta_0} \frac{\dd t_2}{\eta + t_2} \, \im G_{xx} (z + \ii t_1) \, \im G_{xx} (z + \ii t_2)
\\
&\leq 1 + \pBB{\int_0^{\eta_0} \frac{\dd t}{\eta + t} \pbb{\frac{1}{N} \sum_x (\im G_{xx} (z + \ii t))^2}^{1/2}}^2\,.
\end{align*}
In order to obtain a factor $\ang{G(z)}$ inside the integral, we estimate one factor of $\im G_{xx}$ using that the map $\eta \mapsto \eta \im G(E + \ii \eta)$ is operator increasing and the other by using that the map $\eta \mapsto \eta^{-1} \im G(E + \ii \eta)$ is operator decreasing, which yields
\begin{equation*}
\im G_{xx} (z + \ii t)^2 \leq \frac{\eta + \eta_0}{\eta} \im G_{xx}(z) \im G_{xx}(z + \ii \eta_0) \lesssim \log N \im G_{xx}(z)\,.
\end{equation*}
We conclude that
\begin{equation*}
\frac{1}{N} \sum_x \abs{G_{xx}}^2 \lesssim 1 + \log N \im \ang{G(z)} \pbb{\int_0^{\eta_0} \frac{\dd t}{\eta + t}}^2 \lesssim 1 + \log N (\log \log N)^2 \im \ang{G(z)}\,.
\end{equation*}
Plugging this into the second step of \eqref{error_cumul_outline} yields that \eqref{goe_sketch_error} is bounded by
\begin{equation*}
\frac{\log N (\log \log N)^2}{(Np)^{1/2}} \E[1 + (\im \ang{G})^{3/2}]\,,
\end{equation*}
which allows us to conclude for $Np \geq (\log N)^{2 + \kappa}$.

\subsection{Preliminaries} \label{subsec:prelim}
Throughout the entire section, we  adjust our notation slightly and denote the rescaled adjacency matrix $\theta A$ of the Erd\H{o}s-Rényi graph, previously written as $H$ in \eqref{eq:ERmodel}, by
\begin{equation} \label{eq:rescaledredefined}
\theta A = \theta p \Smat + H \quad \text{with} \quad \E H = 0 \quad \text{and} \quad \E h_{xy}^2 = \frac{1}{N} \quad \text{for} \quad x \neq y \,. 
\end{equation}
Fix a small $\kappa > 0$ and assume henceforth that
\begin{equation} \label{eq:Npgekappa}
Np \geq (\log N)^{2+\kappa} \,. 
\end{equation} 
Recalling the definition of the domain from \eqref{eq:domain}, we define the event
\begin{equation} \label{eq:goodevent}
\Omega_N = \Omega_N(C) \deq \left\{ \max_{x,y} |G_{xy}(z)| + |G_{\bm e \bm e}(z)|\leq C \quad \text{for all} \quad z \in \mathbb{D}(\kappa, C) \cup \overline{\mathbb{D}(\kappa, C)} \right\} \,. 
\end{equation}
As a consequence of Lemma \ref{lem:Mprop} and Theorem \ref{thm:lolaw},\footnote{The result on $\overline{\mathbb{D}(\kappa, C)}$ simply follows by complex conjugation.} it holds that, for any $D > 0$ there exists $C> 0$ such that 
\begin{equation}
\P[\Omega_N(C)^{\mathsf{c}}] \leq N^{-D} \,. 
\end{equation}

The first key input for our proof is the following control on Green function entries. As a consequence of our local law in Theorem \ref{thm:lolaw}, we find them to be bounded by $\log N$ down to microscopic scales with $|\im z| \asymp 1/N$. 
\begin{lemma}[Monotonicity estimate] \label{lem:monotone}
On the event $\Omega_N$, for every $z \in \C \setminus \R$ with $E \deq \Re z$ satisfying $|E| \leq 2 - \kappa$ and $\eta \deq |\im z|$, we have 
\begin{equation}
\max_{x,y} |G_{xy}(z)| + |G_{\bm e \bm e}(z)| \leq \Lambda(z) \quad \text{with} \quad \Lambda(z) \deq C \max\left\{ 1, \frac{C \log N}{N \eta} \right\} \,. 
\end{equation}
In particular, if $z = E + w/(N \pi \rho_E)$ with $w$ in a compact set $K \subset \C \setminus \R$, then $\Lambda(z) \leq C_K \log N$. 
\end{lemma}
\begin{proof}
Using Theorem \ref{thm:lolaw}, the proof is identical to that of \cite[Lemma 2.5]{BourHua2026} and hence omitted. 
\end{proof}

The second key input is the following first order cumulant expansion formula with remainder, which we recall from \cite{HKR}.

\begin{lemma}[Cumulant expansion; see Lemma 2.4 in \cite{HKR}] \label{lem:cumexp}
Let $h$ be a centred real random variable satisfying
\begin{equation}
\E h = 0 \,, \quad \E h^2 = \frac{1}{N}\,,
\end{equation}
with finite moments of all orders. Let $f : \R\to \C$ be smooth. Then 
\begin{equation}
\E [h f(h)] = \frac{1}{N} \E[f'(h)] + \mathcal{R}
\end{equation}
where the remainder $\mathcal{R}$ satisfies
\begin{equation}
|\mathcal{R}| \leq \E \left[\big(N^{-1} |h| + |h|^3\big) \sup_{|\nu| \leq |h|} |f''(\nu)| \right]
\end{equation}
\end{lemma}

For $x \leq y$ and a real parameter $\nu$ define the matrix $\Delta^{xy}$ by
\begin{equation}
\big(\Delta^{xy}\big)_{ij} \deq   \delta_{ix} \delta_{jy} + (1 - \delta_{ij}) \delta_{iy} \delta_{jx} 
\end{equation}
and 
\begin{equation}
H^{(xy,\nu)} \deq H + (\nu - h_{xy}) \Delta_{xy} \,, \quad G^{(xy, \nu)}(z) \deq \big( \theta p E + H^{(xy,\nu)} - z\big)^{-1}
 \,. 
\end{equation}
For a differentiable function $F$ of the matrix entries, set
\begin{equation}
\partial_{xy} F(H) \deq \frac{\dd}{\dd \nu} F(H^{(xy, \nu)}) \bigg\vert_{\nu = h_{xy}} \,, \quad x \leq y
\end{equation}
and we use the convention that $\partial_{yx} \deq \partial_{xy}$ for $x \leq y$. 

To control the remainder term in the cumulant expansion, we need the following stability lemma for the resolvent. It allows to control the perturbed resolvent $G^{(xy, \nu)}$ in terms of the unperturbed resolvent $G$ as easily follows using resolvent expansions as in Lemma \ref{lem:resolventexpand}. 

\begin{lemma}[Resolvent stability] \label{lem:stability}Fix a spectral parameter $z \in \C \setminus \R$ and assume that, for  $\Lambda \in [0, (Np)^{1/2 - \epsilon}]$ for some $\epsilon >0$, we have
\begin{equation}
\max_{a,b} |G_{ab}(z)| \leq \Lambda \,. 
\end{equation}
Then, uniformly in all indices $a,b, x,y \in [N]$, it holds that
\begin{equation} \label{eq:singlediff}
\sup_{|\nu| \leq 10(Np)^{-1/2}} \left| G_{ab}^{(xy,\nu)}(z) - G_{ab}(z) \right| \leq \frac{C}{\sqrt{Np}} \Lambda^2
\end{equation}
for some constant $C> 0$. 

Moreover, we have the following: 
\begin{align}
\sup_{|\nu| \leq 10(Np)^{-1/2}} \left| G_{xy}^{(xy,\nu)}(z) \right| &\leq C \left(  \left| G_{xy}(z) \right| + \frac{1}{\sqrt{Np}} \left| G_{xx}(z) G_{yy}(z) \right|  \right) \label{eq:offdiag}\\
\sup_{|\nu| \leq 10(Np)^{-1/2}} \left| G_{xx}^{(xy,\nu)}(z) \right| &\leq C   \left| G_{xx}(z) \right|  \label{eq:diag}
\end{align}
uniformly in indices $x,y \in [N]$, as well as
\begin{equation}
\sup_{|\nu| \leq 10(Np)^{-1/2}} \left| \big\langle G^{(xy,\nu)}(z)\big\rangle \right| \leq C\left( \left| \big\langle G(z)\big\rangle \right| + \frac{1}{N |\im z|}\right) \label{eq:trace}
\end{equation}
\end{lemma}
\begin{proof}
The proof of \eqref{eq:singlediff} is almost identical to that of \cite[Lemma 2.7]{BourHua2026} and hence omitted. The remaining claims follow by simple resolvent expansion (see Lemma \ref{lem:resolventexpand}), allowing to establish self-consistent inequalities as in the arguments around \eqref{eq:resexp} or \eqref{eq:xyupexpand} above. We omit the details for brevity.
\end{proof}

\subsection{Cumulant expansion error estimate and proof of Theorem \ref{thm:GOE}} \label{subsec:cumulants}
The goal of this section is to give the proof of Theorem \ref{thm:GOE}. As proven in \cite[Theorem 1.8]{BourHua2026}, the key to establishing $\op{Sine}_1$ microscopic spectral statistics is to verify a microscopic version of the \emph{loop equations} with parameter $\beta = 1$, given by \eqref{eq:loopeqmicro} below.

\begin{proposition}[Microscopic loop equations] \label{prop:loop}
Fix $E \in (-2,2)$ and $\kappa > 0$. Suppose that $(\log N)^{2+\kappa} \leq Np \leq N^{1-\kappa}$.
For $w \in \C \setminus \R$ define the rescaled and shifted Stieltjes transform
	\begin{equation} \label{eq:sdef}
s_N(w) \deq \int \frac{1}{s - z} \, \mu^E(\dd s) + \frac{E}{2 \pi \rho_E}= \sum_{\lambda \in \spec(H)} \frac{1}{N\pi \rho_E (\lambda - E) - w} + \frac{E}{2 \pi \rho_E} \,. 
	\end{equation}
	Then, for every $k \in \N$ and every compact set $K \subset \C \setminus \R$, uniformly for $w, w_2, ... , w_k \in K$, we have
\begin{multline} \label{eq:loopeqmicro}
		\E \Bigg[ \left(1 +   s_N(w)^2 + \partial_{w} s_N(w)\right) \Bigg(\prod_{j=2}^{k} s_N(w_j) \Bigg) \\
+ 2 \sum_{j=2}^k \partial_{w_j} \frac{s_N(w) - s_N(w_j)}{w - w_j} \Bigg(\prod_{\substack{i=2 \\ i \neq j}}^{k}s_N(w_i)\Bigg) \Bigg] = o(1)\left( 1 + \E \left[\mathcal{R}_N^{k+1}\right] \right)
\end{multline}
	where $o(1)$ vanishes as $N \to \infty$ and can be replaced by, e.g., $O((\log N)^{-\kappa/10}) $. The error term $\mathcal{R}_N$ is given by 
	\begin{equation}
\mathcal{R}_N \deq \max\{ |s_N(w)|, |s_N(w_2)|, ... , |s_N(w_k)|\} \,. 
	\end{equation}
\end{proposition}
\begin{proof}[Proof of Theorem \ref{thm:GOE}]
This is an immediate consequence of Proposition \ref{prop:loop}, using \cite[Theorem~1.8]{BourHua2026}. 
\end{proof}
In order to show Proposition \ref{prop:loop}, as in \cite[Section 2.3]{BourHua2026}, we use the cumulant expansion from Lemma \ref{lem:cumexp}. The following lemma controls the various error terms occurring in this expansion. 
\begin{lemma}[Errors in the cumulant expansion] \label{lem:cumexperror} Fix $k \in \N$ and a compact set $K \subset \C \setminus \R$. Fix $E \in (-2,2)$, let $w, w_2, ... , w_k \in K$ and set 
	\begin{equation} \label{eq:zscale}
z = E + \frac{w}{N \pi \rho_E}\,, \quad z_i = E + \frac{w_i}{N \pi \rho_E} \,, \quad i =2,... , k \,. 
	\end{equation}
	Let 
	\begin{equation}
\mathcal{G}_k \deq \prod_{i=2}^k \langle G(z_i) \rangle \,, \quad R_N \deq \max\{ |\langle G(z) \rangle|, |\langle G(z_2) \rangle|, ... , |\langle G(z_k) \rangle|\} \,. 
	\end{equation}
	For $x,y \in [N]$ define
	\begin{equation}
f_{xy}(\nu) \deq G_{xy}^{(xy, \nu)}(z)  \prod_{i=2}^k \langle G^{(xy, \nu)}(z_i) \rangle  \,. 
	\end{equation}
	Then 
	\begin{align}
\frac{1}{N} \sum_{x,y} \E \left[ \left( N^{-1}|h_{xy}| + |h_{xy}|^3 \right) \sup_{|\nu| \leq \frac{10}{ \sqrt{Np}}} \left| \partial_\nu^2 f_{xy}(\nu)\right| \right] &= O\left((\log N)^{-\kappa/10}\right) \left( 1 + \E \left[R_N^{k+1}\right] \right)\label{eq:thirdorder} \\
	p \theta	\E\left[ \big| G_{\bm e \bm e}(z) \mathcal{G}_k \big| \right] &= O\left(\frac{(\log N)^{5}}{N^{1/2}}\right) \left( 1 + \E \left[R_N^{k+1}\right] \right)\label{eq:eecontrol} \\
\frac{1}{N^2} \sum_{x} \big| \E \big[ \partial_{xx} \big(G_{xx}(z) \mathcal{G}_k\big) \big] \big| &= O\left(\frac{(\log N)^{5}}{N}\right) \left( 1 + \E \left[R_N^{k+1}\right] \right) \label{eq:diagonal}
	\end{align}
uniformly in $w, w_2, ... , w_k \in K$. 
\end{lemma}

We point out that, recalling \eqref{eq:rhodef}--\eqref{eq:sdef} and \eqref{eq:zscale}, the relation between the ``standard'' variables $z$ and $\langle G(z) \rangle$ used in Lemma \ref{lem:cumexperror} and the microscopic variables $w$ and $s_N(w)$ used in Proposition \ref{prop:loop} is given by
\begin{equation*}
	z = E + \frac{w}{N \pi \rho_E} \quad \langle G(z) \rangle = \pi \rho_E s_N(w) - \frac{E}{2} \,. 
\end{equation*}

\begin{proof}[Proof of Proposition \ref{prop:loop}]
The argument is very similar to that used to deduce \cite[Proposition 2.3]{BourHua2026} based on \cite[Lemma 2.8]{BourHua2026}, the analog of Lemma \ref{lem:cumexperror} here. In our case, we additionally use \eqref{eq:eecontrol} to control the non-zero expectation of the adjacency matrix and \eqref{eq:diagonal} to remedy that $h_{xx} = 0$ in our case. We omit the details. 
\end{proof}
It thus remains to give the proof of Lemma \ref{lem:cumexperror}. 
\begin{proof}[Proof of Lemma \ref{lem:cumexperror}]
Since $K \subset \C \setminus R$ is compact, there is a  constant $C_K > 0$ such that 
\begin{equation} \label{eq:CK}
 C_K^{-1}  N^{-1} \leq |\im z| \,, |\im z_i| \leq C_K N^{-1} \,, \quad i = 2, ... , k \,. 
\end{equation}
Therefore, using the trivial resolvent bound $\Vert G^{(xy,\nu)}(\tilde{z}) \Vert + \Vert G(\tilde{z})\Vert \leq 2|\im z|^{-1} \leq 2 C_K N$ for all $\tilde{z} \in \{z, z_2, ... , z_k\}$, we have that the contribution of $\Omega_N^{\mathsf{c}}$ to \emph{all} the estimates \eqref{eq:thirdorder}--\eqref{eq:diagonal} is negligible, simply because $\P[\Omega_N^{\mathsf{c}}] \leq N^{-D}$ for some $D$ chosen sufficiently large. We shall henceforth assume that all resolvents are controlled on the very-high-probability event $\Omega_N$. 

In fact, on the event $\Omega_N$, we have, by Lemma \ref{lem:monotone} and Lemma \ref{lem:stability} that, uniformly in $a,b,x,y \in [N]$,
\begin{equation} \label{eq:resolventbound}
\sup_{|\nu| \leq 10(Np)^{-1/2}} \left| G_{ab}^{(xy,\nu)}(z)  \right| \leq C_K \log N
\end{equation}
for all $z$ as in \eqref{eq:zscale}. To simplify notation, in the following we will omit the constants $C_K$ from \eqref{eq:CK} and  $C$ from \eqref{eq:goodevent} and simply write $\lesssim$, which also absorbs other $N$-independent constants. 

We proceed and record the standard resolvent derivative identities 
\begin{equation}
\begin{split}
\partial_\nu G^{(xy,\nu)}(z) = - G^{(xy,\nu)} (z)\Delta_{xy} G^{(xy,\nu)}(z) \\
\partial_{\nu}^2 G^{(xy,\nu)}(z) = 2 G^{(xy,\nu)}(z) \Delta_{xy} G^{(xy,\nu)}(z) \Delta_{xy} G^{(xy,\nu)}(z)
\end{split}
\end{equation}
Hence, in order to prove \eqref{eq:thirdorder}, we need to consider in total seven types of terms with different index structure, that can be created by action of $\partial_\nu^2$ on $f_{xy}(\nu)$. For simplicity of the presentation, in the items below, we drop the superscript $(xy,\nu)$, the spectral parameters, as well as non-differentiated factors of $\langle G \rangle$ occurring in $f_{xy}$: 
\begin{itemize}
\item[(i.a)] both derivatives act on $G_{xy}$, one gets $G_{xy} G_{xx} G_{yy}$; 
\item[(i.b)] both derivatives act on $G_{xy}$, one gets $G_{xy} G_{xy} G_{xy}$; 
\item[(ii.a)] only one derivative acts on $G_{xy}$, the other on a copy of $\langle G \rangle$, one gets $N^{-1}G_{xy} G_{xy} (G^2)_{xy}$; 
\item[(ii.b)] only one derivative acts on $G_{xy}$, the other on a copy of $\langle G \rangle$, one gets $N^{-1}G_{xx} G_{yy} (G^2)_{xy}$; 
\item[(iii.a)] no derivative acts on $G_{xy}$, but both on one copy of $\langle G \rangle$, one gets $N^{-1}G_{xy} G_{yy} (G^2)_{xx}$; 
\item[(iii.b)] no derivative acts on $G_{xy}$, but both on one copy of $\langle G \rangle$, one gets $N^{-1}G_{xy} G_{xy} (G^2)_{xy}$;
\item[(iii.c)] no derivative acts on $G_{xy}$, but both on a fresh copy of $\langle G \rangle$, one gets $N^{-2}G_{xy} (G^2)_{xy} (G^2)_{xy}$. 
\end{itemize} 

In the following, we will only analyse Cases (i.a) and (i.b) in detail; the other ones can be handled similarly and are hence just discussed briefly at the end of the proof. We shall additionally use $\E |h_{xy}|^3 \lesssim N^{-1} (Np)^{-1/2}$ and independence of $h_{xy}$ with $ \sup_{|\nu| \leq 10/ \sqrt{Np}} \left| \partial_\nu^2 f_{xy}(\nu)\right|$ to bound the expectation of the product by the expectation of $ \sup_{|\nu| \leq 10/ \sqrt{Np}} \left| \partial_\nu^2 f_{xy}(\nu)\right|$. To ease the notation we shall henceforth suppress the dependence on spectral parameters, whenever it does not lead to confusion.  From now on, we also assume without loss of generality that $\eta \deq \im z > 0$.

In Case (i.a), we thus have the estimate
\begin{align} \label{eq:caseia}
&\frac{1}{N^2 (Np)^{1/2}} \sum_{x,y} \E \left[ \mathbf{1}_{\Omega_N} \sup_{|\nu| \leq \frac{10}{ \sqrt{Np}}} \left| G_{xy}^{(xy, \nu)} G_{xx}^{(xy, \nu)} G_{yy}^{(xy, \nu)} \prod_{i=2}^k \langle G^{(xy,\nu)}(z_i) \rangle \right|  \right] \\ \notag
\lesssim \; &\frac{1}{N^2 (Np)^{1/2}} \sum_{x,y} \left(\E \left[ \mathbf{1}_{\Omega_N}  | G_{xy} G_{xx} G_{yy}| \big(1 + |\mathcal{G}_k| \big)  \right] + \frac{1}{\sqrt{Np}} \E \left[ \mathbf{1}_{\Omega_N}  | G_{xx} G_{xx} G_{yy} G_{yy}| \big( 1 + | \mathcal{G}_k|\big)  \right]\right)
 \end{align}
 by application of Lemma \ref{lem:stability}, in particular \eqref{eq:offdiag}--\eqref{eq:trace}. 
 
 We use the representation 
  \begin{equation}
 	G (z)  = G(z + \ii \eta_0)  -  \int_{0}^{\eta_0} \partial_\eta G (z + \ii t) \, \dd t
 \end{equation}
to bound the diagonal Green function entries as
 \begin{equation} \label{eq:Gxxest}
 	|G_{xx} (z) | \lesssim 1 +   \int_{0}^{\eta_0} \frac{\im G_{xx} (z + \ii t)}{\eta + t} \, \dd t \,. 
 \end{equation}	 
 where we choose $\eta_0 = C_K (\log N)/N$ with $C_K$ from \eqref{eq:CK}. Armed with \eqref{eq:Gxxest}, abbreviating $G^{(t)} \equiv G(z + \ii t)$ we find the first term in \eqref{eq:caseia} to be bounded as
 \begin{equation*}
 	\begin{split}
&\frac{1}{N^2 (Np)^{1/2}} \sum_{x,y} \E \left[ \mathbf{1}_{\Omega_N}  | G_{xy} G_{xx} G_{yy}| \big(1 + |\mathcal{G}_k| \big)  \right] \\
\lesssim \; &\frac{1}{N^2 (Np)^{1/2}} \sum_{x,y} \E \left[\mathbf{1}_{\Omega_N}  | G_{xy}| \big(1 + |\mathcal{G}_k| \big)  \right] \\
&+ \frac{1}{N^2 (Np)^{1/2}} \int_{0}^{\eta_0} \hspace{-1mm} \dd t_1  \int_{0}^{\eta_0} \hspace{-1mm}\dd t_2 \sum_{x,y} \E \left[ \mathbf{1}_{\Omega_N}  | G_{xy}| \frac{\im G_{xx}(z + \ii t_1)}{\eta + t_1} \frac{\im G_{xx}(z + \ii t_2)}{\eta + t_2}\big(1 + |\mathcal{G}_k| \big)  \right] \\
\lesssim \; &\frac{1}{N^{1/2} (Np)^{1/2}} \E \left[ \mathbf{1}_{\Omega_N} \left(\frac{1}{N}\sum_{x,y}|G_{xy}|^2\right)^{1/2} \big(1 + |\mathcal{G}_k| \big) \right] \\
& + \frac{1}{N^{1/2} (Np)^{1/2}}  \E \left[ \mathbf{1}_{\Omega_N} \big(1 + |\mathcal{G}_k| \big) \left(\frac{1}{N}\sum_{x,y}|G_{xy}|^2 \right)^{1/2} \prod_{i=1}^2\int_{0}^{\eta_0} \hspace{-1mm}\frac{\dd t_i}{\eta + t_i} \left(\frac{1}{N} \sum_{x} \big( \im G_{xx}^{(t_i)} \big)^2\right)^{1/2}  \right] \,.
 	\end{split}
 \end{equation*}
 By application of a Ward identity, the first term on the right-hand side admits the bound 
 \begin{equation}
\frac{1}{N^{1/2} (Np)^{1/2}} \E \left[ \mathbf{1}_{\Omega_N} \left(\frac{1}{N}\sum_{x,y}|G_{xy}|^2\right)^{1/2} \big(1 + |\mathcal{G}_k| \big) \right] \lesssim \frac{1}{\sqrt{Np}} \left( 1 + \E \left[R_N^{k+1}\right] \right) \,. 
 \end{equation}
 For the second term, we also employ  a Ward identity and additionally use Lemma \ref{lem:monotone} to estimate, on the event $\Omega_N$, one power of $\im G_{xx}^{(t_i)}$ inside the summation by  $\im G_{xx}^{(t_i)}\lesssim \frac{\log N}{N (\eta + t_i)}$ in order to get
\begin{align}
&\frac{1}{N^{1/2} (Np)^{1/2}}  \E \left[ \mathbf{1}_{\Omega_N} \big(1 + |\mathcal{G}_k| \big) \left(\frac{1}{N}\sum_{x,y}|G_{xy}|^2 \right)^{1/2} \prod_{i=1}^2\int_{0}^{\eta_0} \hspace{-1mm}\frac{\dd t_i}{\eta + t_i} \left(\frac{1}{N} \sum_{x} \big( \im G_{xx}^{(t_i)} \big)^2\right)^{1/2}  \right]
\notag \\
\lesssim \; &\frac{\log N}{N (Np)^{1/2}}  \E \left[ \mathbf{1}_{\Omega_N} \big(1 + |\mathcal{G}_k| \big) \langle \im G \rangle^{1/2} \prod_{i=1}^2\int_{0}^{\eta_0} \hspace{-1mm}\frac{\dd t_i}{(\eta + t_i)^{3/2}} \langle  \im G^{(t_i)} \rangle^{1/2}  \right]
\notag \\ \label{eq:imim}
\lesssim \; &\frac{\log N}{(Np)^{1/2}}  \E \left[ \mathbf{1}_{\Omega_N} \big(1 + |\mathcal{G}_k| \big) \langle \im G \rangle^{3/2} \left(\int_{0}^{\eta_0} \hspace{-1mm}\frac{\dd t}{\eta + t}\right)^{2}  \right] \lesssim \frac{\log N (\log \log N)^2}{(Np)^{1/2}}  \left( 1 + \E \left[R_N^{k+1}\right] \right) \,.
\end{align}
 Here, to go to the last line, we additionally used that $\eta \mapsto \eta^{-1}\im G(E + \ii \eta)$ is monotonically decreasing as an operator in order to bound $\im G^{(t_i)}$ in terms of $\im G$. In the ultimate step, we used that 
 \begin{equation} \label{eq:loglog}
\int_{0}^{\eta_0} \hspace{-1mm}\frac{\dd t}{\eta + t} \lesssim \log\left(\frac{\eta + \eta_0}{\eta}\right) \lesssim \log \log N 
 \end{equation}
 since $\eta_0 \asymp (\log N)/N$ and $\eta \asymp 1/N$. 
 We are thus left with controlling the second term in \eqref{eq:caseia}.  For each factor of $G_{xx}$ and $G_{yy}$ we employ \eqref{eq:Gxxest} and find, similarly to above, the result to be bounded by
\begin{align}
&\frac{1}{Np}  \left( 1 + \E \left[R_N^{k+1}\right] \right) + \frac{1}{Np} \E \left[  \mathbf{1}_{\Omega_N}\big(1 + |\mathcal{G}_k| \big) \prod_{i=1}^2\int_{0}^{\eta_0} \hspace{-1mm}\frac{\dd t_i}{\eta + t_i} \int_{0}^{\eta_0} \hspace{-1mm}\frac{\dd \tilde{t}_i}{\eta + \tilde{t}_i} \left(\frac{1}{N} \sum_{x} \big( \im G_{xx}^{(t_i)}   \im G_{xx}^{(\tilde{t}_i)}\big)\right)   \right] 
\notag \\
\lesssim \; &\frac{1}{Np}  \left( 1 + \E \left[R_N^{k+1}\right] \right) + \frac{(\log N)^2}{N^2 p} \E \left[  \mathbf{1}_{\Omega_N}\big(1 + |\mathcal{G}_k| \big) \prod_{i=1}^2\int_{0}^{\eta_0} \hspace{-1mm}\frac{\dd t_i \, \langle \im G^{(t_i)} \rangle^{1/2}}{(\eta + t_i)^{3/2}}  \, \int_{0}^{\eta_0} \hspace{-1mm}\frac{\dd \tilde{t}_i \, \langle \im G^{(\tilde{t}_i)} \rangle^{1/2}}{(\eta + \tilde{t}_i)^{3/2}}   \right] 
\notag \\ \label{eq:imimimim}
\lesssim \; &\frac{(\log N)^2 \, (\log \log N)^4 }{N p} \left( 1 + \E \left[R_N^{k+1}\right] \right) \,. 
\end{align}
 Note that now there are in total four integrals producing the $\log \log N$-divergence. 
 Collecting the above estimates, we have thus shown that 
 \begin{equation}
\eqref{eq:caseia} \lesssim (\log N)^{-\kappa/10}  \left( 1 + \E \left[R_N^{k+1}\right] \right) \,. 
 \end{equation}
 
  In Case (i.b), we have the estimate  
\begin{align}
 		&\frac{1}{N^2 (Np)^{1/2}} \sum_{x,y} \E \left[ \mathbf{1}_{\Omega_N} \sup_{|\nu| \leq \frac{10}{ \sqrt{Np}}} \left| G_{xy}^{(xy, \nu)} G_{xy}^{(xy, \nu)} G_{xy}^{(xy, \nu)} \prod_{i=2}^k \langle G^{(xy,\nu)}(z_i) \rangle \right|  \right] \notag \\
 		\lesssim  \; &\frac{1}{N^2 (Np)^{1/2}} \sum_{x,y} \Big(\E \left[ \mathbf{1}_{\Omega_N}  | G_{xy} G_{xy} G_{xy}| \big(1 + |\mathcal{G}_k| \big)  \right]
\notag \\ \label{eq:caseib}
		&+ \frac{1}{(Np)^{3/2}} \E \Big[ \mathbf{1}_{\Omega_N}  | (G_{xx})^3 (G_{yy})^3 | \big( 1 + | \mathcal{G}_k|\big)  \Big]\Big)
\end{align}
 by application of Lemma \ref{lem:stability}, in particular \eqref{eq:offdiag}--\eqref{eq:trace}. 
 
 For the first term on the right-hand side of \eqref{eq:caseib} we bound one factor of $G_{xy}$ by $\log N$ with the aid of Lemma \ref{lem:monotone}. The other two factors are handled by a Schwarz inequality together with a Ward identity, yielding that this contribution is bounded by (a constant times)
 \begin{equation}
\frac{\log N}{\sqrt{Np}} \left( 1 + \E \left[R_N^{k+1}\right] \right) \,. 
 \end{equation}
 The second term on the right-hand side of \eqref{eq:caseib} can be handled similarly to \eqref{eq:imim} and \eqref{eq:imimimim}. First, we estimate one factor of $G_{xx}$ and $G_{yy}$ trivially by Lemma \ref{lem:monotone}, yielding a $(\log N)^2$-factor. For the remaining two factors of $G_{xx}$ and $G_{yy}$ we employ \eqref{eq:Gxxest} and find this contribution to admit the bound
\begin{equation}
\frac{(\log N)^4 (\log \log N)^4}{(Np)^2} \left( 1 + \E \left[R_N^{k+1}\right] \right) \,, 
\end{equation}
exactly as in \eqref{eq:imimimim}. To summarize, collecting the estimates above, we have shown that 
\begin{equation}
\eqref{eq:caseib} \lesssim (\log N)^{-\kappa/10}  \left( 1 + \E \left[R_N^{k+1}\right] \right) \,. 
\end{equation}

For the remaining cases, we argue as follows. Case (ii.a) can be reduced to the treatment of Case (i.b) since $N^{-1} (G^2)_{xy} \lesssim \log N$ by a Schwarz inequality together with a Ward identity and Lemma \ref{lem:monotone}. Case (ii.b) can be handled similarly to Case (i.a) estimating $|N^{-1} (G^2)_{xy}| \lesssim \big(\im G_{xx}\big)^{1/2} \big(\im G_{yy}\big)^{1/2}$. The same applies to Case (iii.a). Case (iii.b) is similar to Case (i.b), while the last remaining Case (iii.c) is again treatable similarly to Case (i.a). The details are left to the reader. Therefore, we have proven \eqref{eq:thirdorder}. 

The proofs for \eqref{eq:eecontrol} and \eqref{eq:diagonal} are straightforward, noting that, in case of \eqref{eq:eecontrol} the prefactor $p \theta$ is bounded by $1/\sqrt{N}$, while in case of \eqref{eq:diagonal} there is only a single summation. Both make the simple monotonicity bound from Lemma \ref{lem:monotone} affordable. This concludes the proof of Lemma \ref{lem:cumexperror}. 
\end{proof}

\appendix

\section{Additional technical results and proofs} \label{app:technical}
In this appendix, we collect some additional technical results and proofs of statements in the~main~text. 
\subsection{Concentration for sums Bernoulli random variables: Bennett's inequality}
In Section \ref{subsec:generateproof}, we frequently use Bennett's inequality, stated below (see, e.g., \cite[Theorem 2.9]{BLM13}).
\begin{lemma}[Bennett's inequality] \label{lem:Bennett}
	Let $N \in \N$ and $p \in (0,1)$. If $X_1, \dots, X_n$ are i.i.d.\ $\op{Bernoulli}(p)$ random variables, then
	\begin{equation}
		\P \left( \frac{1}{Np(1-p)} \left| \sum_{n=1}^N \big(X_n - p\big)\right| \geq t \right) \leq 2 \ee^{- Np (1 - p) h(t)}\,,
	\end{equation}
	where $h(u) \deq (1 + u) \log(1 + u) - u$. 
\end{lemma}
In the main text, we frequently apply Lemma \ref{lem:Bennett} in the following ways: First, for any $D > 0$, there exists a $C > 0$ such that, for all $ p \leq 1/2 $ and denoting $\nu \deq \sqrt{\log N /(Np)}$, we have
\begin{equation} \label{eq:Bennettpractice}
	\begin{split}
		\P \left( \frac{1}{Np} \left| \sum_{n=1}^N \big(X_n - p\big)\right| \geq C (\nu + \nu^2)\right) \leq C N^{-D} \quad &\text{and} \quad \P \left( \frac{1}{Np}  \sum_{n=1}^N X_n  \geq C (1 + \nu^2 ) \right) \leq C N^{-D} \\
		\text{i.e.} \qquad \frac{1}{Np} \left| \sum_{n=1}^N \big(X_n - p\big)\right|  \preceq \nu + \nu^2 \qquad \qquad  &\text{and} \qquad \qquad  \frac{1}{Np}  \sum_{n=1}^N X_n \preceq 1 + \nu^2 \,. 
	\end{split}
\end{equation}

Moreover, we have the following estimate (used in Section \ref{sec:isolaw}). Let $N \in \N$, $m \leq (\log N)^C$, and $p \leq N^{-1/C}$ for some constant $C > 0$. If $X_1, \dots, X_m$ are i.i.d.\ $\op{Bernoulli}(p)$ random variables, then
\begin{equation} \label{eq:IBennett}
\sum_{n =1}^m X_n \preceq 1\,,
\end{equation}
which follows by simple application of Lemma \ref{lem:Bennett} in the regime of large argument of $h$, where $h(u) \asymp u \log u$. 

\subsection{Bernoulli characteristic flow: Proofs of Lemmas \ref{lem:BerProp} and \ref{lem:Mt'}}
\begin{proof}[Proof of Lemma \ref{lem:BerProp}]
The first property (a) is obvious from the form of the ODE \eqref{eq:Bernflow}. For property (b), we first abbreviate (recall $p_t \deq 1 - \ee^{-t}$ from \eqref{eq:pt})
\begin{equation} \label{eq:abbreviate}
\mu_t \deq \theta p_t \,, \qquad \sigma_t \deq \frac{p_t(1 - p_t)}{p(1-p)} \,. 
\end{equation}
Then taking the derivative  of \eqref{eq:timeMDE} in $t$ and plugging in \eqref{eq:Bernflow} we infer 
\begin{equation} \label{eq:Mderiv}
	\begin{split}
\partial_t M_t&= - M_t \pB{(\partial_t \mu_t) \Smat + z_t - \theta \Smat + \frac{1-p_t}{p(1-p)} m_t - (\partial_t \sigma_t) m_t - \sigma_t (\partial_t m_t)} M_t \\
&= M_t + \sigma_t M_t^2 \left(\partial_t m_t - m_t\right)\,,
	\end{split}
\end{equation}
where we used the shorthand notations $M_t \equiv M_t(z_t)$ and $m_t \equiv m_t(z_t)$ and in the second step used that 
\begin{equation} \label{eq:partialmusigma}
\partial_t \mu_t = \theta - \mu_t \quad \text{and} \quad \partial_t \sigma_t = \frac{1-p_t}{p (1-p)} - 2 \sigma_t \,. 
\end{equation}
Next, taking the trace of \eqref{eq:Mderiv}, we obtain
\begin{equation}
\big(1 - \sigma_t \langle M_t^2 \rangle\big) \big(\partial_t m_t - m_t\big) = 0 \,. 
\end{equation}
Assuming that $\big(1 - \sigma_t \langle M_t^2 \rangle\big) \neq 0$, we conclude that $\partial_t m_t = m_t$ and hence $\partial_t M_t = M_t$ using \eqref{eq:Mderiv} again. This simple ODE has the solution given in Lemma \ref{lem:BerProp}~(b). 

It thus remains to prove that $\big(1 - \sigma_t \langle M_t^2 \rangle\big) \neq 0$. In order to see this, we take the imaginary part of the MDE \eqref{eq:timeMDE}, reading\footnote{Here, we additionally used that $z_t \in \mathbb{H} \mathbf{1}+ \R \Smat$.}
\begin{equation} \label{eq:ImpartMDE}
\frac{\im M_t}{M_tM_t^*} = \eta_t + \sigma_t \im m_t \,. 
\end{equation}
From this we directly infer $\big(1 - \sigma_t \langle |M_t|^2 \rangle\big) \im m_t = \eta_t \langle |M_t|^2 \rangle$ and thus 
\begin{equation} \label{eq:trivbound}
1 > \sigma_t \langle |M_t|^2 \rangle \,. 
\end{equation}
By the Cauchy-Schwarz inequality, this immediately shows $\big(1 - \sigma_t \langle M_t^2 \rangle\big) \neq 0$. 

Property (c) follows by taking the imaginary part of \eqref{eq:Bernflow}, realizing that it reduces to a scalar ODE, and plugging in that $\im m_t(z_t) = \ee^t \im m_0(z_0)$ (as a consequence of part (b)). 

Finally, for part (d), we again take the imaginary part of \eqref{eq:Bernflow}. Then, dividing by $\eta_r^\alpha$ we obtain, dropping the negative term $-\eta_r^{1 - \alpha}$, and applying a change of variables, 
\begin{equation}
\int_{s}^{t} \frac{1 - p_t}{1 - p} \frac{\im m_r(z_r)}{p \eta_r^{\alpha}} \dd r \leq - \int_{s}^{t} \frac{\partial_r \eta_r}{\eta_r^{\alpha}} \dd r = \int_{\eta_t}^{\eta_s} \frac{1}{\eta^\alpha} \dd \eta  = \begin{cases}
\log\left(\frac{\eta_s}{\eta_t}\right) &\text{for} \quad \alpha = 1 \\
\frac{1}{\alpha - 1} \left(\frac{1}{\eta_t^{\alpha - 1}} - \frac{1}{\eta_s^{\alpha - 1}}\right) &\text{for} \quad \alpha > 1 \,. 
\end{cases}
\end{equation}
This concludes the proof of Lemma \ref{lem:BerProp}. 
\end{proof}

\begin{proof}[Proof of Lemma \ref{lem:Mt'}]
We adopt the abbreviations \eqref{eq:abbreviate} from the proof of Lemma \ref{lem:BerProp}. Then, first, differentiating the MDE \eqref{eq:timeMDE} in the spectral parameter, we find
\begin{equation} \label{eq:M'}
 M_t'  = \big(1 + \sigma_t m_t'\big) M_t^2 \,. 
\end{equation}
Differentiating in time and employing \eqref{eq:partialmusigma} together with $\partial_t M_t = M_t$ yields
\begin{equation} \label{eq:partialMt'}
	\begin{split}
\partial_t M_t' & = \big( 2 + (\partial_t \sigma_t + 2 \sigma_t ) m_t' + \sigma_t (\partial_t m_t')\big) M_t^2 + 2\big(1 + \sigma_t m_t'\big) M_t (\partial_t M_t) \\
& = \left( 2 + \frac{1 - p_t}{p(1-p)} m_t' + \sigma_t (\partial_t m_t') \right)M_t^2 \,. 
	\end{split}
\end{equation}
As in the proof of Lemma \ref{lem:BerProp}, we take the trace to get
\begin{equation}
\partial_t m_t' =  \left(2 + \frac{1 - p_t}{p(1-p)} m_t' \right) \frac{\langle M_t^2 \rangle}{1 - \sigma_t \langle M_t^2 \rangle} = \left(2 + \frac{1 - p_t}{p(1-p)} m_t' \right) m_t'\,,
\end{equation}
where in the second step we used \eqref{eq:M'}. Plugging this into \eqref{eq:partialMt'} and involving \eqref{eq:M'} again, we deduce the claim. 
\end{proof}

\subsection{$M$-computations: Proofs of Lemmas \ref{lem:Mprop} and \ref{lem:M2bound}}
\begin{proof}[Proof of Lemma \ref{lem:Mprop}]
Using the notation $\theta = (Np(1-p))^{-1/2}$, we can rewrite the MDE \eqref{eq:MDE} as
\begin{equation} \label{M_MDE}
- \frac{1}{M(z)} = (z + p \theta ) - N p \theta \Pi_{\bm e} + \langle M(z) \rangle \,,
\end{equation}
which is a rank-one perturbation of the semicircular MDE $- m_{\rm sc}(z)^{-1} = z + m_{\rm sc}(z) $ at the shifted spectral parameter $z + p \theta$. Inversion of \eqref{M_MDE} yields
\begin{equation} \label{eq:Mspecdec}
M(z) = \frac{\mathbf{1} - \Pi_{\bm e}}{- (z + p \theta) - \langle M(z) \rangle} + \frac{\Pi_{\bm e}}{Np \theta - (z + p \theta) - \langle M(z) \rangle}\,,
\end{equation}
and taking the trace yields
\begin{equation*}
\ang{M(z)} = \frac{-1}{z + p \theta + \ang{M(z)}} + \frac{1}{N}  \pbb{\frac{1}{z + p \theta + \ang{M(z)}} + \frac{1}{Np \theta - (z + p \theta) - \langle M(z) \rangle}}\,.
\end{equation*}
This equation is a $1/N$-perturbation of the scalar equation $- m_{\rm sc}(z)^{-1} = z + m_{\rm sc}(z)$ at the shifted spectral parameter $z + p \theta$, which characterizes $m_{\rm sc}$. By a standard stability analysis (see e.g.\ \cite[Lemma 5.5]{BenyachKnowles2017}) using a continuity argument in $\eta$, and using that $Np \theta \asymp \sqrt{N p}$ tends to infinity with $N$, we deduce that, for $|z| \lesssim 1$,
\begin{equation} \label{eq:Mmsc}
\big| \langle M(z) \rangle - m_{\rm sc}(z + p \theta)\big| \lesssim \frac{1}{N} \,. 
\end{equation}
Now \eqref{eq:Mprop} easily follow from \eqref{eq:Mspecdec} and \eqref{eq:Mmsc}, using that $|m_{\rm sc}(z)| \asymp 1$ for $|z| \lesssim 1$. 

For \eqref{eq:ImMprop}, we take the imaginary part of \eqref{eq:Mspecdec}, yielding that
\begin{equation}
\big(\im M(z)\big)_{\bm e \bm e} = \frac{\im z + \im \langle M(z) \rangle}{|Np \theta - (z + p \theta) - \langle M(z) \rangle|^2} \asymp \frac{1}{Np} \quad \text{for} \quad |z| \lesssim 1 \,. 
\end{equation}
This concludes the proof of Lemma \ref{lem:Mprop}. 
\end{proof}

\begin{proof}[Proof of Lemma \ref{lem:M2bound}]
By differentiating the MDE \eqref{eq:timeMDE} w.r.t.~the spectral parameter $z_t$, we obtain (recall the notation \eqref{eq:abbreviate})
\begin{equation}
M_t' = \frac{M_t^2}{1 - \sigma_t \langle M_t^2 \rangle} \,. 
\end{equation}
Hence, using Lemma \ref{lem:BerProp}~(b) and \eqref{eq:Mspecdec}, we obtain
\begin{equation}
\big|(M_t')_{\bm v \bm w}\big| \lesssim \frac{1}{|1 - \sigma_t \langle M_t^2 \rangle|} \left(|\langle \bm v, (\mathbf{1} - \Pi_{\bm e}) \bm w \rangle| + \frac{|\langle \bm v, \Pi_{\bm e} \bm w \rangle|}{Np}\right) \quad \text{for any} \quad |z| \lesssim 1 \,. 
\end{equation} 
It thus remains to show that $|1 - \sigma_t \langle M_t^2 \rangle| \geq 1/C_\kappa$ for $z_T \in \mathbb{D}$. Using the explicit form of $\sigma_t$ and Lemma \ref{lem:BerProp}~(b), we have 
\begin{equation}
1 - \frac{\ee^t - 1}{\ee^T - 1} \langle M_T^2 \rangle = 1 - \frac{\ee^t - 1}{\ee^T - 1} \langle |M_T|^2 \rangle - 2 \ii \frac{\ee^t - 1}{\ee^T - 1} \langle M_T \im M_T \rangle \,. 
\end{equation}
The real part is now lower bounded by
\begin{equation}
	\begin{split}
1 - \langle|M_T|^2 \rangle &+ \frac{T - t}{\ee^T - 1}\langle|M_T|^2 \rangle + \frac{2 t}{\ee^T - 1} \langle (\im M_T)^2 \rangle \\
&\geq \frac{1}{2} \left( \left(1 - \frac{t}{T}\right)\frac{\im m_T}{\im z_T + \im m_T} + \frac{t}{T} (\im m_T)^2 \right) \geq 1/C_\kappa
	\end{split}
\end{equation}
for some $C_\kappa > 0$ depending only on $\kappa$. Here, in the first step, we used \eqref{eq:trivbound} and \eqref{eq:ImpartMDE} for $t = T$, and the inequality $\langle R \rangle \leq \langle R^2 \rangle$ for any $R \geq 0$. In the last step we then used \eqref{eq:Mmsc} and elementary properties of the Stieltjes transform of the semicircular density. This concludes the proof. 
\end{proof}

\subsection{Resolvent expansion: Proof of Lemma \ref{lem:entforG}}
We only prove the statement concerning a single resolvent in Lemma~\ref{lem:entforG}~(i); all the other statements in Lemma \ref{lem:entforG} can be shown similarly and we hence omit their proofs for brevity. Dropping the time and argument, we have
\begin{equation}
\big|\big(G-M\big)_{\bm a \bm b}\big| \leq \big|\big(G^{ab} - M\big)_{\bm a \bm b}\big| + \big|\big(G^{ab} - G\big)_{\bm a \bm b}\big| \leq \Cf \Lambda_{{\rm ent}}^{(1)} + \big|\big(G^{ab} - G\big)_{\bm a \bm b}\big| \,. 
\end{equation}
To control $\big|\big(G^{ab} - G\big)_{\bm a \bm b}\big|$, we employ a resolvent expansion, similarly to Lemma \ref{lem:resolventexpand} and find that 
\begin{equation}
\big|\big(G^{ab} - G\big)_{\bm a \bm b}\big| \lesssim \frac{1}{\sqrt{Np}} \left[ |G^{ab}_{\bm a a} G_{b \bm b}| + |G^{ab}_{\bm a b} G_{a \bm b}| \right] \,,
\end{equation}
from which by iteration, we easily obtain $\big|\big(G^{ab} - G\big)_{\bm a \bm b}\big| \lesssim (Np)^{-1/2}$. We conclude the desired. \qed

\section{Concentration for Markov jump processes} \label{app:CDC}

In this appendix, for the reader's convenience, we give a self-contained review of some classical facts about Markov jump processes. See e.g.\ \cite[Appendix B]{shorack2009empirical} for a comprehensive account.

Let $\Gamma$ be a finite set. Let $(X_t)_{t \geq 0}$ be a right-continuous Markov jump process on $\Gamma$ with generator
\begin{equation*}
\cal L \phi(x) = \sum_{y} \cal L(x,y) (\phi(y) - \phi(x))\,.
\end{equation*}
That is, for any $\phi : \Gamma \to \C$ we have
\begin{equation} \label{generator}
\frac{\dd}{\dd t} \E[\phi(X_t)] = \E[\cal L \phi(X_t)]\,.
\end{equation}
Throughout this appendix, we use the canonical filtration $(\cal F_t)_{t \geq 0}$ generated by the process $(X_t)_{t \geq 0}$. Moreover, to streamline the presentation, we suppose that the observable $\phi$ does not explicitly depend on time, although this is unimportant for the arguments; see Remark \ref{rem:time_dep_phi} below.
\begin{lemma}\label{lem:dynkin_mart}
For any $\phi : \Gamma \to \C$, the process
\begin{equation} \label{dynkin}
\cal M_t \deq \phi(X_t) - \phi(X_0) - \int_0^t \dd s \, \cal L \phi(X_s)
\end{equation}
is a càdlàg martingale.
\end{lemma}
\begin{proof}
By the Markov property of $(X_t)_{t \geq 0}$ we have, for $0 \leq s < t$,
\begin{align*}
\E[\cal M_t \,|\, \cal F_s] &= \cal M_s + \E \qbb{\phi(X_t) - \phi(X_s) - \int_s^t \dd r \, \cal L \phi(X_r) \,\bigg|\, \cal F_s}
\\
&= \cal M_s + \E \qbb{\phi(\tilde X_{t-s}) - \phi(\tilde X_0) - \int_0^{t-s} \dd r \, \cal L \phi(\tilde X_r) \,\bigg|\, \tilde X_0 = X_s} = 0\,,
\end{align*}
where $(\tilde X_t)_{t \geq 0}$ denotes an independent copy of $(X_t)_{t \geq 0}$, and in the last step we used \eqref{generator}.
\end{proof}

As a consequence, $\abs{\cal M_t}^2$ is a submartingale, and by the Doob-Meyer decomposition there is a unique increasing predictable process starting at zero, $\ang{\cal M}_t$, such that $\abs{\cal M_t}^2 - \ang{\cal M}_t$ is a martingale. The process $\ang{\cal M}_t$ is the \emph{predictable quadratic variation} of $\cal M_t$.

\begin{lemma} \label{lem:quadr_var}
The predictable quadratic variation of $\cal M_t$ is
\begin{equation} \label{M_quadr_var}
\ang{\cal M}_t = \int_0^t \dd s \, \sum_y \cal L(X_s, y) \, \abs{\phi(y) - \phi(X_s)}^2\,.
\end{equation}
\end{lemma}
\begin{proof}
Without loss of generality, we suppose that $\phi(X_0) = 0$. The claimed expression for $\ang{\cal M}_t$ can be obtained by computing, using \eqref{generator},
\begin{align*}
\frac{\dd}{\dd t} \E[\abs{\cal M_t}^2] &= \frac{\dd}{\dd t} \E \qbb{\abs{\phi(X_t)}^2 + \absbb{\int_0^t \dd s \, \cal L \phi(X_s)}^2 - 2 \re \phi(X_t) \int_0^t \dd s \, \cal L \phi(X_s)}
\\
&= \E \qb{\cal L \abs{\phi}^2(X_t) - 2 \re \phi(X_t) \cal L \phi(X_t)}\,.
\end{align*}
Defining 
\begin{equation*}
A_t \deq \int_0^t \dd s \, \pb{\cal L \abs{\phi}^2(X_s) - 2 \re \phi(X_s) \cal L \phi(X_s)}\,,
\end{equation*}
we have therefore shown that $\frac{\dd}{\dd t} \E[\abs{\cal M_t}^2] = \frac{\dd}{\dd t} \E[A_t]$.
The same calculation combined with the Markov property of $(X_t)_{t \geq 0}$, like in the proof of the previous lemma, shows that $\abs{\cal M_t}^2 - A_t$ is a martingale.

Next, we write
\begin{align*}
A_t &= \int_0^t \dd s \, \sum_y L(X_s, y) \qB{\abs{\phi(y)}^2 - \abs{\phi(X_s)}^2 - 2 \re \phi(X_s) (\phi(y) - \phi(X_s))}
\\
&= \int_0^t \dd s \, \sum_y L(X_s, y) \, \abs{\phi(y) - \phi(X_s)}^2\,.
\end{align*}
In particular, $A_t$ is increasing and predictable (even continuous), and hence must be equal to $\ang{\cal M}_t$.
\end{proof}

The following result is a well known instance (see e.g.\ \cite{vandegeer1995} or \cite[Eq.~(18) in Appendix B.6]{shorack2009empirical}) of the Burkholder-Davis-Gundy inequality. 

\begin{proposition}[Martingale concentration inequality] \label{prop:DBG}
Suppose that the jumps of the martingale \eqref{dynkin} satisfy $|\Mart_t - \Mart_{t-}| \leq K$ almost surely for some constant $K > 0$. Then for any $a,b >0$ we have
\begin{equation} \label{BDG}
\proba{\exists t \geq 0 :  |\Mart_t| \geq a, \langle \Mart \rangle_t \leq b^2} \leq 4 \exp{\left[-\frac{a^2}{2(aK + b^2)}\right]}\,.
\end{equation}
\end{proposition}

\begin{proof}
Suppose first that $\phi$ is real.
Fix $\lambda > 0$.
Applying Lemma \ref{lem:dynkin_mart} to $\psi \deq \ee^{\lambda \phi}$, we find that
\begin{equation*}
\cal N_t \deq \psi(X_t) - \psi(X_0) - \int_0^t \dd s \, \cal L \psi(X_s) =
\psi(X_t) - \psi(X_0) - \int_0^t \dd s \, \kappa(X_s) \psi(X_s)
\end{equation*}
is a martingale, where we defined
\begin{equation*}
\kappa(x) \deq \sum_y \cal L(x,y) \pb{\ee^{\lambda(\phi(y) - \phi(x))} - 1}\,.
\end{equation*}
We can solve the equation $\dd \psi(X_t) = \kappa(X_t) \psi(X_t) + \dd \cal N_t$ by variation of constants, which implies that
\begin{equation*}
\cal V_t \deq \ee^{-\int_0^t \dd s \, \kappa(X_s)} \psi(X_t)
\end{equation*}
is a martingale, since
\begin{equation*}
\dd \cal V_t = \ee^{-\int_0^t \dd s \, \kappa(X_s)} \dd \cal N_t\,.
\end{equation*}
Defining
\begin{equation*}
\rho(x) \deq \sum_y \cal L(x,y) \pb{\ee^{\lambda (\phi(y) - \phi(x))} - 1 - \lambda(\phi(y) - \phi(x))}\,,
\end{equation*}
we conclude that
\begin{equation*}
\cal Z_t \deq \cal V_t / \psi(X_0) = \ee^{\lambda \cal M_t - \int_0^t \dd s \, \rho(X_s)}
\end{equation*}
is a martingale.
By power series expansion, we deduce that
\begin{equation*}
\frac{\ee^u - 1 - u}{u^2} \leq \frac{\ee^v - 1 - v}{v^2}
\end{equation*}
whenever $\abs{u} \leq v$. By assumption on the jump sizes, $\abs{\phi(y) - \phi(x)} \leq K$ whenever $\cal L(x,y) \neq 0$. We conclude that
\begin{equation*}
\rho(x) \leq h(\lambda) \sum_y \cal L(x,y) (\phi(y) - \phi(x))^2 \,, \qquad h(\lambda) \deq \frac{\ee^{\lambda K} -1 -\lambda K}{K^2}\,.
\end{equation*}
By Lemma \ref{lem:quadr_var}, we deduce that
\begin{equation} \label{Z_t_rho}
\int_0^t \dd s \, \rho(X_s) \leq h(\lambda) \ang{\cal M}_t\,.
\end{equation}

Next, we define the stopping time (recall that $\cal M_t$ is right-continuous and $\ang{\cal M}_t$ is continuous)
\begin{equation*}
\tau \deq \inf \h{t \geq 0 : \cal M_t \geq a,  \ang{\cal M}_t \leq b^2}\,.
\end{equation*}
Using \eqref{Z_t_rho}, we estimate 
\begin{equation} \label{tau_infty_estimate}
\P(\tau < \infty) \leq \E \qB{\ind{\tau < \infty} \, \ee^{\lambda (\cal M_\tau - a)} \ee^{-h(\lambda) (\ang{\cal M}_\tau - b^2)}} \leq \ee^{- \lambda a + h(\lambda) b^2} \E[\ind{\tau \leq \infty} \cal Z_\tau] \leq \ee^{- \lambda a + h(\lambda) b^2}\,,
\end{equation}
where in the last step we used Fatou's lemma and optional stopping for the martingale $\cal Z_t$ to conclude that
\begin{equation*}
\E[\ind{\tau \leq \infty} \cal Z_\tau] = \E \qB{\lim_{n \to \infty} \ind{\tau \leq n} \cal Z_{\tau \wedge n}} \leq \lim_{n \to \infty} \E[\ind{\tau \leq n} \cal Z_{\tau \wedge n}] \leq \lim_{n \to \infty} \E[\cal Z_{\tau \wedge n}] = 1\,.
\end{equation*}
To conclude the proof, by convexity of $\exp$ we obtain for $\lambda K < 1$ that $h(\lambda) \leq \frac{\lambda^2}{2 (1 - \lambda K)}$.
Using \eqref{tau_infty_estimate} with the choice $\lambda \deq \frac{a}{b^2 + Ka}$, and repeating the same argument for $-\cal M_t$, we conclude from a union bound that
\begin{equation*}
\proba{\exists t \geq 0 :  |\Mart_t| \geq a, \langle \Mart \rangle_t \leq b^2} \leq 2 \exp{\left[-\frac{a^2}{2(aK + b^2)}\right]}\,.
\end{equation*}

Finally, the claim for complex $\phi$ follows by applying the above estimate to the real and imaginary parts of $\phi = \phi^{(1)} + \ii \phi^{(2)}$ separately, using that the resulting splitting $\cal M_t = \cal M_t^{(1)} + \ii \cal M_t^{(2)}$ satisfies $\ang{\cal M}_t = \ang{\cal M^{(1)}}_t + \ang{\cal M^{(2)}}_t$.
\end{proof}

\begin{remark}
The estimate \eqref{BDG} is in general sharp up to constants. In the limit where $K \to 0$, it corresponds to subgaussian tails for continuous martingales. For martingales with jumps, $K > 0$, there is a transition from subgaussian to subexponential tails for large enough $a$, depending on $K$. The same phenomenon is manifested in Bennett's concentration inequality (see Lemma \ref{lem:Bennett}).
\end{remark}

\begin{remark} \label{rem:time_dep_phi}
The above was written up for case where $\phi$ does not explicitly depend on time. However, by a straightforward extension, the results of this appendix extend to time-dependent observables $\phi_s$. In that case, \eqref{dynkin} and \eqref{M_quadr_var} read
\begin{equation*}
\cal M_t \deq \phi(X_t) - \phi(X_0) - \int_0^t \dd s \, \pb{ \cal L \phi_s(X_s) + \partial_s \phi_s(X_s)}
\end{equation*}
and
\begin{equation*}
\ang{\cal M}_t = \int_0^t \dd s \, \sum_y \cal L(X_s, y) \, \abs{\phi_s(y) - \phi_s(X_s)}^2\,.
\end{equation*}
Then Proposition \ref{prop:DBG} remains valid with the same proof.
\end{remark}

\section{A simple proof of average local law for GOE / GUE with factor $\sqrt{\log N}$} \label{sec:GOE_ll}
For the Gaussian Orthogonal/Unitary Ensembles (GOE/GUE), the optimal (in terms of $N$-dependence) average local law is known to be as follows: For $z \in \C\setminus \R$, let $G(z) \deq (W - z)^{-1}$ be the resolvent of a GOE/GUE matrix $W$. Then, for any $D > 0$ there exists a $C > 0$ such that \cite[Theorem~1.1]{bourgade2022optimal}
\begin{equation} \label{eq:GUEoptimal}
\proba{\big|\langle G(z) - m_{\rm sc}(z) \rangle\big| \leq C \frac{(\log N)^{1/2}}{N \eta}} \geq 1 - N^{-D} 
\end{equation}
for spectral parameters $z =  e+ \ii \eta$ satisfying $e \in [-2+\eta, 2 - \eta]$ and $\eta \leq 1$. The proof of this result in \cite{bourgade2022optimal} is based on a delicate analysis of loop equations, that is designed to handle the more general class of $\beta$-ensembles as well. The key point here is the optimal power of $\log N$ in the error term of the local law; see \cite[Proposition 3.2 and the following Remark]{cipolloni2025maximum} for a related result, in particular discussing the possibility of proving \eqref{eq:GUEoptimal} for non-Hermitian matrices with generally distributed i.i.d.~entries with light tails on the analog of the domain \eqref{eq:largedomain}. 

The goal of this appendix is to give a short proof of \eqref{eq:GUEoptimal} for spectral parameters (notice that $\log N$ is replaced by $(\log N)^{1/2}$ compared to \eqref{eq:domain})
  \begin{equation}  \label{eq:largedomain}
z \in 	\widetilde{\mathbb{D}}(\kappa, C) = [-2 + \kappa, 2 - \kappa] \times \ii [C N^{-1} (\log N)^{1/2}, 1]  \subset \HH \,, \quad \text{for} \quad  \kappa, C > 0\,. 
\end{equation}
More precisely, we prove the following. 
\begin{proposition}[$\sqrt{\log N}$ in the average law for GOE/GUE] \label{prop:Gauss}
Fix $\kappa > 0$. Then, for any $D > 0$, there exists a constant $C = C(D, \kappa) >0 $ such that
\begin{equation}
\proba{\bigcap_{z \in \widetilde{\mathbb{D}} (\kappa, C)}\left\{\big|\langle G(z) - m_{\rm sc}(z) \rangle\big| \leq C \frac{(\log N)^{1/2}}{N \eta}\right\}} \geq 1  -  N^{-D} \,. 
\end{equation}
\end{proposition}
\begin{proof}
The overall structure of the argument is similar to those applied for the proof of Theorem~\ref{thm:lolaw}: That is, we consider a dynamics and argue that a suitably designed stopping time (see \eqref{eq:stoptime2} below) actually equals the maximal time of the evolution. Instead of the Bernoulli flow, we now consider the usual Brownian characteristic flow, governed by the equations
	\begin{equation} \label{eq:flow}
		\dd W_t = \frac{\dd B_t}{\sqrt{N}} \,, \qquad \partial_t z_t = - m_t(z_t)\,, 
	\end{equation}
	where we fix the time interval to be $t \in [0,1]$. Moreover, $B_t$ is a real symmetric or complex Hermitian matrix valued standard Brownian motion (depending on the symmetry class of $W$). We start with initial condition $W_0 = 0$ and end with terminal condition $z_{T=1} \in \mathbb{D}$. Finally, $m_t(z_t)$ solves the equation
	\begin{equation}
- \frac{1}{m_t(z_t)} = z_t + t m_t(z_t) \,,
	\end{equation}
	from which we infer that $|m_t(z_t)| \asymp 1$ and $|m_t'(z_t)| \leq C_\kappa$ where $m_t'(z_t) = \tfrac{\dd }{\dd w} m_t(w) \vert_{w = z_t}$
	
For fixed $z_T \in \mathbb{D}$, we define two observables 
\begin{equation}
\phi_t^{(1)} \deq \langle G_t(z_t) - m_t(z_t) \rangle \,, \qquad \phi_t^{(2)} \deq \langle G^2_t(z_t) - m_t'(z_t) \rangle
\end{equation}
and note that $\phi_0^{(i)} = 0$ for $i \in [2]$. Now, abbreviating $\eta_t \deq \im z_t$, for a constant $\Cf > 0$, we define the stopping time
\begin{equation} \label{eq:stoptime2}
	\tau = \tau (\Cf) \deq \inf \left\{ t \in [0,1] : N \eta_t \big| \phi_t^{(1)} \big| \geq \Cf (\log N)^{1/2}\quad \text{or} \quad  N \eta_t^2 \big| \phi_t^{(2)} \big|\geq \Cf^2 (\log N)^{1/2} \right\} 
\end{equation}
and suppose that $N \eta_T \geq (1 + \Cf)^4 (\log N)^{1/2}$. By continuity and since the condition is trivial (identically zero) at time zero, we have that $\tau > 0$ with very high probability and shall henceforth assume that all times are in the interval $t \in [0,\tau]$. 

Next, dropping the $z$-arguments for brevity, we compute the differential of $\phi_t^{(i)}$ using Ito's formula together with \eqref{eq:flow} and obtain
\begin{equation} \label{eq:flows}
	\begin{split}
\dd \phi_t^{(1)} &= m_t' \phi_t^{(1)} \dd t+ \phi_t^{(1)} \phi_t^{(2)} \dd t + \frac{\mathbf{1}_{\beta = 1}}{N} \mathcal{O}\left( \frac{\langle \im G_t \rangle}{\eta_t^2} \right) \dd t + \dd \Mart_t^{(1)}\,,  \\
\dd \phi_t^{(2)} &= 2m_t' \phi_t^{(2)} \dd t+ \big( \phi_t^{(2)}\big)^2 \dd t + 2 \phi_t^{(1)} \langle G_t^3\rangle + \frac{\mathbf{1}_{\beta = 1}}{N} \mathcal{O}\left( \frac{\langle \im G_t \rangle}{\eta_t^3} \right) \dd t +   \dd \Mart_t^{(2)}\,,
	\end{split}
\end{equation}
where we abbreviated 
\begin{equation}
\dd \Mart_t^{(1)} \deq - \frac{1}{\sqrt{N}} \langle G_t^2 \dd B_t \rangle\,, \qquad \dd \Mart_t^{(2)} \deq - \frac{1}{\sqrt{N}} \langle G_t^3 \dd B_t \rangle\,,
\end{equation}
and the shorthand notation $\mathbf{1}_{\beta = 1}$ indicates that these terms are only present in the GOE case. The implicit constants in the $\mathcal{O}$-notation are bounded by, say, $10$, independent of any particular constant from above. 

Completely analogously to our proofs in Sections \ref{sec:Gronproof}--\ref{sec:lemproofs}, we now control the quadratic variation of the martingale terms in \eqref{eq:flows} and the size of the generator terms. The quadratic variations admit the bounds
\begin{equation} \label{eq:martest}
	\begin{split}
	[\mathcal{M}_t^{(1)}]^2 &\leq \int_{0}^{t} \dd s \frac{\langle |G_s|^4 \rangle}{N^2} \leq \int_{0}^{t} \dd s \frac{1}{N^2 \eta_s^3} \left(1 + \frac{\Cf (\log N)^{1/2}}{N \eta_s}\right)\lesssim \left(\frac{1}{N \eta_t}\right)^2 \,, \\
	[\mathcal{M}_t^{(2)}]^2 &\leq \int_{0}^{t} \dd s \frac{\langle |G_s|^6 \rangle}{N^2} \leq \int_{0}^{t} \dd s \frac{1}{N^2 \eta_s^5} \left(1 + \frac{\Cf (\log N)^{1/2}}{N \eta_s}\right)\lesssim \left(\frac{1}{N \eta_t^2}\right)^2\,. 
	\end{split}
\end{equation}
The remaining terms in \eqref{eq:flows} can be bounded as
\begin{equation} \label{eq:gen1}
	\begin{split}
& \int_{0}^{t} \dd s \, \left|m_s' \phi_s^{(1)} + \phi_s^{(1)} \phi_s^{(2)}  + \frac{\mathbf{1}_{\beta = 1}}{N} \mathcal{O}\left( \frac{\langle \im G_s \rangle}{\eta_s^2} \right)\right| \\
& \qquad \qquad \leq C_\kappa \left(\int_{0}^{t} \dd s \left(1 + \Cf^2 \frac{(\log N)^{1/2}}{N \eta_s^2}\right)\big|\phi_s^{(1)}\big| + \frac{(\log N)^{1/2}}{N \eta_t}\right)
	\end{split}
\end{equation}
and
\begin{equation} \label{eq:gen2}
	\begin{split}
		& \int_{0}^{t} \dd s \, \left|2m_s' \phi_s^{(2)} + \big( \phi_s^{(2)}\big)^2  + 2 \phi_s^{(1)} \langle G_s^3\rangle + \frac{\mathbf{1}_{\beta = 1}}{N} \mathcal{O}\left( \frac{\langle \im G_s \rangle}{\eta_s^3} \right) \dd t \right| \\
		& \hspace{6cm} \leq C_\kappa \left(\int_{0}^{t} \dd s \left(1 + \Cf^2 \frac{(\log N)^{1/2}}{N \eta_s^2}\right)\big|\phi_s^{(2)}\big| + \frac{(\log N)^{1/2}}{N \eta_t^2}\right)
	\end{split}
\end{equation}
where we used the integration rules
\begin{equation}
\int_{0}^{t} \dd s\frac{1}{\eta_s} \lesssim \log\left(\frac{\eta_0}{\eta_t}\right) \quad \text{and} \quad \int_{0}^{t} \dd s \frac{1}{\eta_s^\alpha} \lesssim \frac{1}{\eta_t^{\alpha - 1}} \quad \text{for} \quad \alpha > 1 
\end{equation}
and the definition of the stopping time \eqref{eq:stoptime2}. 

We can then obtain a similar Grönwall estimate as in Proposition \ref{prop:Gron} (see in particular \eqref{eq:avGron1}--\eqref{eq:avGron2}) by involving the Burholder-Davis-Gundy inequality for continuous martingales (corresponding to Proposition \ref{prop:DBG} with $K=0$); see e.g.\ \cite{vandegeer1995} or \cite[Eq.~(18) in Appendix B.6]{shorack2009empirical}. Deducing the validity of the local law in Proposition \ref{prop:Gauss} is then deduced by appropriately choosing the constant $\Cf$ as done in \eqref{eq:Cchoice}, eventually showing that $\tau(\Cf) = 1$ with probability at least $1 - N^{-\Cf/C_\kappa}$ for some constant $C_\kappa > 0$ depending only on $\kappa$. This concludes the proof of Proposition \ref{prop:Gauss}. 
\end{proof}

{\small
\bibliographystyle{amsplain-nodash} 
\bibliography{ref}
}

\bigskip

\noindent
Joscha Henheik, University of Oxford (\href{mailto:joscha.henheik@maths.ox.ac.uk}{joscha.henheik@maths.ox.ac.uk})
\\
Antti Knowles, University of Geneva (\href{mailto:antti.knowles@unige.ch}{antti.knowles@unige.ch})

\end{document}